\documentclass[12pt,a4paper,reqno]{amsart}

\usepackage[margin=1in,footskip=0.25in]{geometry}
\usepackage{amsmath, amsthm, amssymb}
\usepackage{graphicx}
\usepackage{color}
\usepackage[abs]{overpic}
\usepackage{quiver}
\usepackage{multirow}
\usepackage{url}

\newtheorem{theorem}{Theorem}[section]

\newtheorem{remark}[theorem]{Remark}

\newtheorem{teo}{Theorem}[section]

\newtheorem{defi}{Definition}[section]
\newtheorem{prop}{Proposition}[section]

\newcommand{\sgn}{{\operatorname{sgn}}}

\newcommand{\re}{\mathbb{R}}

\date{}

\begin{document}
	
\title[Local bifurcations via a nonlinear double-regularization process
]{%The local bifurcations in a class of piecewise smooth Filipov system with a nonregular switching curve via a nonlinear double-regularization process
Local bifurcations in a class of piecewise-smooth Filippov systems with a nonregular switching curve via a nonlinear double regularization process}

\author{Claudio A. Buzzi$^1$}
\address{$^1$ Mathematics Department, Universidade Estadual Paulista Julio de Mesquita Filho, 15054-000 São José do Rio Preto, São Paulo, Brazil }
\email{claudio.buzzi@unesp.br}

\author{Yagor Romano Carvalho$^{2,3}$}
\address{$^2$ Mathematics Department, Universidade de São Paulo, 13566-590 São Carlos, São Paulo, Brazil}
\email{yagor.carvalho@usp.br}
\address{$^3$ Departament de Matemàtiques, Universitat Autònoma de Barcelona, 08193 Bellaterra, Barcelona, Spain}
\email{yagor.romano@uab.cat}

\dedicatory{Dedicated to the memory of Yagor's grandmothers}

%\allowdisplaybreaks[4]

	%	37G15, 37D45,  34A36  34A60}

%discontinuous piecewise differential systems 
%\textcolor{red}{(arrumar)}

\begin{abstract}
%We are interested in the study of piecewise smooth vector fields, in the case where the edges of the smooth parts of the fields are not regular via a regularization process. More precisely, we are interested in an analysis of the preservation of the bifurcations of a class of piecewise smooth vector fields with a non-regular discontinuity set according to a regularization process, which approximates such a class by a continuous vector field.
%O projeto é dedicado ao estudo dos campos de vetores suaves por partes, no caso em que os bordos das partes suaves dos campos não são regulares. Dentro deste contexto, primeiramente será feito um estudo visando uma classificação dos sistemas que não são estruturalmente estáveis. Necessariamente tal estudo passa por uma análise das bifurcações de codimensão baixa: 1 e 2. Finalmente o projeto será concluído com uma análise da regularização dos sistemas estudados nas etapas anteriores. Para realizar essa última etapa será necessário desenvolver uma generalização do método de regularização de Sotomayor e Teixeira para o caso singular.
%We are interested in analyzing the preservation of bifurcations in a class of piecewise smooth vector fields with a nonregular discontinuity set under a smoothing process, which approximates this class by a smooth vector field. 
%We examine cases in which the codimension is either preserved or not, as well as whether the property of the bifurcation being generic is maintained or not.
We are interested in analyzing the preservation of bifurcations in a class of piecewise smooth vector fields with a nonregular switching set under a smoothing process that approximates them by smooth vector fields. We examine cases in which the codimension is either preserved or altered, as well as whether the generic nature of the bifurcation is maintained.

\end{abstract}
\maketitle

\section{Introduction}

%The area of Dynamical Systems has developed and today has a lot of ramifications, for instance, the study of structural stability and bifurcations. A planar differential system is structurally stable if its dynamics do not change under small perturbations, that is if their phase portraits are qualitatively similar. Among the most widespread stability theories for dynamical systems is that of Andronov
% \cite{AndroLeontGor1973} and Peixoto \cite{PeiPei1959,Pei1962}, which provided necessary and sufficient conditions for smooth two-dimensional differential systems to be structurally stable.  With the understanding of structurally stable systems, then the focus was turned to studying bifurcations. If a planar differential system is not structurally stable then it belongs to the bifurcation set. In this way, the qualitative structure of the solutions set changes suddenly when the planar differential system passes through a point  in the bifurcation set, see \cite{Kuznetsov1998,Perko2013}. Moreover, in some sense, it is possible to classify a bifurcation with respect to its level of degeneracy. In a few words, the minimum number of parameters explaining the dynamic of all differential systems in a neighborhood of a fixed element of bifurcation set is called its codimension. Consequently, as greater the codimension as greater the level of degeneration.

The area of Dynamical Systems has developed significantly and now encompasses many branches, such as the study of structural stability and bifurcations. A planar differential system is said to be \emph{structurally stable} if its qualitative dynamics do not change under small perturbations; that is, its phase portraits remain topologically equivalent. Among the most influential theories of structural stability for dynamical systems are those developed by Andronov and his collaborators~\cite{AndroLeontGor1973}, and by Peixoto~\cite{PeiPei1959,Pei1962}, who provided necessary and sufficient conditions for the structural stability of smooth two-dimensional differential systems.

Once structurally stable systems were understood, attention naturally shifted toward the study of bifurcations. If a planar differential system is not structurally stable, then it belongs to the \emph{bifurcation set}. In this case, the qualitative structure of the solution set changes abruptly when the system passes through a point in the bifurcation set; see, for instance,~\cite{Kuznetsov1998,Perko2013}. Moreover, bifurcations can be classified according to their degree of degeneracy. Roughly speaking, the \emph{codimension} of a bifurcation is the minimum number of parameters required to describe all possible dynamics of differential systems in a neighborhood of a given vector field in the bifurcation set. Consequently, higher codimension corresponds to a higher level of degeneracy.

%In this direction, we are interested in the analysis of stability and bifurcations of planar piecewise smooth systems through a smoothing process. The interest in such systems has increased significantly since the pioneering works of Van der Pol \cite{Van1920,Van1926} and Andronov \cite{AndroLeontGor1973,AndVitKha1987}, mainly motivated by its range of applications in several areas of applied sciences. Using techniques involving control theory and differential inclusions we have some results about dynamics in the switching curve \cite{AlexSei1998,DieLop2011}. Piecewise smooth differential systems are used for modeling phenomena presenting some abrupt behavior in their dynamics.  There are problems in engineering, physics, and biology producing this kind of behavior, see \cite{BerBudChaKow2008,Hen1997,KusGraRin2003,LeiNij2004,MakLam2012,MakLamb2012,Murray2002,SimJohn2010,ZhuMose2003}. The terminology and basic results of differential systems in this context were also stated by Filippov in \cite{Fil1988}. And, piecewise smooth systems using the convention of Fillipov are called Filippov systems.

In this direction, we are interested in the analysis of stability and bifurcations of planar piecewise smooth systems via a smoothing process. Interest in such systems has increased significantly since the pioneering works of Van der Pol~\cite{Van1920,Van1926} and Andronov~\cite{AndroLeontGor1973,AndVitKha1987}, mainly motivated by their wide range of applications in several areas of the applied sciences. 

Using techniques from control theory and differential inclusions, several results have been obtained concerning the dynamics on the switching manifold; see, for instance,~\cite{AlexSei1998,DieLop2011}. Piecewise smooth differential systems are commonly used to model phenomena that exhibit abrupt changes in their dynamics. Such behavior arises in problems from engineering, physics, and biology; see, for example,~\cite{BerBudChaKow2008,Hen1997,KusGraRin2003,LeiNij2004,MakLam2012,MakLam2012,Murray2002,SimJohn2010,ZhuMose2003}. 

The terminology and fundamental results for differential systems in this setting were established by Filippov~\cite{Fil1988}. Accordingly, piecewise smooth systems defined using Filippov’s convention are referred to as \emph{Filippov systems}.

%Inspired by the works \cite{Kozlova1984,LliSot1996,SotTei1998,Teixeira1977,Teixeira1990}, in the literature there is a generic Structural Stability Theorem for Filippov systems on surfaces \cite{BrouPugSlo2001}, being a generalization of Peixoto's result \cite{Pei1962}. In addition, for planar Filippov systems, Kuznetsov, Gragnani, and Rinaldi presented a bifurcation list of codimension one \cite{KusGraRin2003}. Using the general approach given in \cite{Sotomayor1974}, in 2011 Guardia, Seara, and Teixeira characterized generic bifurcations of codimension two \cite{GuaSeaTeixeira2011}, see also \cite{DiPagPon2008} for some higher codimension bifurcations.

Inspired by the works~\cite{Kozlova1984,LliSot1996,SotTei1998,Teixeira1977,Teixeira1990}, a generic Structural Stability Theorem for Filippov systems on surfaces was established in~\cite{BruPuSim2001}, as a generalization of Peixoto’s classical result~\cite{Pei1962}. In addition, for planar Filippov systems, Kuznetsov, Gragnani, and Rinaldi presented a complete list of codimension-one bifurcations in~\cite{KusGraRin2003}. Using the general approach developed in~\cite{Sotomayor1974}, Guardia, Seara, and Teixeira characterized the generic codimension-two bifurcations in~\cite{GuaSeaTeixeira2011}; see also~\cite{DiPagPon2008} for results concerning higher-codimension bifurcations.

%Modeling rigid-body dynamics with friction is difficult due to several discontinuities appearing in the behavior of rigid bodies and in the Coulomb friction law \cite{Stewart2000}. Therefore, during the last few years, considerable efforts have been made to develop more tools for studying this kind of system. Among such tools, we highlight the smoothing process, consisting of a way to approximate piecewise smooth systems by smooth systems.
%The advantage of a smoothing process is that all the classic results of smooth systems can be applied in this context. One of the most widespread smoothing processes in the research literature corresponds to the Sotomayor–Teixeira regularization \cite{LliTei1997,SotTei1998}. This fact is mainly because of its intrinsic relation with Filippov’s convention in systems having a regular switching curve, see \cite{TeiSilva2012}. For example, in \cite{BuzSilTei2006, LliSilTei2007,LliSilTei2008}, using a mix of singular perturbation theory, blowing-up method, and Sotomayor–Teixeira regularization the authors studied the local behavior around a switching regular.

Modeling rigid-body dynamics with friction is challenging due to the presence of several discontinuities in the behavior of rigid bodies and in the Coulomb friction law~\cite{Stewart2000}. Consequently, in recent years, considerable efforts have been devoted to developing analytical tools for studying this type of system. Among these tools, we highlight smoothing processes, which provide a way to approximate piecewise smooth systems by smooth ones. An important advantage of smoothing processes is that they allow the application of the classical theory of smooth dynamical systems. One of the most widely used smoothing techniques in the literature is the Sotomayor--Teixeira regularization~\cite{LliTei1997,SotTei1998}. This is mainly due to its intrinsic connection with Filippov’s convention for systems with a regular switching curve; see~\cite{TeiSilva2012}. For instance, in~\cite{BuzSilTei2006,LliSilTei2007,LliSilTei2008}, by combining singular perturbation theory, the blow-up method, and Sotomayor--Teixeira regularization, the authors studied the local dynamics near a regular switching manifold.

%Sotomayor-Teixeira regularization can be seen as the main example in the set of linear regularizations \cite{SilMezNov2022,PerRonSil2023}, and as an alternative to it, we can cite the hysteresis switching  \cite{Liberzon2003}. Jeffrey \cite{Jeffrey2014} studied a simple piecewise smooth model of an object's stick-slip motion on a flat surface, revealing that linear regularizations might not be sufficiently general for real applications. Therefore,  linear regularizations neglect a vast expanse of phenomena occurring in smooth systems that approach piecewise smooth systems, \cite{Jeffrey2014,NovJeff2015}. 
%
%With the discussion until here, assuming that bifurcations of codimension zero are equivalent to structural stability then in the study of bifurcations, some natural questions arise: What happens to a particular bifurcation of codimension $k$ in the world of piecewise smooth systems through some smoothing process? Is there a change in the bifurcation codimension?

The Sotomayor--Teixeira regularization can be regarded as a canonical example within the class of linear regularizations~\cite{SilMezNov2022,PerRonSil2023}. As an alternative approach, we can cite the hysteresis switching  \cite{Liberzon2003}. Jeffrey~\cite{Jeffrey2014} studied a simple piecewise smooth model of an object undergoing stick--slip motion on a flat surface, showing that linear regularizations may not be sufficiently general for real applications. Consequently, linear regularizations can neglect a wide range of phenomena that arise in smooth systems approaching piecewise smooth dynamics~\cite{Jeffrey2014,NovJeff2015}.

In view of the discussion until here, and assuming that codimension-zero bifurcations correspond to structural stability, several natural questions arise in the study of bifurcations. What happens to a given bifurcation of codimension $k$ in a piecewise smooth system under a smoothing process? In particular, is the bifurcation codimension preserved or altered?

%Facing these questions, some authors have been working to find answers. Considering Filippov systems defined in the two-dimensional sphere partitioned into two connected components separated by a regular switching curve,  Sotomayor and Teixeira \cite{SotTei1998},  using its regularization, provided conditions ensuring that regularized vector fields are structurally stable. Sotomayor and Machado \cite{SotoMac2002} adjusted the conditions given in \cite{SotTei1998} to planar Filippov systems in a domain partitioned into two subdomains also separated by a regular switching curve. Moreover, considering the planar Filippov systems described in \cite{SotoMac2002} and the bifurcations classified in \cite{GuaSeaTeixeira2011}, then using a particular case of Sotomayor-Teixeira regularization, some answers about the previous question for $k=1,2$ were given in the works \cite{BuzSan2019,Maciel2009}. In \cite{Maciel2009}, the authors confirmed that codimension one bifurcations, in the world of piecewise smooth systems, are preserved 
%in the world of smooth systems. And, in the same line, Buzzi and Santos \cite{BuzSan2019} proved that the codimension two of the saddle-fold singularity is also preserved. To understand how piecewise smooth systems near the fold singularity evolve under smoothing processes is also very important to mention \cite{BonLarTere2018,BonTere2016,KriHog2015,NovRon2021}.

Facing these questions, several authors have sought answers. For Filippov systems defined on the two-dimensional sphere, partitioned into two connected components by a regular switching curve, Sotomayor and Teixeira~\cite{SotTei1998} used their regularization to establish conditions guaranteeing that the regularized vector fields are structurally stable. Sotomayor and Machado~\cite{SotoMac2002} later adapted these conditions to planar systems in domains similarly divided by a regular switching curve. Moreover, considering the planar Filippov systems described in \cite{SotoMac2002} and the bifurcations classified in \cite{GuaSeaTeixeira2011}, a particular application of the Sotomayor–Teixeira regularization provided insights for codimensions $k=1,2$~\cite{BuzSan2019,Maciel2009}. Specifically, Maciel~\cite{Maciel2009} confirmed that codimension-one bifurcations in piecewise smooth systems are preserved under smoothing in the world of smooth systems, while Buzzi and Santos~\cite{BuzSan2019} demonstrated the preservation of the codimension-two saddle–fold bifurcation. The evolution of piecewise smooth systems near fold singularities under smoothing processes is further explored in~\cite{BonLarTere2018,BonTere2016,KriHog2015,NovRon2021}.

%All the references cited in the last paragraph are situations where Filippov’s convention can be applied. However, there exist some problems modeled by piecewise smooth systems having a nonregular switching curve \cite{AprBanVanBruBouFaIrRadRinVee2012, FanXueChen2018, LucGaj2006,TanOsoBer2009}. In these cases, the switching curve can possess codimension greater than one in some points, and Filippov’s convention \cite{Fil1988} does not apply. Because of this, a few years ago it has been increased the study of these kinds of systems \cite{BuzMedTor2020, BerBudCha2001, LliTei2015, Nov2014} where in the works  \cite{LliSilTei2009, LliSilTei2015} for example, the authors provided a systematic approach to study them. And, initial primordials works of stability and bifurcation must to be highlighted \cite{SanTon2023, LarSeaTei2021}.

All the references cited in the previous paragraph concern situations where Filippov’s convention can be applied. However, there exist problems modeled by piecewise smooth systems with a nonregular switching curve \cite{AprBanVanBruBouFaIrRadRinVee2012, FanXueChen2018, LucGaj2006, TanOsoBer2009}. In these cases, the switching curve may have codimension greater than one at certain points, so Filippov’s convention \cite{Fil1988} does not apply. Consequently, in recent years, the study of such systems has increased \cite{BuzMedTor2020, BerBudCha2001, LliTei2015, Nov2014}. For instance, in \cite{LliSilTei2009, LliSilTei2015}, the authors presented a systematic approach for analyzing these systems. Additionally, the foundational works on stability and bifurcation should be highlighted \cite{SanTon2023, LarSeaTei2021}.

In this paper, we aim to analyze low-codimension bifurcations in a class of planar piecewise smooth systems with a nonregular switching curve under a smoothing process. The analysis focuses on determining whether the codimension of a bifurcation is preserved after the smoothing. Our results are directly applicable to a given system, as we work in the original variables of the system.

% More concretely, we have already seen the natural rise of piecewise smooth systems in the context of many real applications. Motivated by equations expressed as
% \begin{equation}\label{eqmotiv}
% 	\ddot{x} + a \dot x = sgn( x \, f(x,\dot{x})),
% \end{equation}  
% which are founded in Control Theory and Engineering, see \cite{Anosov1959,PerBar2002}. Moreover, this kind of discontinuity was introduced to solve a stabilization problem \cite{Barbashin1970}. In this way, the authors in  \cite{LarSeaTei2021} understood that was necessary a systematic program toward the bifurcation theory for the system \eqref{eqmotiv} with $f(x,\dot{x})=x \cdot \dot{x}$.
  
  More concretely, we have already seen the natural emergence of piecewise smooth systems in the context of many real-world applications. Motivated by equations of the form
  \begin{equation}\label{eqmotiv}
  	\ddot{x} + a \dot{x} = \mathrm{sgn}\big(x \, f(x,\dot{x})\big),
  \end{equation}
  which arise in Control Theory and Engineering, see \cite{Anosov1959,PerBar2002}, this type of discontinuity was also introduced to address a stabilization problem \cite{Barbashin1970}. In this context, the authors in \cite{LarSeaTei2021} recognized the necessity of a systematic program for the bifurcation theory of system \eqref{eqmotiv}, particularly when $f(x,\dot{x}) = x \cdot \dot{x}$.

%  
%     In this direction,  let $X, Y: \mathcal{U} \subset \re^2 \rightarrow \re^2$ two sufficiently smooth vector fields defined in a bounded neighborhood $U$ of the origin. Considering $f:\mathcal{U} \subset \re^2  \rightarrow x_1 \, x_2 \in \re$ such that $f(x_1,x_2) = x_1 \, x_2 $, follow that the origin is a nondegenerated critical point of $f$. And, $\Sigma = f^{-1}(0)$ is a nonregular curve which is  the union of two lines that intersect transversally at the origin, that is, $\Sigma=\Sigma_1 \cup \Sigma_2$ with 
     
     In this context, let $X, Y: \mathcal{U} \subset \mathbb{R}^2 \rightarrow \mathbb{R}^2$ be two sufficiently smooth vector fields defined on a bounded neighborhood $\mathcal{U}$ of the origin. Consider the function $f: \mathcal{U} \subset \mathbb{R}^2 \rightarrow \mathbb{R}$ defined by $f(x_1,x_2) = x_1 x_2$. Then, the origin is a nondegenerate critical point of $f$, and $\Sigma = f^{-1}(0)$ is a nonregular curve given by the union of two lines intersecting transversally at the origin, that is, $\Sigma = \Sigma_1 \cup \Sigma_2$, with
  \begin{equation*}
  	\Sigma_1=\{(x_1,x_2) \in \mathcal{U}  :x_1=0\} \; \; \;  \; \mbox{and} \; \; \; \;  \Sigma_2=\{(x_1,x_2) \in \mathcal{U}: x_2=0\}.  
  \end{equation*}
  %     Since $\Sigma$ can be seen as the union of two regular manifolds, $\Sigma= \Sigma_1 \cup \Sigma_2$,  then we can 
 In addition, we decompose $\Sigma_1$ as $\Sigma_1 = \overline{\Sigma_1^+ \cup \Sigma_1^-}$, where
  \begin{equation*}
  	\Sigma_1^+ =\{ (0,x_2) \in \Sigma_1 : x_2>0\} \; \; \;  \; \mbox{and} \; \; \; \; \Sigma_1^- =\{ (0,x_2) \in \Sigma_1 : x_2<0\} 
  \end{equation*}
  are regular curves. Analogously, we can write $\Sigma_2 = \overline{\Sigma_2^+ \cup \Sigma_2^-}$. From now on, let $\Omega$ denote the set of all piecewise vector fields $Z=(X,Y)$ such that
  %which can be described as
  \begin{equation}\label{sistemacruz}
  	Z(p) = 
  	%\left\{ \begin{array}{l}
  		X(p) \; \mbox{if} \; p \in \overline{\mathcal{U}^+} \; \; \;  \; \mbox{and} \; \; \; \;
  		%=\overline{\Sigma_+^+ \cup \Sigma_-^+}
  		%	\\
  		Z(p) =  Y(p) \; \mbox{if} \; p \in \overline{\mathcal{U}^-}
  		% =\overline{\Sigma_+^- \cup \Sigma_-^-}\
  		%	\end{array} \right. ,
  \end{equation}
      where $\mathcal{U}^+ = \{(x_1,x_2) \in \mathcal{U} : f(x_1,x_2) > 0\}$ and 
    $\mathcal{U}^- = \{(x_1,x_2) \in \mathcal{U} : f(x_1,x_2) < 0\}$ are the open sets separated by $\Sigma$, such that 
    $\mathcal{U}^\pm = \Sigma_+^\pm \cup \Sigma_-^\pm$, with 
    $\Sigma_+^\pm = \mathcal{U}^\pm \cap \Sigma_1^+$ and $\Sigma_-^\pm = \mathcal{U}^\pm \cap \Sigma_1^-$.  
    Observe that the system \eqref{sistemacruz} is a planar piecewise vector field with a nonregular switching curve 
    $\Sigma = f^{-1}(0)$, see Figure \ref{figura4regioes}.

 \begin{figure}[!htb]
 	
 	\begin{center}

 		\begin{overpic}[scale=0.4,tics=5]{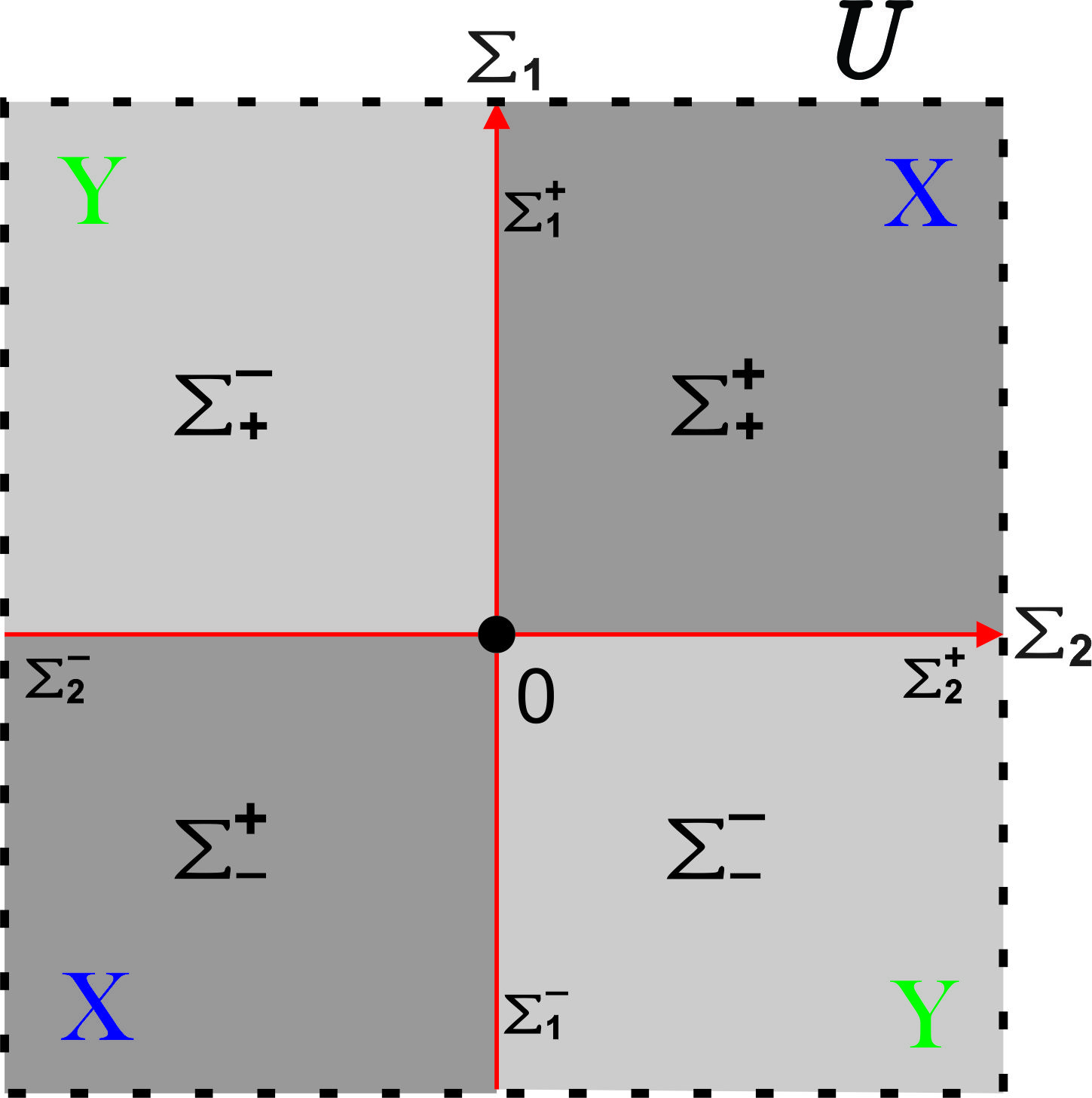}

 		\end{overpic}

 	\end{center}
 	
 		\caption{Decomposition of the set $\mathcal{U}$.
 		}\label{figura4regioes}
 		
 \end{figure}

Considering Filippov's convention on $\Sigma \setminus \{(0,0)\}$, choosing a definition to remove the ambiguity of the orbits passing through the origin, and using the concepts of local $\Sigma$-equivalence and weak equivalence of unfoldings for the systems in $\Omega$ \eqref{sistemacruz}, Larrosa, M-Seara, and Teixeira \cite{LarSeaTei2021} established, near the origin, the generic conditions for the set of locally $\Sigma$-structurally stable vector fields, denoted by $\Omega_0 \subset \Omega$. They also classified the fields exhibiting a generic bifurcation of codimension one at the origin, $\Xi_1 \subset \Omega_1 = \Omega \setminus \Omega_0$. More details are provided in Section \ref{section2}. 

Our aim is to study the smoothing of the families $\Omega_0$ and $\Xi_1$. Before presenting our results, we highlight the smoothing process to be used. In general, regularizing a piecewise smooth vector field $Z$ in \eqref{sistemacruz} consists of obtaining a $k$-parameter family of smooth vector fields $Z^R_{\epsilon}$, with $\epsilon \in \mathbb{R}^k$, such that $Z^R_{\epsilon}$ converges pointwise to $Z$ in $U \setminus \Sigma$ with respect to a given topology. The characteristics of $Z$ and each $Z^R_{\epsilon}$ can exhibit intricate and fascinating behavior. Utkin \cite{Utkin2013} and later Panazzolo and Silva \cite{PanPau2017} observed that different regularization processes may lead to different sliding dynamics. 

As previously mentioned, the behavior of the regularization $Z^R_{\epsilon}$ provides valuable information for understanding the dynamics of $Z$; see also \cite{LliMerNov2015}, where averaging theory was used in this context. Regularization has been applied to study the dynamics of piecewise smooth vector fields whose nonregular switching curve contains points of codimension two, formed as the union of two regular curves \cite{DieEli2015}. Inspired by this, we use the Sotomayor–Teixeira double regularization \cite{SilNun2019,PanPau2017}, a generalization of the Sotomayor–Teixeira regularization applied twice.

%In this way, let us consider $X_{+,+}=X_{-,-}=X$ and $X_{-,+}=X_{+,-}=Y$, then for the vector fields in \eqref{sistemacruz}, the Sotomayor-Teixeira double regularization will be 
In this way, let us consider $X_{+,+} = X_{-,-} = X$ and $X_{-,+} = X_{+,-} = Y$. Then, for the vector fields defined in \eqref{sistemacruz}, the Sotomayor--Teixeira double regularization is given by
\begin{equation}\label{eq2par}
	\begin{aligned}
	Z^R_{\varepsilon, \eta}(p) &= \displaystyle\sum_{ \alpha,\beta \in \lbrace +,- \rbrace } \frac{1}{4} \left( 1 + \alpha \,  \varphi_\varepsilon(x_1) \right) \left( 1 + \beta  \, \psi_\eta(x_2) \right) \, X_{ \alpha, \beta }(p)= \\
	& \phantom{\displaystyle\sum_{ \alpha,\beta \in \lbrace +,- \rbrace }}=\frac{1}{2}  \left( 1 +\varphi_\varepsilon(x_1) \, \psi_\eta(x_2)\right) \, X(p) + \frac{1}{2} \,  \left( 1 - \varphi_\varepsilon(x_1) \, \psi_\eta(x_2) \right) \,Y(p) ,
		\end{aligned}
\end{equation}
% where $\varphi_\varepsilon(x_1)=\varphi(x_1/\varepsilon)$ and $\psi_\eta(x_2)=\psi(x_2/\eta)$. Here, $\varepsilon$ and $\eta$ are the {\it regularization parameters} which are suficiently small and,  $\varphi,\psi$ are sufficiently piecewise smooth functions called transitions functions satisfying  $\varphi(s)=\psi(s)=sgn(s)$ for $|s| \geq 1$ and $\varphi^\prime(s)=\psi^\prime(s)>0$ for $|s|<1$. 
where $\varphi_\varepsilon(x_1) = \varphi(x_1/\varepsilon)$ and $\psi_\eta(x_2) = \psi(x_2/\eta)$. Here, $\varepsilon$ and $\eta$ are the \textit{regularization parameters}, which are sufficiently small, and $\varphi$ and $\psi$ are sufficiently smooth functions, called \textit{transition functions}, satisfying $\varphi(s) = \psi(s) = \operatorname{sgn}(s)$ for $|s| \geq 1$ and $\varphi'(s) = \psi'(s) > 0$ for $|s| < 1$.
 %Observe that, i
% In the equation \eqref{eq2par}, when $x_1\, x_2 \neq 0$ we have two Sotomayor-Teixeira regularizations, respectively, for $\eta>0$ and $\varepsilon>0$:
 In equation~\eqref{eq2par}, when $x_1 x_2 \neq 0$, we obtain two Sotomayor--Teixeira regularizations, corresponding respectively to $\eta > 0$ and $\varepsilon > 0$:
 
\begin{equation*}\label{eq1par} 
	Z_{\pm, \eta}^R(p) = \displaystyle\sum_{ \beta \in \lbrace +,- \rbrace }  \frac{1}{2}\left( 1 \pm \beta  \psi_\eta(x_2) \right) X_{ \pm, \beta }(p) \; \, \mbox{and} \; \, Z_{\varepsilon, \pm}^R(p) = \displaystyle\sum_{ \alpha \in \lbrace +,- \rbrace }  \frac{1}{2}\left( 1 \pm  \alpha  \varphi_\varepsilon(x_1) \right) X_{ \alpha, \pm }(p),
\end{equation*}
where the switching curve of $Z_{\pm, \eta}^R$ and $Z_{\varepsilon, \pm}^R $ are, $\Sigma_2^\pm$ and $ \Sigma_1^\pm$. Sotomayor-Teixeira double regularization  permit the exact configuration of Figure \ref{figuraregdulpa}.
\begin{figure}[!htb]\label{regularizacaocruz}
	
	\begin{center}

		\vspace{0.4cm}
		
		\begin{overpic}[scale=0.9,tics=5]{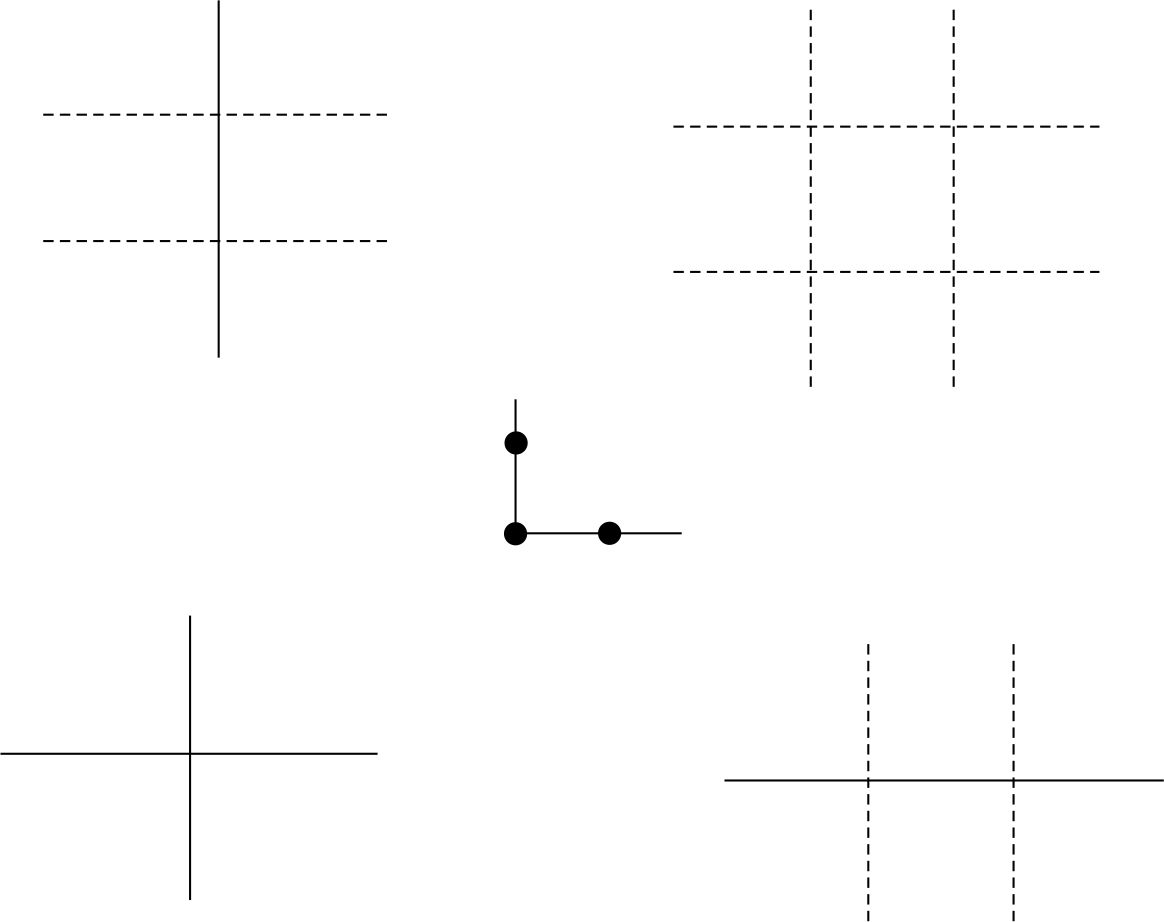}
			\put(-16,159){{\parbox{0.5\linewidth}{$2 \eta \; \bigg\updownarrow$}}}
			
			\put(19,158){{\parbox{0.5\linewidth}{$Z_{-, \eta}^R$}}}
			
			\put(54,158){{\parbox{0.5\linewidth}{$Z_{+, \eta}^R$}}}
			
			\put(17,182){{\parbox{0.5\linewidth}{$X_{-, +}$}}}
			
			\put(53,182){{\parbox{0.5\linewidth}{$X_{+, +}$}}}
			
			\put(17,131){{\parbox{0.5\linewidth}{$X_{-, -}$}}}
			
			\put(53,131){{\parbox{0.5\linewidth}{$X_{+, -}$}}}
			
			\put(83,113){{\parbox{0.5\linewidth}{$\displaystyle\nwarrow$}}}
			
			\put(100,103){{\parbox{0.5\linewidth}{$\eta$}}}
			
			%%%%%%%%%%%%%%%%%%%%%%%
			
			\put(187,94){{\parbox{0.5\linewidth}{$2\varepsilon$}}}
			
			\put(181,107){{\parbox{0.5\linewidth}{$\displaystyle\longleftrightarrow$}}}

			\put(239,154){{\parbox{0.5\linewidth}{$\bigg\updownarrow \; 2\eta$}}}
			
			\put(131,105){{\parbox{0.5\linewidth}{$\displaystyle\nearrow$}}}
			
			\put(148,154){{\parbox{0.5\linewidth}{$Z_{-, \eta}^R$}}}
			
			\put(211,154){{\parbox{0.5\linewidth}{$Z_{+, \eta}^R$}}}
			
			\put(182,154){{\parbox{0.5\linewidth}{$Z_{\varepsilon, \eta}^R$}}}
			
			\put(182,182){{\parbox{0.5\linewidth}{$Z_{\varepsilon, +}^R$}}}
			
			\put(182,125){{\parbox{0.5\linewidth}{$Z_{\varepsilon, -}^R$}}}
			
			\put(147,183){{\parbox{0.5\linewidth}{$X_{-, +}$}}}
			
			\put(212,183){{\parbox{0.5\linewidth}{$X_{+, +}$}}}
			
			\put(147,124){{\parbox{0.5\linewidth}{$X_{-, -}$}}}
			
			\put(212,124){{\parbox{0.5\linewidth}{$X_{+, -}$}}}
			
			%%%%%%%%%%%%%%%%%%%%%%
			
			\put(140,60){{\parbox{0.5\linewidth}{$\searrow$}}}

			\put(130,70){{\parbox{0.5\linewidth}{$\varepsilon$}}}

			\put(199,-22){{\parbox{0.5\linewidth}{$2\varepsilon$}}}
			
			\put(193,-10){{\parbox{0.5\linewidth}{$\displaystyle\longleftrightarrow$}}}
			
			\put(194,40){{\parbox{0.5\linewidth}{$Z_{\varepsilon, +}^R$}}}
			
			\put(194,12){{\parbox{0.5\linewidth}{$Z_{\varepsilon, -}^R$}}}

			\put(159,43){{\parbox{0.5\linewidth}{$X_{-, +}$}}}
			
			\put(224,43){{\parbox{0.5\linewidth}{$X_{+, +}$}}}
			
			\put(159,10){{\parbox{0.5\linewidth}{$X_{-, -}$}}}
			
			\put(224,10){{\parbox{0.5\linewidth}{$X_{+, -}$}}}
			
			%%%%%%%%%%%%%%%%

			\put(90,65){{\parbox{0.5\linewidth}{$\swarrow$}}}
			
			\put(8,50){{\parbox{0.5\linewidth}{$X_{-, +}$}}}
			
			\put(48,50){{\parbox{0.5\linewidth}{$X_{+, +}$}}}
			
			\put(8,18){{\parbox{0.5\linewidth}{$X_{-, -}$}}}
			
			\put(48,18){{\parbox{0.5\linewidth}{$X_{+, -}$}}}
			%			\put(6,-20){{\parbox{0.5\linewidth}{Fonte:  Elaborado pelo autor.}}}
		\end{overpic}
		
		\vspace{0.5cm}
		
		\caption{Sotomayor--Teixeira double regularization.}\label{figuraregdulpa}
	\end{center}
	
\end{figure}
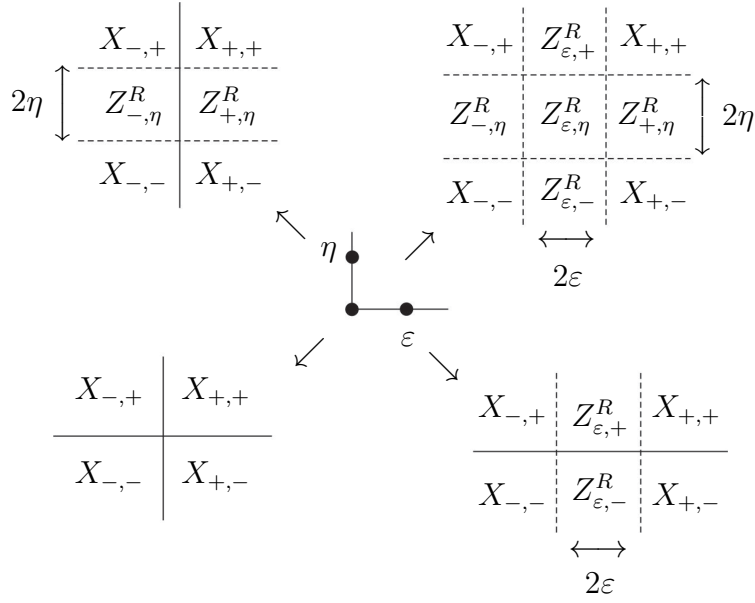

%A double regularization of  the piecewise vector field \eqref{sistemacruz} is a two-parameter family of smooth vector fields
%$Z_{\varepsilon, \eta}^R:\mathcal{U} \rightarrow \re^2$, such  that $Z_{\varepsilon, \eta}^R$ converges pointwise to $Z$ on $\mathcal{U} \setminus \Sigma$, when $\varepsilon, \eta \rightarrow 0$. Denoting by 
%$$\left[ X,Y\right]= \left\{ \frac{1}{2} \left( 1 +\lambda  \right) \, X(p) +\frac{1}{2} \left( 1 - \lambda \,  \right) \, Y(p) : \lambda \in [-1,1] \right\}$$
%the convex combination of $X$ and $Y.$ Then, a double regularization of \eqref{sistemacruz} is linear if $Z_{\varepsilon, \eta}^R(p) \in \left[ X(p),Y(p)\right]$, for any $p \in \mathcal{U}$. The Sotomayor-Teixeira double regularization \eqref{eq2par} is an example of a linear regularization. It is also possible to consider more general regularizations. A continuous combination of \eqref{sistemacruz} is an one-parameter family of smooth vector fields $\tilde{Z}^R(\lambda, \cdot),$ with $\lambda \in [-1,1] $ such that $\tilde{Z}^R(1,p)=X(p)$ and $\tilde{Z}^R(-1,p)=Y(p)$, for all $p \in \mathcal{U}$. The double regularization $Z_{\varepsilon, \eta}^R$ of \eqref{sistemacruz} will be called nonlinear if there exists a continuous combination $\tilde{Z}^R$ sucht that $Z_{\varepsilon, \eta}^R \in \{\tilde{Z}^R(\lambda,p): \lambda \in [-1,1]\}$.

A double regularization of the piecewise vector field \eqref{sistemacruz}
is a two-parameter family of smooth vector fields
$Z_{\varepsilon,\eta}^R : \mathcal{U} \to \re^2$
such that $Z_{\varepsilon,\eta}^R$ converges pointwise to $Z$
on $\mathcal{U} \setminus \Sigma$ as $\varepsilon,\eta \to 0$.
Denote by
\[
\bigl[X,Y\bigr]
=
\left\{
\frac{1}{2}(1+\lambda)\,X(p)
+
\frac{1}{2}(1-\lambda)\,Y(p)
:\;
\lambda \in [-1,1]
\right\}
\]
the convex combination of the vector fields $X$ and $Y$ at the point $p$.
We say that a double regularization of \eqref{sistemacruz} is
\emph{linear} if
\[
Z_{\varepsilon,\eta}^R(p) \in \bigl[X(p),Y(p)\bigr],
\qquad \text{for all } p \in \mathcal{U}.
\]
The Sotomayor--Teixeira double regularization \eqref{eq2par}
is an example of a linear regularization.

More general regularizations can also be considered.
A \emph{continuous combination} of \eqref{sistemacruz} is a one-parameter
family of smooth vector fields
$\tilde{Z}^R(\lambda,\cdot)$, with $\lambda \in [-1,1]$,
such that
$\tilde{Z}^R(1,p)=X(p)$ and $\tilde{Z}^R(-1,p)=Y(p)$
for all $p \in \mathcal{U}$.
A double regularization $Z_{\varepsilon,\eta}^R$ of \eqref{sistemacruz}
is called \emph{nonlinear} if there exists a continuous combination
$\tilde{Z}^R$ such that
\[
Z_{\varepsilon,\eta}^R(p)
\in
\{\tilde{Z}^R(\lambda,p) : \lambda \in [-1,1]\},
\qquad \text{for all } p \in \mathcal{U}.
\]

%The transition functions used in \eqref{eq2par} have positive derivative on the interval $(-1,1)$, because of this, we say that they are monotonous transition functions, see Figure \ref{transicoes} (a)-(b). We can consider more general transition functions, for instance allowing them not to be monotonous \cite{PerRonSil2023}, called nonmonotonous transition functions (Figure \ref{transicoes} (c)).
The transition functions used in \eqref{eq2par} have positive derivatives on the interval
$(-1,1)$, and are therefore referred to as \emph{monotone transition functions} (see
Figure~\ref{transicoes}(a)--(b)). Alternatively, one can consider \emph{nonmonotone transition functions} \cite{PerRonSil2023}, which do not satisfy the monotonicity property, as illustrated in Figure~\ref{transicoes}(c).
%, or even the loss of differentiability in a finite number of points, possibiliting the enrichment the dynamics os the differential system. 
%Therefore, we are interested in more general transition functions. 
Therefore, we are interested in \emph{nonmonotone transition functions}.

\begin{defi}\label{defitrang}
	%A sufficiently piecewise differentiable function differentiable function 
%	A sufficiently piecewise smooth function $\varphi: \re \rightarrow \re $ such that 
%\begin{equation}\label{condfuncaotransi}
%	\displaystyle\lim_{x \rightarrow - \infty} \varphi(x)=-1 \; \; \; \; \; \; \mbox{and} \; \; \; \; \; \; \displaystyle\lim_{x \rightarrow + \infty} \varphi(x)=1
%\end{equation}
%will be called a transition function generalized.
A sufficiently piecewise smooth function $\varphi: \mathbb{R} \rightarrow \mathbb{R}$ satisfying
\begin{equation}\label{condfuncaotransi}
	\lim_{x \rightarrow -\infty} \varphi(x) = -1, \quad \text{and} \quad 
	\lim_{x \rightarrow +\infty} \varphi(x) = 1
\end{equation}
will be called a \emph{generalized transition function}.

\end{defi}
 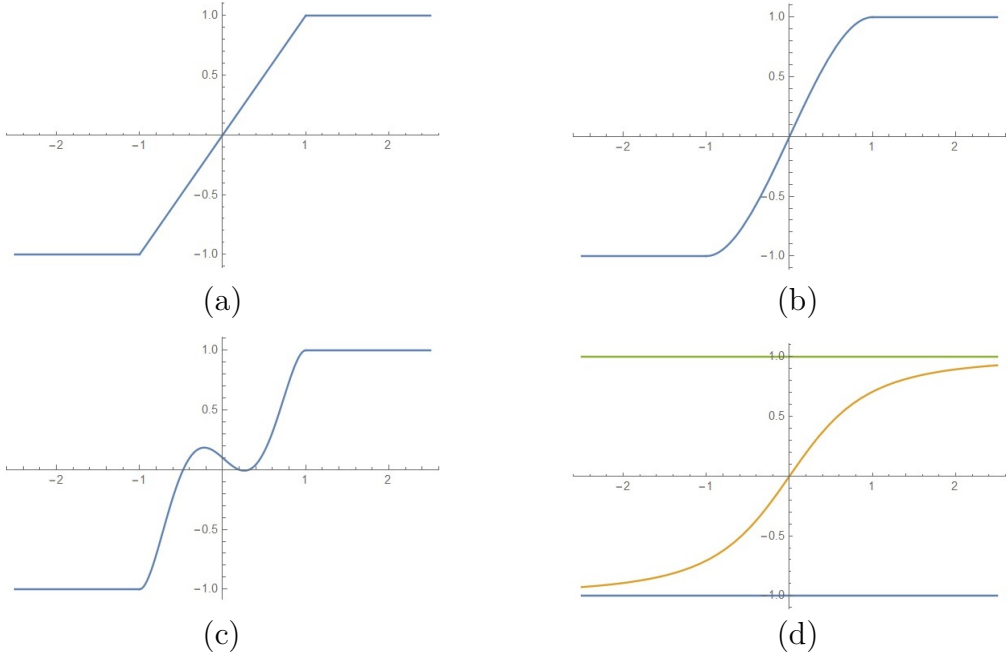
\begin{figure}
 	\begin{center}
 		
 		\begin{overpic}[scale=0.4,tics=5]{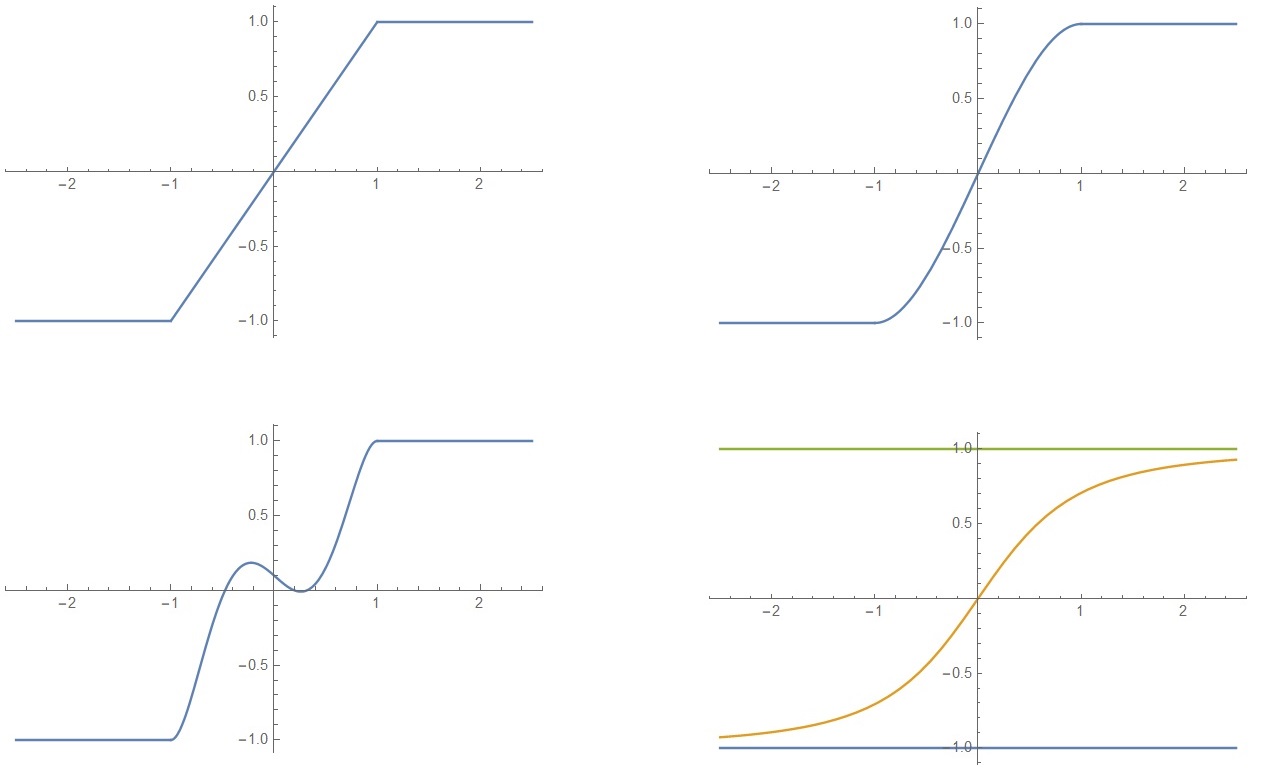}
 			\vspace{0.2cm}
 			
 			%			\put(80,342){{(A)}}

 			\put(75,116){{(a)}}
 			
 			\put(290,116){{(b)}}
 			
 			\put(75,-10){{(c)}}
 			
 			\put(290,-10){{(d)}}
 			
 			%		\put(143,103){{\parbox{0.55\linewidth}{$ 
 						%					\Sigma_1$}}}
 			%		
 			%		\put(375,103){{\parbox{0.55\linewidth}{$ 
 						%					\Sigma_1$}}}
 			%		
 			%		
 			%		
 			%		\put(200,39){{\parbox{0.55\linewidth}{$ 
 						%					\Sigma_2$}}}
 			%		
 			%		\put(430,39){{\parbox{0.55\linewidth}{$ 
 						%					\Sigma_2$}}}
 			%		

 		\end{overpic}

 		$\;$
 		
 		\caption{(a)-(b) Monotonous transition functions;
 			 (c) Nonmonotonous transition function; (d) Generalized transition function .}\label{transicoes}
 		
 	\end{center}
 	
 \end{figure}

%Following the Definition \ref{defitrang}, we have not always the exact configuration of Figure \ref{figuraregdulpa} for the Sotomayor-Teixeira double regularization.  There are situations in which we need to interpret Figure \ref{figuraregdulpa} with the vector fields $X_{ \pm, \pm}$ for points with a big norm, near infinity. These aspects became clearer later when we described some examples in Section \ref{section4}. Moreover, cases also can exist where the transition function is not monotone. That said, the Sotomayor-Teixeira double regularization \eqref{eq2par} can take values that are not in the convex combination of $X$ and $Y$, and even does not satisfy the conditions of the continuous combination of $X$ and $Y$. In other words, we need a more general definition than linear and nonlinear regularizations. For us, a continuous combination generalized of \eqref{sistemacruz} will be  a one-parameter family of smooth vector fields $\tilde{Z}_g^R(\lambda, \cdot),$ such that 
%$$	\displaystyle\lim_{\lambda \rightarrow 1}  \tilde{Z}_g^R(\lambda,p)=X(p) \; \mbox{and} \;  \displaystyle\lim_{\lambda \rightarrow -1}  \tilde{Z}_g^R(\lambda,p)=Y(p), \; \mbox{for all } \; p \in \mathcal{U}.$$
%Therefore, as before, the double regularization $Z_{\varepsilon, \eta}^R$ of \eqref{sistemacruz} will be nonlinear generalized if there exists a continuous combination generalized $\tilde{Z}_g^R$ sucht that $Z_{\varepsilon, \eta}^R \in \tilde{Z}_g^R(\lambda,p)$.

Following Definition~\ref{defitrang}, the Sotomayor-Teixeira double regularization does not always produce the exact configuration shown in Figure~\ref{figuraregdulpa}. There are situations in which Figure~\ref{figuraregdulpa} must be interpreted with the vector fields $X_{\pm,\pm}$ evaluated at points with large norm, i.e., near infinity. These aspects become clearer in the examples presented in Section~\ref{section4}. Moreover, there are cases where the transition function is not monotone. Consequently, the Sotomayor-Teixeira double regularization \eqref{eq2par} can take values outside the convex combination of $X$ and $Y$, and it may not satisfy the conditions of a continuous combination of $X$ and $Y$. In other words, a more general definition than linear or nonlinear regularizations is needed. For our purposes, a \emph{generalized continuous combination} of \eqref{sistemacruz} is a one-parameter family of smooth vector fields $\tilde{Z}_g^R(\lambda, \cdot)$ such that
\[
\lim_{\lambda \to 1} \tilde{Z}_g^R(\lambda,p) = X(p) \quad \text{and} \quad 
\lim_{\lambda \to -1} \tilde{Z}_g^R(\lambda,p) = Y(p), \quad \text{for all } p \in \mathcal{U}.
\]

%Accordingly, the double regularization $Z_{\varepsilon, \eta}^R$ of \eqref{sistemacruz} will be called a \emph{generalized nonlinear} regularization if there exists a generalized continuous combination $\tilde{Z}_g^R$ such that 
%\[
%Z_{\varepsilon, \eta}^R \in \tilde{Z}_g^R(\lambda,p).
%\]

Accordingly, the double regularization $Z_{\varepsilon, \eta}^R$ of \eqref{sistemacruz} is called a \emph{generalized nonlinear} regularization if there exists a generalized continuous combination $\tilde{Z}_g^R$ such that 
\[
Z_{\varepsilon, \eta}^R \in \tilde{Z}_g^R(\lambda,p) \quad \text{for some } \lambda \in [-1,1] \text{ and all } p \in \mathcal{U}.
\]

%a continuous combination of \eqref{sistemacruz} is a one parameter family of smooth vector fields $\tilde{Z}^R(\lambda,\cdot),$ with $\lambda \in [-1,1] $ such that $\tilde{Z}^R(1,p)=X(p)$ and $\tilde{Z}^R(-1,p)=Y(p)$, for all $p \in \mathcal{U}$. The double regularization $Z_{\varepsilon, \eta}^R$ \eqref{sistemacruz} will be called nonlinear if there exists a continuous double combination .$\tilde{Z}^R$ sucht that $Z_{\varepsilon, \eta}^R \in \{\tilde{Z}^R(\lambda_1 \, \lambda_2,p): \lambda \in [-1,1]\}$.

%observe que quando a imagem ta fora perdemos a linearidade

%t is monotonous (Figure \ref{transicoes}(d)) and otherwise it is non-monotonous (Figure \ref{transicoes}(c)). And as said before, if the transition function satisfies \eqref{proptransi}, we have a Sotomayor-Teixeira transition function, which is clearly monotonous Sotomayor, Teixeira, and their collaborators performed a double regularization process with two-time scales, in case the switching manifold has a transversal self-intersection. 

%Motivated by the works \cite{SilMezNov2022, NovJeff2015}, if $X_{+,+}=X_{-,-}=X=(X_1,X_2)$ and $X_{-,+}=X_{+,-}=Y=(Y_1,Y_2)$, then for the piecewise vector field \eqref{sistemacruz} we are going to consider the following class of nonlinear double regularization $Z^R_{\varepsilon, \eta}(p_0)=	(Z^{R_1}_{ \varepsilon , \eta }(p)	,Z^{R_2}_{ \varepsilon , \eta }(p)	)$:
Motivated by the works \cite{SilMezNov2022, NovJeff2015}, assuming 
$X_{+,+} = X_{-,-} = X = (X_1, X_2)$ and $X_{-,+} = X_{+,-} = Y = (Y_1, Y_2)$, 
for the piecewise vector field \eqref{sistemacruz}, we consider the following class of nonlinear double regularizations:
\[
Z^R_{\varepsilon, \eta}(p_0) = \big(Z^{R_1}_{\varepsilon, \eta}(p), Z^{R_2}_{\varepsilon, \eta}(p)\big).
\]

\begin{equation}\label{eq2par1}
	\begin{aligned}
		Z^R_{\varepsilon, \eta}(p) &= \displaystyle\sum_{ \alpha,\beta \in \lbrace +,- \rbrace } \frac{1}{4} \left( 1 + \alpha \,  \varphi_\varepsilon(x_1) \right) \left( 1 + \beta  \, \psi_\eta(x_2) \right) \, X_{ \alpha, \beta }(p) + \xi \, G( \varphi_\varepsilon(x_1) , \psi_\eta(x_2))=\\
			&\, \\
		&= \dfrac{1}{2}	\left( \begin{array}{ccc} 
			X_1(p) &  Y_1(p)\\
			X_2(p) &  Y_2(p)  
		\end{array} \right)
		\left( \begin{array}{l} 
		1 + \varphi_\varepsilon(x_1) \, \psi_\eta(x_2) \\
		1 - \varphi_\varepsilon(x_1) \, \psi_\eta(x_2) \\
		\end{array}\right)
		+
		\xi \,	\left( \begin{array}{l} 
		G_1( \varphi_\varepsilon(x_1) , \psi_\eta(x_2))  \\
			G_2( \varphi_\varepsilon(x_1) , \psi_\eta(x_2)) 
		\end{array}\right), \\
			\end{aligned}
	\vspace{0.3cm}
	\end{equation}
%	where $\xi$ is a sufficiently small parameter,
%	$\varphi_\varepsilon(x_1)=\varphi(x_1/\varepsilon)$ and $\psi_\eta(x_2)=\psi(x_2/\eta)$ are transition functions generalized, $G=(G_1,G_2)$
%is a continuous real function on a convenient domain
%such that $G(\pm1,\pm1)=0.$ As previously, $\varepsilon$ and $\eta$ are the {\it regularization parameters} which are sufficiently small, and respectively, for $\eta>0$ and $\varepsilon>0$, we also have two nonlinear regularizations:	
where $\xi$ is a sufficiently small parameter, 
$\varphi_\varepsilon(x_1) = \varphi(x_1 / \varepsilon)$ and $\psi_\eta(x_2) = \psi(x_2 / \eta)$ are generalized transition functions, 
and $G = (G_1, G_2)$ is a continuous real function defined on a suitable domain such that $G(\pm 1, \pm 1) = 0$. 
As before, $\varepsilon$ and $\eta$ are the \emph{regularization parameters}, which are sufficiently small. 
Moreover, for $\eta > 0$ and $\varepsilon > 0$, we also have two nonlinear regularizations:
	\begin{equation*}\label{eq1par} 
	\begin{aligned}
			& Z_{\pm, \eta}^R(p) = \displaystyle\sum_{ \beta \in \lbrace +,- \rbrace }  \frac{1}{2}\left( 1 \pm \beta  \psi_\eta(x_2) \right) X_{ \pm, \beta }(p) + \xi \, G(\pm1, \psi_\eta(x_2))\; \, \mbox{and} \\
			& Z_{\varepsilon, \pm}^R(p) = \displaystyle\sum_{ \alpha \in \lbrace +,- \rbrace }  \frac{1}{2}\left( 1 \pm  \alpha  \varphi_\varepsilon(x_1) \right) X_{ \alpha, \pm }(p) + \xi \, G( \varphi_\varepsilon(x_1),\pm 1),
	\end{aligned}
	\end{equation*}
where the switching curve of $Z_{\pm, \eta}^R$ and $Z_{\varepsilon, \pm}^R $ are, respectively, $\Sigma_2^\pm$ and $ \Sigma_1^\pm$.
%$$F(r,s)=\frac{1}{4} \, \left[r \, s + \frac{\sum_{j=1}^k \, a_j \, (2^2 \, f_j(r,s)-2^{\alpha_j} \, f_j(1,1) \, r\,s+)}{1+\sum_{j=1}^k \, a_j \,2^{\alpha_j} \, f_j(1,1)}\right] \; \mbox{with} \; f_j(1,0)+f_j(0,1)=0,$$
%and  $f_j$ is a fixed continuous $(m,n)$-quasi-homegeneous functions of degree $\alpha_j\ \in \re^+$ for each $j$, that is, given $m,n,\alpha_i \in \na$ we have $f_i(t^m \, r,t^n \, s)=t^{\alpha_i} \, f(r,s)$ for all $(r,s)\in \re^2$ and all $0\geq t \in \re$. The quasi-homogeneous functions are a natural generalization of homogeneous functions. 
% polynomial given by
%$F(r,s)=\sum_{i=1}^n a_i \, r^i \, s^i/\sum_{i=1}^n 4^i \, a_i$ with $\sum_{i=1}^n 4^i \, a_i  \neq 0$.
%Estou tentando definir F de modo que ela seja uma perturbacao da dupla reg
%Observe that when $a_j=0$ for all $j$ then we have the expression \eqref{eq2par}. 
%However, if $a_j\neq 0$ for some $j$ then the expression \eqref{eq2par1} can be seen as a some kind perturbation of \eqref{eq2par}.
% In addtion, considering $p=(x_1,x_2)$ then for each $x_1 \neq 0$ if $\varepsilon \rightarrow 0$ then $x_1/\varepsilon \rightarrow \sgn(x_1) \cdot \infty$, that is, $\varphi(x_1/\varepsilon) \rightarrow \sgn(x_1)$. And, the same happens for each $x_2 \neq 0$ and $\psi(x_2/\eta)$. Therefore, the nonlinear double regularization \eqref{eq2par1} is such that 
%$Z^R_{\varepsilon, \eta} \rightarrow X_{\pm,\pm} \in \{X,Y\}$, when $\varepsilon, \eta \rightarrow 0$.
In addition, considering $p=(x_1,x_2)$, for each $x_1 \neq 0$, as $\varepsilon \rightarrow 0$ we have $x_1/\varepsilon \rightarrow \sgn(x_1) \cdot \infty$, which implies $\varphi(x_1/\varepsilon) \rightarrow \sgn(x_1)$. Similarly, for each $x_2 \neq 0$, $\psi(x_2/\eta) \rightarrow \sgn(x_2)$ as $\eta \rightarrow 0$. Therefore, the nonlinear double regularization \eqref{eq2par1} satisfies 
$Z^R_{\varepsilon, \eta} \rightarrow X_{\pm,\pm} \in \{X,Y\}$ as $\varepsilon, \eta \rightarrow 0$.
% Remember that a quasi-homogeneous function $f: U \subset \re^2 \rightarrow \re $ has the following property:  $f(0,0)=0$.
%For instance, simple examples of the nonlinear double regularization \eqref{eq2par1} can be constructed considering $G_1$ and $G_2$ polynomials of degree $m_1$ and $m_2$, respectively. As we will see, our result about the existence of equilibrium points also can be applied to homogeneous functions given by
For instance, simple examples of the nonlinear double regularization \eqref{eq2par1} can be constructed by considering $G_1$ and $G_2$ as polynomials of degrees $m_1$ and $m_2$, respectively. As we will see, our results on the existence of equilibrium points can also be applied to homogeneous functions of the form
\begin{equation}\label{eqexemplosexpoente}
	\begin{aligned}
		&G_i(r,s)=\sum_{j=0}^{m_i}a_{j}  \, |r|^{\alpha_i} \,  |s|^{\beta_i}, \; \quad \mbox{or} \quad
		G_i(r,s)=\sum_{j=0}^{m_i}a_{j} \, sgn(r) \, |r|^{\alpha_i} \, sgn(s) \,  |s|^{\beta_i},
	\end{aligned}
\end{equation}
%\begin{equation}\label{eqexemplosexpoente}
%	\begin{aligned}
%		&G_i(r,s)=\sum_{j=0}^{m_i}a_{m_i-j,j} \, (sgn(r) \, r^{\alpha_i})^{m_i-j} \, (sgn(s) \, s^{\beta_i})^{j} , \; \; \mbox{or} \\
%		& G_i(r,s)= \sum_{j=0}^{m_i}a_{m_i-j,j} \, ( - r^{\alpha_i})^{m_i-j} \, (- s^{\beta_i})^{j} ,
%	\end{aligned}
%\end{equation}
%where  $m_i \in \na$ and $0 \leq \alpha_i,\beta_i<1$ for $i=1,2$, see Section \ref{section4}.
where $m_i \in \mathbb{N}$ and $0 \le \alpha_i, \beta_i < 1$ for $i = 1,2$, see Section~\ref{section4}.

%For $m_1=2$ and $m_2=3$ if we consider \eqref{eqexemplosexpoente} then $G$ will be given by the following equations
%\begin{equation*}
%	\begin{aligned}
%		&G(r,s)=(a_{2,0}(-r^(2 \alpha_1) Sign[r]^2 + Abs[s]^(2 \[Beta][j]) Sign[s]^2)) , \; \; \mbox{or} \\
%		& G_i(r,s)= \sum_{j=0}^{m_i}a_{m_i-j,j} \, ( - r^{\alpha_i})^{m_i-j} \, (- s^{\beta_i})^{j} ,
%	\end{aligned}
%\end{equation*}

%$\;$
%
%
%non-smooth Xj where our approach can be used are
%colocar  o exemplo polinomial que fiz e o com raiz do meu artigo de homogeneous,
%
%$\;$

%G pode ser uma comb linear de funcoes homegeneas onde tem caso polinomial e o caso nao suave
% fazer um caso que tem os termos homogeneos de grau 2 por exemplo igual fiz no papel

%The main goal of this paper is to analyze the bifurcations of low codimension classified in \cite{LarSeaTei2021}, a class of piecewise smooth vector fields with nonregular switching curve, under the smoothing process \eqref{eq2par1} in a neighborhood at the origin. The origin is where the switching curve loses its regularity. We would like to check the phenomena that may occur in the regularization chosen in this neighborhood. Furthermore, if we consider other neighborhoods then we already have cases described and cited in the literature. Using the smoothing process \eqref{eq2par1} we prove that the codimension of the bifurcations, in the world of smooth systems,  is not always preserved or is not always generic. As far as we know, there are no results in this direction.

The main goal of this paper is to analyze the bifurcations of low codimension, classified in \cite{LarSeaTei2021}, in a class of piecewise-smooth vector fields with a nonregular switching curve, under the smoothing process \eqref{eq2par1} in a neighborhood of the origin. The origin is the point where the switching curve loses its regularity. We aim to investigate the phenomena that may occur under the chosen regularization in this neighborhood. Furthermore, for other neighborhoods, the corresponding cases have already been described and cited in the literature. Using the smoothing process \eqref{eq2par1}, we show that the codimension of the bifurcations in the world of smooth system is not always preserved, nor is it always generic. To the best of our knowledge, there are no results in this direction.

%In this way, we can announce one of the main results of this article. Considering $det [ Z](p) = (X_1 \, Y_2-X_2 \, Y_1)(p)$ and $H_i(x_1,x_2)=\varphi_\varepsilon(x_{1}) \, \psi_ \eta(x_{2}) \, (X_i - Y_i)(x_{1},x_{2}) + (X_i + Y_i)(x_{1},x_{2})$  for $i=1,2$, then firstly, we take into account conditions about the existence of equilibrium points for the nonlinear double regularization.

In this way, we can state one of the main results of this article. Let $det [ Z](p) = (X_1 \, Y_2-X_2 \, Y_1)(p)$ and $H_i(x_1,x_2)=\varphi_\varepsilon(x_{1}) \, \psi_ \eta(x_{2}) \, (X_i - Y_i)(x_{1},x_{2}) + (X_i + Y_i)(x_{1},x_{2})$  for $i=1,2$.
First, we consider conditions for the existence of equilibrium points for the nonlinear double regularization.

\begin{teo}\label{teonep}
	%Consider a piecewise smooth vector field 
%		 Consider $Z \in \Omega$ \eqref{sistemacruz}, then its nonlinear double regularization $Z^R_{\varepsilon, \eta}$  has an equilibrium point if and only if there is $p_0=( p_{10},p_{20}) \in \mathcal{U}$ such that $det [Z](p_0) =0$ and $H_i(p_0)
%	= 0$ for some $i =1,2.$  If $\xi \neq 0$ then $G( \varphi_\varepsilon(p_{10}) , \psi_\eta(p_{20})) = (0,0)$ for every $\xi$ sufficiently small.
	Consider $Z \in \Omega$ as in \eqref{sistemacruz}. Then, its nonlinear double regularization $Z^R_{\varepsilon, \eta}$ has an equilibrium point if and only if there exists $p_0 = (p_{10}, p_{20}) \in \mathcal{U}$ such that $\det[Z](p_0) = 0$ and $H_i(p_0) = 0$ for some $i = 1,2$. Moreover, if $\xi \neq 0$, then 
$
	G\big(\varphi_\varepsilon(p_{10}), \psi_\eta(p_{20})\big) = (0,0)
$
	for every $\xi$ sufficiently small.
	
%	
%	
%	if $G( \varphi_\varepsilon(p_{10}) , \psi_\eta(p_{20})) \neq (0,0)$ and for any continuous curve $\gamma(t)=(\gamma_1(t),\gamma_2(t))$ passing through $p_0  \in \mathcal{U} $ in $t=t_0$ satisfies
%	$$H_1(\gamma_1(t),\gamma_2(t))\, G_2( \varphi_\varepsilon(\gamma_1(t)) , \psi_\eta(\gamma_2(t)) ) -H_2(\gamma_1(t),\gamma_2(t)) \, G_1( \varphi_\varepsilon(\gamma_1(t)), \psi_\eta(\gamma_2(t)) ) \neq 0, $$ 
%	for $t \neq t_0$ in a neigborhood of $p_0$	 then the equilibrium point exists for $\xi=0$, otherwise an equilibrium exists for every $\xi$ sufficiently small.
% Consider $Z \in \Omega$ \eqref{sistemacruz}, then its nonlinear double regularization $Z^R_{\varepsilon, \eta}$ \eqref{eq2par1} has an equilibrium point in $p_0=( p_{10},p_{20}) \in \mathcal{U}$  if and only $det Z(p_0) =0$ and $H_i(p_0)=\varphi_\varepsilon(p_{10}) \, \psi_ \eta(p_{20}) \, (X_i - Y_i)(p_{10},p_{20}) + (X_i + Y_i)(p_{10},p_{20})= 0$ for some $i =1,2.$ Moreover if $G( \varphi_\varepsilon(p_{10}) , \psi_\eta(p_{20})) \neq (0,0)$ and for any continuous curve $\gamma(t)=(\gamma_1(t),\gamma_2(t))$ passing through $p_0  \in \mathcal{U} $ satisfies
%$$H_1(\gamma_1(t),\gamma_2(t))\, G_2( \varphi_\varepsilon(\gamma_1(t)) , \psi_\eta(\gamma_2(t)) ) -H_2(\gamma_1(t),\gamma_2(t)) \, G_1( \varphi_\varepsilon(\gamma_1(t)), \psi_\eta(\gamma_2(t)) ) \neq 0, $$ 
%then the equilibrium point exists only for $\xi=0$, otherwise an equilibrium exists for every $\xi$ sufficiently small.
\end{teo}

%	As we will see, the proof of the last theorem also shows us that if $G_i \, G_j \neq 0$ and $H_i \, G_j - H_j \, G_i \neq 0$ then the equilibrium only exists for $\xi=0$, because $\xi=-H_i/G_i$ for $i=1,2$. Otherwise, there is an equilibrium for every $\xi$ in a suitable neighborhood.
	
	%Como veremos o ultimo teorema nos mostra que se temos um equilíbrio entao obrigatoriamente teremos $\xi=0$. 

%The last theorem shows us important conditions about the existence of equilibrium points for our nonlinear double regularization.
The last theorem provides important conditions for the existence of equilibrium points in our nonlinear double regularization.
%As we will see in its proof if $((G_i) ^2 +  (G_j )^2)( \varphi_\varepsilon(p_{01}), \psi_\eta(p_{01}) )\neq 0$ and $(H_i \cdot G_j - H_j \cdot G_i) ( x_1,x_2) \neq 0$ in some small neighborhood of $p_0$ then the equilibrium only exists for $\xi=0$, because the sufficiently small parameter must be given by $\xi=-(H_i/G_i)(x_1,x_2)$ for some $i=1,2$. Otherwise, there is an equilibrium for every $\xi$ in a suitable neighborhood.
	%The last theorem shows us important conditions for the existence of equilibrium points. As we will see, the proof of the last theorem also shows us that if $G_i \, G_j ( \varphi_\varepsilon(p_{01}), \psi_\eta(p_{02}))\neq 0$ and $H_i \, G_j - H_j \, G_i \neq 0$ then the equilibrium only exists for $\xi=0$, because $\xi=-H_i/G_i$ for $i=1,2$. Otherwise, there is an equilibrium for every $\xi$ in a suitable neighborhood.
%Thus, considering the nonlinear double regularization \eqref{eq2par1} 
%of the smooth piecewise system \eqref{sistemacruz}, 
%when the above conditions are not satisfied, 
%this implies the nonexistence of equilibrium points, 
%which, in a certain sense, leads to an analysis of structural stability. 
%With this in mind, let $G$ be a sufficiently smooth real function, 
%and consider the following classes for $Z=(X,Y)$:
Thus, considering the nonlinear double regularization \eqref{eq2par1} 
of the smooth piecewise system \eqref{sistemacruz}, 
when the above conditions are not satisfied, 
there are no equilibrium points. 
This, in a certain sense, motivates an analysis of structural stability. 
With this in mind, let $G$ be a sufficiently smooth real function, 
and consider the following classes for $Z = (X,Y)$:
\begin{itemize}
	\item[($A_0$)] $X_i \, Y_i(0) > 0$ for $i = 1, 2$;
	
	\item[($B_0$)] $X_i \, Y_i(0) < 0$ for $i = 1, 2$, and $\det Z(0) \neq 0$;
	
	\item[($C_0$)] $X_i \, Y_i(0) > 0$, $X_j \, Y_j(0) < 0$ for $i,j = 1,2$, $i \neq j$, and
	moreover, if $X_1 \, X_2(0) < 0$, then $\big(X_1 \, Y_2 / X_2 \, Y_1\big)^2(0) \neq 1$.
\end{itemize}
As proved in \cite{LarSeaTei2021}, the $\Sigma$-structurally stable set in $\Omega$ corresponds to $\Omega_0=$ ($A_0$)$\cup$($B_0$)$\cup$($C_0$) (see Section~\ref{section2})
We then obtain the following result.

%\begin{teo}	\label{teo1}
%	If $Z$ is $\Sigma$-structurally stable in $\Omega$, that is, $Z \in \Omega_0$ 
%	and the trasition functions satisfies $\varphi \, \psi(0) \in (-1,1)$ then there are real numbers $\varepsilon_0>0$ and $\eta_0>0$ such that
%	%its Sotomayor-Teixeira double regularization generalized  $Z_{\varepsilon, \eta}^R$$
%	its nonlinear double regularization $Z^R_{\varepsilon, \eta}$ \eqref{eq2par1} is structurally stable near to the origin for $(\varepsilon,\eta) \in (0,\varepsilon_0) \times (0,\eta_0)$ and any $\xi$ sufficiently small.
%	%and class $G$. 
%	In this case, the structural stability is preserved.
%\end{teo}

\begin{teo} \label{teo1}
	If $Z$ is $\Sigma$-structurally stable in $\Omega$, that is, if $Z \in \Omega_0$, 
	then there exist real numbers $\varepsilon_0 > 0$ and $\eta_0 > 0$ such that 
	its nonlinear double regularization $Z^R_{\varepsilon, \eta}$ \eqref{eq2par1}
	is structurally stable near the origin for 
	$(\varepsilon, \eta) \in (0, \varepsilon_0) \times (0, \eta_0)$, 
	and for any sufficiently small $\xi$, provided that one of the following conditions holds:
	\begin{itemize}
		\item[(i)] $Z$ satisfies $(B_0)$ or $(C_0)$;
		\item[(ii)] $Z$ satisfies $(A_0)$, and the transition functions satisfy $\varphi \, \psi(0) \in [-1,1]$.
	\end{itemize}
	In this case, the structural stability is preserved.
\end{teo}

%As a consequence, we can see that the Sotomayor-Teixeira transition functions guarantee the preservation of structural stability. Recall that Sotomayor-Teixeira transition functions have images given by $sgn(u)$ to $|u| \geq 1$ and have a positive derivative for $|u|<1$, that is, we have $\varphi \, \psi(0) \in (-1,1)$. As we will see in the proof of Theorem \ref{teo1}, the condition imposed for the transition functions is only important in one of the subclasses of structurally stable systems. In the other subclasses, we can remove this condition, which makes the result even more general for these subclasses. On the other hand, for more general transition functions, that is, allowing $\varphi \, \psi(0) \notin [-1,1]$, they lead to the possibility of the appearance of an equilibrium point at the origin for the reminiscent subclass.

As a consequence, we see that the Sotomayor-Teixeira transition functions guarantee the preservation of structural stability. Recall that the Sotomayor-Teixeira transition functions map $u$ to $\mathrm{sgn}(u)$ for $|u| \geq 1$ and have a positive derivative for $|u| < 1$, that is, $\varphi \, \psi(0) \in (-1,1)$. As we will see in the proof of Theorem \ref{teo1}, the condition imposed on the transition functions is only relevant for one of the subclasses of structurally stable systems. In the other subclasses, this condition can be removed, making the result even more general for these cases. On the other hand, for more general transition functions, that is, allowing $\varphi \, \psi(0) \notin [-1,1]$, there arises the possibility of the appearance of an equilibrium point at the origin in the corresponding subclass.
%If the equilibrium point is hyperbolic the result is still valid using the Hartman-Grobman Theorem, but if the equilibrium point is nonhyperbolic then the result is invalid and the codimension is not preserved. In section \ref{} we will show an example of a case in which the origin is a non-hyperbolic equilibrium point of the nonlinear double regularization of a structurally stable vector field Z ($Z \in \Omega_0$).
%Suppose that there is some smoothness on function $G$. 
%In that case,
%if the equilibrium point is hyperbolic then the result is still valid using the Hartman-Grobman Theorem. However, if the equilibrium point is nonhyperbolic, the result is invalid and the codimension is not preserved. In Section \ref{section4} we will show examples in which the origin is a non-hyperbolic equilibrium point of the nonlinear double regularization of a structurally stable vector field $Z \in \Omega_0$ and more details will be provided in section \ref{sectioneqorigem}
In that case, 
if the equilibrium point is hyperbolic, the result remains valid by the Hartman-Grobman Theorem. 
However, if the equilibrium point is nonhyperbolic, the result does not hold, and the codimension is not preserved. 
In Section \ref{section4}, we present examples in which the origin is a nonhyperbolic equilibrium point of the nonlinear double regularization of a structurally stable vector field $Z \in \Omega_0$. 
Further details are provided in Section \ref{sectioneqorigem}.

%Once we have studied the set $\Omega$ via a nonlinear double regularization process then now, our attention will be given in the same way to the piecewise smooth vector fields having a generic bifurcation of codimension one at the origin, $Z \in \Xi_1$, and consider the following classes for $Z=(X,Y)$:
%\begin{itemize}
%		\item[($A_1$)] 	$X_i(0)=0, \; X_j(0) \frac{\partial }{\partial x_j} X_i(0) \neq 0$ and $Y_i(0) \neq 0$ for $i=1,2$, or $Y_i(0)=0$, $Y_j(p) \frac{\partial }{\partial x_j} Y_i(0) \neq 0$ and $X_i(0) \neq 0$ for $i=1,2$;
%	
%	\item[($B_1$)] $X_i \, Y_i(0) <0$, $det Z(0)=0$ and $(det Z)_ {x_i}(0)$ $ \neq 0$ for $i=1,2$;
%	
%	\item[($C_1$)] $X_i \, Y_i(0) >0$, $X_j \, Y_j(0) <0$ for $i,j = 1,2, \; i \neq j$ and  $X_1 \, X_2(0)<0$. Moreover, if $X_1 \, X_2(0) < 0$ and $(X_1 \, Y_2/X_2 \, Y_1) (0)=- 1$ then $\beta_Z \neq 0$ and $\eta_Z \neq 0$ such that
%		\begin{equation*}
%		\phi_Z(x) = \alpha_Z ^2 x + (\alpha_Z + \alpha_Z ^2) \beta_Z x^2 + \eta_Z x^3 + \mathcal{O}(x^4), \; x \in \Sigma_2^-
%	\end{equation*}
%is the {\it first return map} associated of $Z$. 
%\end{itemize}

Once we have studied the set $\Omega$ via a nonlinear double regularization process, 
our attention will now be given, in a similar way, to the piecewise-smooth vector fields exhibiting a generic bifurcation of codimension one at the origin, $Z \in \Xi_1$. 
We consider the following classes for $Z = (X,Y)$:
\begin{itemize}
	\item[($A_1$)] $X_i(0) = 0$, $X_j  (X_i)_{x_j} (0) \neq 0$, and $Y_i(0) \neq 0$ for $i=1,2$, 
	or $Y_i(0) = 0$, $Y_j (Y_i)_{x_j} (0) \neq 0$, and $X_i(0) \neq 0$ for $i=1,2$;
	
	\item[($B_1$)] $X_i \, Y_i(0) < 0$, $\det Z(0) = 0$, and $(\det Z)_{x_i}(0) \neq 0$ for $i=1,2$;
	
	\item[($C_1$)] $X_i \, Y_i(0) > 0$, $X_j \, Y_j(0) < 0$ for $i,j = 1,2$, $i \neq j$, and $X_1 \, X_2(0) < 0$. 
	Moreover, if $X_1 \, X_2(0) < 0$ and $(X_1 \, Y_2 / X_2 \, Y_1)(0) = -1$, then there exist $\beta_Z \neq 0$ and $\eta_Z \neq 0$ such that
	\begin{equation*}
		\phi_Z(x) = \alpha_Z^2 x + (\alpha_Z + \alpha_Z^2) \beta_Z x^2 + \eta_Z x^3 + \mathcal{O}(x^4), \quad x \in \Sigma_2^-,
	\end{equation*}
	is the {\it first return map} associated with $Z$.
\end{itemize}
%then in \cite{LarSeaTei2021} the authors proved that the {\it firs return map} associated of $Z$ can be described as.
	%veand suppose the following expression for Poincaré map of fields $X$ and $Y$:
	%\begin{equation}\label{poincareX}
	%	\phi_X(x) = a_X x + b_X x^2 + c_X x^3 + \of(x^4) ,
	%\end{equation}
	%\begin{equation}\label{poincareY}
	%	\phi_Y(y) = a_Y y + b_Y y^2 + c_Y y^3 + \of(y^4).
	%\end{equation}
	%Therefore by the composition in the proper way of \eqref{poincareX} and \eqref{poincareY} we obtain Poincaré map of $\phi_Z$ given by
More details can be found in Section~\ref{section2}, based on \cite{LarSeaTei2021}, where the classification of piecewise vector fields in $\Omega \setminus \Omega_0$ with a generic codimension-one bifurcation at the origin is given by the set $\Xi_1=$($A_1$)$\cup$($B_1$)$\cup$($C_1$). Thus, we can state the following result.

%\begin{teo}	\label{teo2}
%	If $Z$ has a generic bifurcation of codimension one at origin, that is, $Z \in \Xi_1$ 
%then there exist real numbers $\varepsilon_0 > 0$ and $\eta_0 > 0$ such that 
%its nonlinear double regularization $Z^R_{\varepsilon, \eta}$ \eqref{eq2par1}
%is structurally stable near the origin for 
%$(\varepsilon, \eta) \in (0, \varepsilon_0) \times (0, \eta_0)$ 
%and for any sufficiently small $\xi$, provided that one of the following conditions holds:
%\begin{itemize}
%	\item[(i)] $Z$ satisfies $(A_1)$ or $(C_1)$;
%	\item[(ii)] $Z$ satisfies $(B_1)$ and the transition functions satisfy $\varphi \, \psi(0) \notin [-1,1]$.
%\end{itemize}
% In this case, the codimension one  is not preserved.
%\end{teo}
\begin{teo} \label{teo2}
	If $Z$ has a generic codimension-one bifurcation at the origin, that is, $Z \in \Xi_1$, 
	then there exist real numbers $\varepsilon_0 > 0$ and $\eta_0 > 0$ such that 
	its nonlinear double regularization $Z^R_{\varepsilon, \eta}$ \eqref{eq2par1}
	is structurally stable near the origin for 
	$(\varepsilon, \eta) \in (0, \varepsilon_0) \times (0, \eta_0)$ 
	and for any sufficiently small $\xi$, provided that one of the following conditions holds:
	\begin{itemize}
		\item[(i)] $Z$ satisfies $(A_1)$ or $(C_1)$;
		\item[(ii)] $Z$ satisfies $(B_1)$, and the transition functions satisfy $\varphi \, \psi(0) \notin [-1,1]$.
	\end{itemize}
	In this case, the codimension-one property is not preserved.
\end{teo}

Therefore, as mentioned previously, the proof of Theorem \ref{teo2} will show that the hypothesis regarding the transition function is important for one of the subclasses of vector fields exhibiting a generic codimension-one bifurcation at the origin. 
However, when considering transition functions such that $\varphi \, \psi(0) \in (-1,1)$, for instance, the Sotomayor-Teixeira transition functions, equilibrium points may exist in the nonlinear double regularization for the corresponding subclass. 
As before, if the equilibrium point is hyperbolic, the result remains valid by the Hartman-Grobman Theorem. 
Nevertheless, if a nonhyperbolic equilibrium point exists, the theorem does not hold. 
For example, we obtain the following results for $\alpha$ sufficiently small.

\begin{teo} \label{teo4}
%	Let $Z = (X, Y) \in \Omega$ satisfying condition $(B_1)$, and consider the families 
%	$\widetilde{Z}_{\alpha} = (\widetilde{X}_{\alpha}, \widetilde{Y})$ and 
%	$\widehat{Z}_{\beta} = (\widehat{X}_{\beta}, \widehat{Y})$ in $\Omega$, defined by
	Let $Z = (X, Y) \in \Omega$ satisfy condition $(B_1)$, and consider the families 
	\[
	\widetilde{Z}_{\alpha} = (\widetilde{X}_{\alpha}, \widetilde{Y}) \quad \text{and} \quad 
	\widehat{Z}_{\beta} = (\widehat{X}_{\beta}, \widehat{Y})
	\] 
	in $\Omega$, defined by
	\begin{subequations}\label{eq:pseudoequilibrio_duplo}
		\begin{align}
			\label{eq:pseudoequilibrio_duplo3}
			\widetilde{X}_{\alpha}(x_1, x_2) &=
			\begin{pmatrix}
				a - b\, c_2\, x_1 \\[3pt]
				b + a\, \alpha
			\end{pmatrix},
			&
			\widetilde{Y}(x_1, x_2) &=
			\begin{pmatrix}
				- a \\[3pt]
				- b + a\, c_1\, x_2
			\end{pmatrix}, \\[8pt]
			\label{eq:pseudoequilibrio_duplo4}
			\widehat{X}_{\beta}(x_1, x_2) &=
			\begin{pmatrix}
				a - b\, c_1\, x_1 - b\, c_2\, x_2 \\[3pt]
				b + a\, \beta
			\end{pmatrix},
			&
			\widehat{Y}(x_1, x_2) &=
			\begin{pmatrix}
				- a \\[3pt]
				- b + y^2
			\end{pmatrix}.
		\end{align}
	\end{subequations}
where $a = \operatorname{sgn}(X_1(0))$, $b = \operatorname{sgn}(X_2(0))$, and 
$c_i = \operatorname{sgn}\!\big((\det[Z])_{x_j}(0)\big)$, with $i,j \in \{1,2\}$ and $i \neq j$. If we take $\xi = 0$ and consider the Sotomayor--Teixeira transition functions in~\eqref{eq2par1} satisfying $\varphi(0) = 0$, then:
	\begin{itemize}
		\item[(i)] There exists a parameter value $\alpha_0 = \alpha_0(\varepsilon, \eta)$ for which the family~\eqref{eq:pseudoequilibrio_duplo3} undergoes a transcritical bifurcation at $(0, -\alpha / c_1)$, 
		for $(\varepsilon, \eta) \in (0, \varepsilon_0) \times (0, \eta_0)$. 
		In this case, the codimension is preserved, whereas the genericity condition is not.
		
		\item[(ii)] If $\psi(0) = 0$, the family~\eqref{eq:pseudoequilibrio_duplo4} undergoes a saddle-node bifurcation at $(0,0,0)$, 
		for $(\varepsilon, \eta) \in (0, \varepsilon_0) \times (0, \eta_0)$. 
		In this case, both the codimension and the genericity condition are preserved.
	\end{itemize}

\end{teo}

%
%\begin{teo}
%Consider the family $	\widetilde{Z_{\alpha}} =( 	\widetilde{X_{\alpha}},	\widetilde{Y}) \in \Omega$ such that 
%	\begin{equation}\label{eqpsuedoequilibrioduplo3}
%%	\widetilde{Z_{\alpha}} = \left\{ \begin{array}{l}
%		\widetilde{X_{\alpha}}(x,y) = \left( \begin{array}{c}
%			a - b c_2 x \\ 
%			b + a \alpha
%		\end{array}  \right) , \; x \, y >0, \qquad
%		\widetilde{Y}(x,y) = \left( \begin{array}{c}
%			-a \\ 
%			-b + a c_1 y
%		\end{array}  \right) , \; x \, y <0
%\end{equation}
%with  $a = sgn ( X_1(0) ), \; b=  sgn ( X_2(0) ), \; c_i = sgn (  (det [Z])_{x_j}(0) )$ such that $i,j = 1,2$ and $i \neq j$. If we condiser Sotomayor-Teixeira transition functions in (\ref{eq2par1}) such that $\varphi(0)=0$ and taking $\xi =0$ then there is a parameter value $\alpha_0=\alpha_0(\epsilon,\eta)$ which  family \eqref{eqpsuedoequilibrioduplo3} undergoes a transcritical bifurcation at $(0,-\alpha/c_1)$ for $(\varepsilon,\eta) \in (0,\varepsilon_0) \times (0,\eta_0)$. In this las case, the codimension is preserved but the genericity not.
%\end{teo}

Theorem~\ref{teo2} shows that the codimension is not preserved, whereas Theorem~\ref{teo4} demonstrates that there may also exist situations in which the codimension is preserved but the genericity property is lost.

In Section~\ref{section4}, we present an example (see Figure~\ref{figcurveq}) in which a curve of equilibria may appear during the smoothing process but does not persist for all values of the smoothing parameters. Moreover, as in the previous example, in the double nonlinear regularization of $Z \in \Omega$, there may also exist situations in which a fixed regularization parameter is considered, sufficiently small and dependent on the bifurcation parameter. That is, even though system~(\ref{eq2par1}) is sufficiently close to $Z$, pointwise convergence with respect to the regularization parameters does not occur. To establish this fact for the family~(\ref{eq2par2}), consider the Sotomayor--Teixeira transition functions in~(\ref{eq2par1}) such that $\varphi(0) \neq 0$. Then, there exists $p_0 \in (-\eta, \eta)$ satisfying $\psi(p_0) = 0$. Therefore, taking $\xi = 0$, we show at the end of Section~\ref{section4} that there exists a parameter value $\alpha_0 = -\eta \, c_1 \, p_0$ for which the family~(\ref{eq2par2}) undergoes a saddle-node bifurcation at $(0, -\alpha_0 / c_1)$ for a fixed $\eta = \eta_0$ and $\varepsilon \in (0, \varepsilon_0)$. In this case, both the codimension and the genericity are preserved.

%$\;$

%Caso nao façamos todos os casos tentar escrever algo como abaixo, de onde eu tirei de  `Generic bifurcations of low codimension of planar Filippov Systems`'',

%It is not the purpose of this paper to give a complete study of all the codimension-2 singulari- ties. After listing its whole set, we only study those which present rich dynamics in their unfolding. Moreover, to rigorously complete this study, one would need to see that any generic unfolding of the chosen singularities presents the same behavior as the studied normal form. Nevertheless, since we give the intrinsic conditions which define the codimension-2 singularities and we state the generic non-degeneracy conditions which their generic unfoldings need to satisfy, we expect that any generic unfolding satisfying these conditions presents the same behavior as the normal forms studied in this paper.

Section~\ref{sectioneqorigem} presents a detailed analysis of certain hyperbolicity conditions that may arise for the equilibrium points. The aim of this paper is not to provide a complete study of all possible equilibrium points; rather, our primary interest, as mentioned previously, is to examine whether the codimension is preserved following the regularization process. Accordingly, we present results in this direction. Since the expressions for the derivatives of the smoothed vector field are rather intricate, we limit our analysis to first-order derivatives, which are essential both for constructing the examples presented here and for guiding future investigations.

%Section~\ref{sectioneqorigem} presents a detailed analysis of certain hyperbolicity conditions for the equilibria. This paper does not aim to provide a comprehensive study of all possible singularities; rather, our primary focus, as noted above, is to determine whether the codimension is preserved under the regularization process. Accordingly, we report results in this direction. Given the complexity of the derivatives of the smoothed vector field, we restrict our attention to first-order derivatives, which are crucial both for constructing the examples presented and for guiding further investigations.

The paper is organized as follows. Section~\ref{section2} presents the Filippov conditions and reviews previous results related to the objects under consideration. Section~\ref{section3} contains the proofs of our main results. In Section~\ref{sectioneqorigem}, we perform a hyperbolicity analysis of the equilibrium point at the origin, and finally, Section~\ref{section4} is devoted to the presentation of illustrative examples.

\section{Filippov Systems, Local Structural Stability, and Codimension-One Bifurcations}\label{section2}

%Instead of focusing on the value of the vector field at individual points,  Filippov considered what the vector field looks like around each point, associating a set-valued map \cite{Fil1988} and differential inclusions \cite{Smirnov2002}. This approach is called the Filippov convention and, we call the piecewise smooth vector field \eqref{sistemacruz} the Filippov system when it respects the Filippov convention for its solutions. For this reason, in the literature of Filippov systems, one finds several definitions of orbit  \cite{BruPuSim2001,Fil1988,GuaSeaTeixeira2011}. The authors chose orbits adapted to their purposes because Filippov Systems have ambiguity in the chosen solutions reaching the switching curve.

Instead of focusing on the value of the vector field at individual points, Filippov considered the behavior of the vector field in a neighborhood of each point, associating a set-valued map \cite{Fil1988} and differential inclusions \cite{Smirnov2002}. This approach is known as the Filippov convention, and we refer to the piecewise-smooth vector field \eqref{sistemacruz} as a Filippov system when its solutions satisfy the Filippov convention. For this reason, in the literature on Filippov systems, one finds several definitions of orbits \cite{BruPuSim2001,Fil1988,GuaSeaTeixeira2011}. Different authors adopt notions of orbits suited to their specific purposes, since Filippov systems exhibit ambiguity in the choice of solutions when trajectories reach the switching curve.

%Therefore, firstly in this section, we are interested in presenting the definitions of the local flow of piecewise smooth vector fields \eqref{sistemacruz} in a neighborhood of the origin. These definitions are based on work \cite{LarSeaTei2021} which uses the ideas of the trajectories described in \cite{Fil1988, GuaSeaTeixeira2011}. After that, we are going to show definitions of $\Sigma$-structural stability and bifurcations of codimension one in the context of family $\Omega$, and then we will exhibit the classification carried out by Larrosa, Teixeira, and M-Seara in \cite{LarTereTei2021}, meeting our goals. In the end, we will do a little discussion about structural stability and bifurcations of codimension one for smooth vector fields, which will be useful in the discussion of regularization of the piecewise smooth vector fields \eqref{sistemacruz}.

Therefore, in this section we first present the definitions of the local flow of the piecewise-smooth vector field \eqref{sistemacruz} in a neighborhood of the origin. These definitions are based on the work of \cite{LarSeaTei2021}, which builds on the notion of trajectories introduced in \cite{Fil1988,GuaSeaTeixeira2011}. Next, we introduce the concepts of \(\Sigma\)-structural stability and codimension-one bifurcations in the context of the family \(\Omega\). We then present the classification obtained by Larrosa, Teixeira, and M-Seara in \cite{LarSeaTei2021}, which is directly relevant to our objectives. Finally, we provide a brief discussion on structural stability and codimension-one bifurcations for smooth vector fields. This discussion will be useful in the subsequent analysis of the regularization of the piecewise-smooth vector fields \eqref{sistemacruz}.

In $\mathcal{U} \setminus \Sigma$, a continuous differential equation defines the solutions; therefore, the main issue is the behavior of solutions in a neighborhood of $\Sigma$. Since $\Sigma$ can be regarded as the union of two regular manifolds, $\Sigma = \Sigma_1 \cup \Sigma_2$, we may decompose
$
\Sigma_1 = \overline{\Sigma_1^{+} \cup \Sigma_1^{-}},
$
where
\[
\Sigma_1^{+} = \{ p = (0,x_2) \in \Sigma_1 : x_2 > 0 \}
\quad \text{and} \quad
\Sigma_1^{-} = \{ p = (0,x_2) \in \Sigma_1 : x_2 < 0 \},
\]
which are themselves regular manifolds. Analogously, we can write
$
\Sigma_2 = \overline{\Sigma_2^{+} \cup \Sigma_2^{-}}.
$
In this setting, the Filippov convention \cite{Fil1988} removes the ambiguity in the definition of solutions on the switching set. Consequently, for $i = 1,2$, it is possible to identify the crossing, sliding, and escaping regions on $\Sigma_i^{\pm}$.
\begin{equation*}
	\begin{aligned}
		&\mbox{(CR) Crossing region:} \\
		&\phantom{aaaaaaa}\Sigma_i^c = \{ p \in \Sigma^\pm_i : Xf(p) \cdot Yf(p)>0\} = \{ p \in \Sigma^\pm_i : X_i \cdot Y_i (p) >0 \}, \\
		&\mbox{(SR) Sliding region:} \\
		&\phantom{aaaaaaa}\Sigma_{i}^d = \{ p \in \Sigma^\pm : Xf(p) <0, \; Yf(p)>0\} = \\
		&\phantom{aaaaaaa}\phantom{\Sigma_{i}^d}=\lbrace p \in \Sigma_i^+ : X_i(p)< 0, \; Y_i(p) > 0 \rbrace \cup \lbrace p \in \Sigma_i^- : X_i(p)> 0,\; Y_i(p) < 0 \rbrace,\\
		&\mbox{(ER) Escaping region:} \\
		&\phantom{aaaaaaa}\Sigma_{i}^e = \{ p \in \Sigma^\pm : Xf(p) >0, \; Yf(p)<0\} = \\
		&\phantom{aaaaaaa}\phantom{\Sigma_{i}^d}=\lbrace p \in \Sigma_i^+ : X_i(p)> 0, \; Y_i(p) < 0 \rbrace \cup \lbrace p \in \Sigma_i^- : X_i(p)< 0, \; Y_i(p) > 0 \rbrace,
			\end{aligned} 
	\end{equation*}
%where $Xf(p)= \langle \bigtriangledown f(p), X(p) \rangle$. As a result, each $\Sigma_i$ can be decomposed as the closure of the regions (CR), (SR), and (ER). Then the crossing, sliding, and escaping regions in $\Sigma$ are the union of the corresponding regions in $\Sigma_1$ and $\Sigma_2$. In the regions $\Sigma_i^{s,e}$, for $i=1,2$, there is a method to describe the system's behavior in the switching curve. We choose the convex combination of the two vector fields that is tangent to it. The Filippov sliding vector field $Z^{\Sigma_i}$ is defined by:
where $Xf(p)= \langle \nabla f(p), X(p) \rangle$. As a consequence, each $\Sigma_i$ can be decomposed as the closure of the crossing (CR), sliding (SR), and escaping (ER) regions. The crossing, sliding, and escaping regions in $\Sigma$ are then given by the union of the corresponding regions in $\Sigma_1$ and $\Sigma_2$. In the regions $\Sigma_i^{s,e}$, for $i=1,2$, there exists a well-defined procedure to describe the system's behavior along the switching curve. In this case, one selects the convex combination of the two vector fields that is tangent to the switching manifold. The corresponding Filippov sliding vector field $Z^{\Sigma_i}$ is defined by:
\begin{equation}\label{eqdocampodeslizanteouescape}
	Z_i^s(p) = (1 - \alpha(p) ) X(p) + \alpha(p) Y(p) =  \dfrac{1}{Y_i(p) - X_i(p)} \left[ Y_i \, X_j(p)- X_i \,Y_j(p) \right]_{\displaystyle\left|_{\Sigma_i} \right.},
\end{equation}
where $\alpha(p) = \left[ \dfrac{Xf}{(X - Y)f} \right](p)$, with $i,j = 1,2$ and $i \neq j$. Moreover, a point $p \in \Sigma_i^{s,e}$ is called a \emph{pseudo-equilibrium} if $Z_i^s(p) = 0$, and it is said to be \emph{hyperbolic} if $D Z_i^s(p) \neq 0$.
%\begin{defi}
%	Fix $i=1,2$. The point $p \in \Sigma_i^{s,e}$ is a pseudo-equilibrium if $Z_i^s(p)= 0 $ and it is hyperbolic is $(Z_i^s)^\prime (p) \neq 0$.
%\end{defi}

%\begin{obs}
%	Observe that, if we consider $det Z(p) = (X_1 \cdot Y_2-X_2 \cdot Y_1)(p)$, then for $p \in \Sigma_i^{s,e}$, $Z_i^s$ can be rewrite as 
%	\begin{equation}
%	Z_i^s(p) = \dfrac{(-1)^{i} \cdot det Z(p)}{Y_i(p) - X_i(p)}, \; \mbox{for} \; i=1,2.
%\end{equation}
%\end{obs}
\begin{figure}[!htb]
\begin{center}
		
		\begin{overpic}[scale=0.75,tics=5]{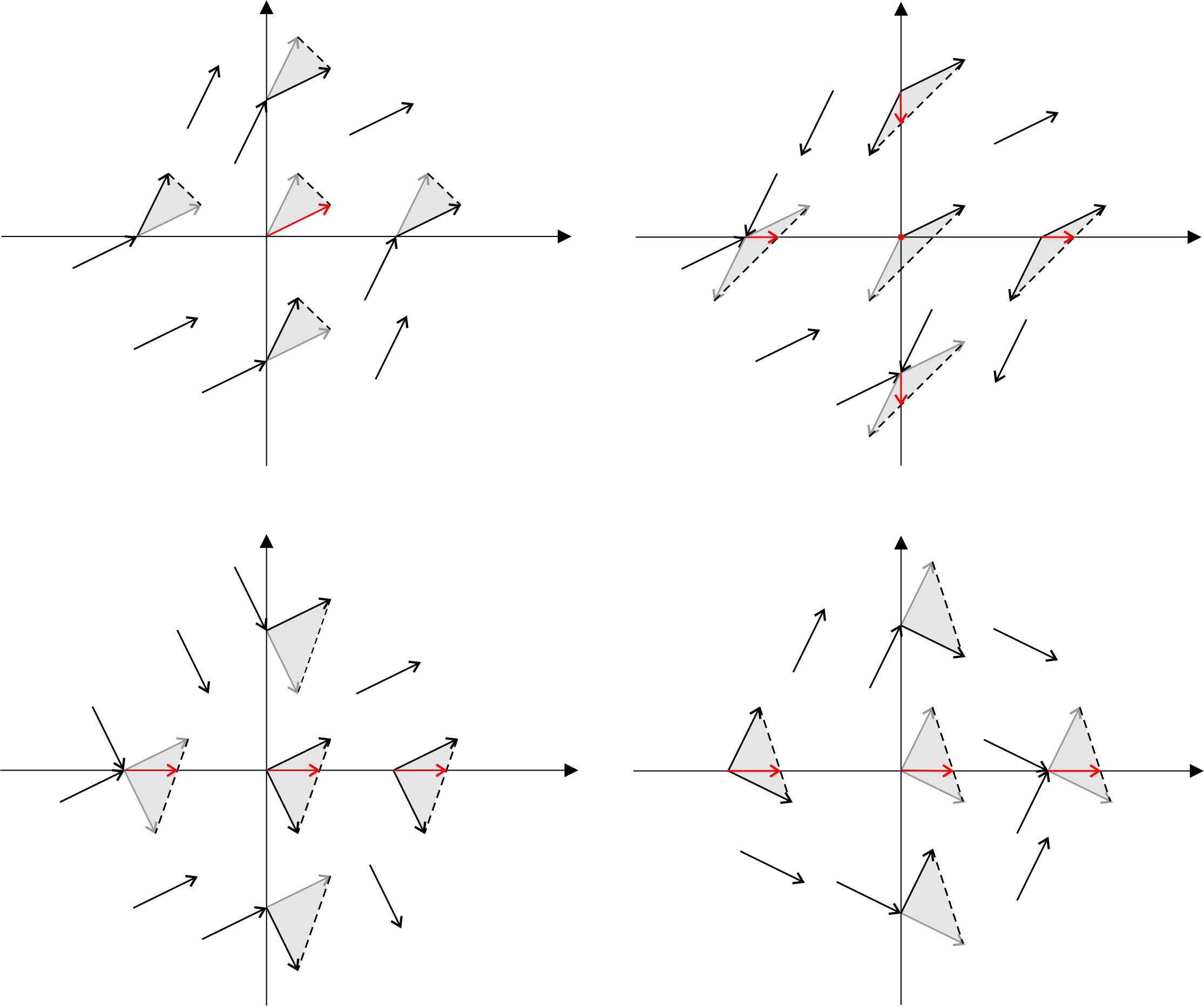}
			\vspace{0.2cm}
			
			%			\put(80,342){{(A)}}
			\put(80,168){{($A_0$)}}
			
			\put(295,168){{($B_0$)}}
			
			\put(80,-13){{($C_{01}$)}}
			
			\put(295,-13){{($C_{02}$)}}
			%campo1
			\put(6,300){{\parbox{0.55\linewidth}{$ 
						\left[ \begin{array}{cc}
							1 \\ 
							2  \\ 
						\end{array} \right]$}}}
			\put(145,300){{\parbox{0.5\linewidth}{$\left[ \begin{array}{cc}
							2 \\ 
							1  \\ 
						\end{array} \right]$}}}
			\put(6,205){{\parbox{0.5\linewidth}{$\left[ \begin{array}{cc}
							2 \\ 
							1  \\ 
						\end{array} \right]$}}}
			\put(145,205){{\parbox{0.5\linewidth}{$ \left[ \begin{array}{cc}
							1 \\ 
							2  \\ 
						\end{array} \right]$}}}
			%campo2
			\put(6,120){{\parbox{0.55\linewidth}{$ 
						\left[ \begin{array}{cc}
							1 \\ 
							-2  \\ 
						\end{array} \right]$}}}
			\put(145,120){{\parbox{0.5\linewidth}{$\left[ \begin{array}{cc}
							2 \\ 
							1  \\ 
						\end{array} \right]$}}}
			\put(6,25){{\parbox{0.5\linewidth}{$\left[ \begin{array}{cc}
							2 \\ 
							1  \\ 
						\end{array} \right]$}}}
			\put(145,25){{\parbox{0.5\linewidth}{$ \left[ \begin{array}{cc}
							1 \\ 
							-2  \\ 
						\end{array} \right]$}}}		
			%campo3
			\put(215,300){{\parbox{0.55\linewidth}{$ 
						\left[ \begin{array}{cc}
							-1 \\ 
							-2  \\ 
						\end{array} \right]$}}}
			\put(360,300){{\parbox{0.5\linewidth}{$\left[ \begin{array}{cc}
							2 \\ 
							1  \\ 
						\end{array} \right]$}}}
			\put(360,205){{\parbox{0.5\linewidth}{$\left[ \begin{array}{cc}
							-1 \\ 
							-2  \\ 
						\end{array} \right]$}}}
			\put(215,205){{\parbox{0.5\linewidth}{$ \left[ \begin{array}{cc}
							2 \\ 
							1  \\ 
						\end{array} \right]$}}}
			%campo4
			\put(215,120){{\parbox{0.55\linewidth}{$ 
						\left[ \begin{array}{cc}
							1 \\ 
							2  \\ 
						\end{array} \right]$}}}
			\put(360,120){{\parbox{0.5\linewidth}{$\left[ \begin{array}{cc}
							2 \\ 
							-1  \\ 
						\end{array} \right]$}}}
			\put(360,25){{\parbox{0.5\linewidth}{$\left[ \begin{array}{cc}
							1 \\ 
							2  \\ 
						\end{array} \right]$}}}
			\put(215,25){{\parbox{0.5\linewidth}{$ \left[ \begin{array}{cc}
							2 \\ 
							-1  \\ 
						\end{array} \right]$}}}	
			%			\put(-18,-10){{\parbox{0.5\linewidth}{Fonte:  Elaborado pelo autor.}}}
			
		\end{overpic}

	\end{center}
	
	$\;$
	\caption{Examples of Filippov conventions and orbits at origin.}\label{exemplos_inclusao_filippov}

\end{figure}
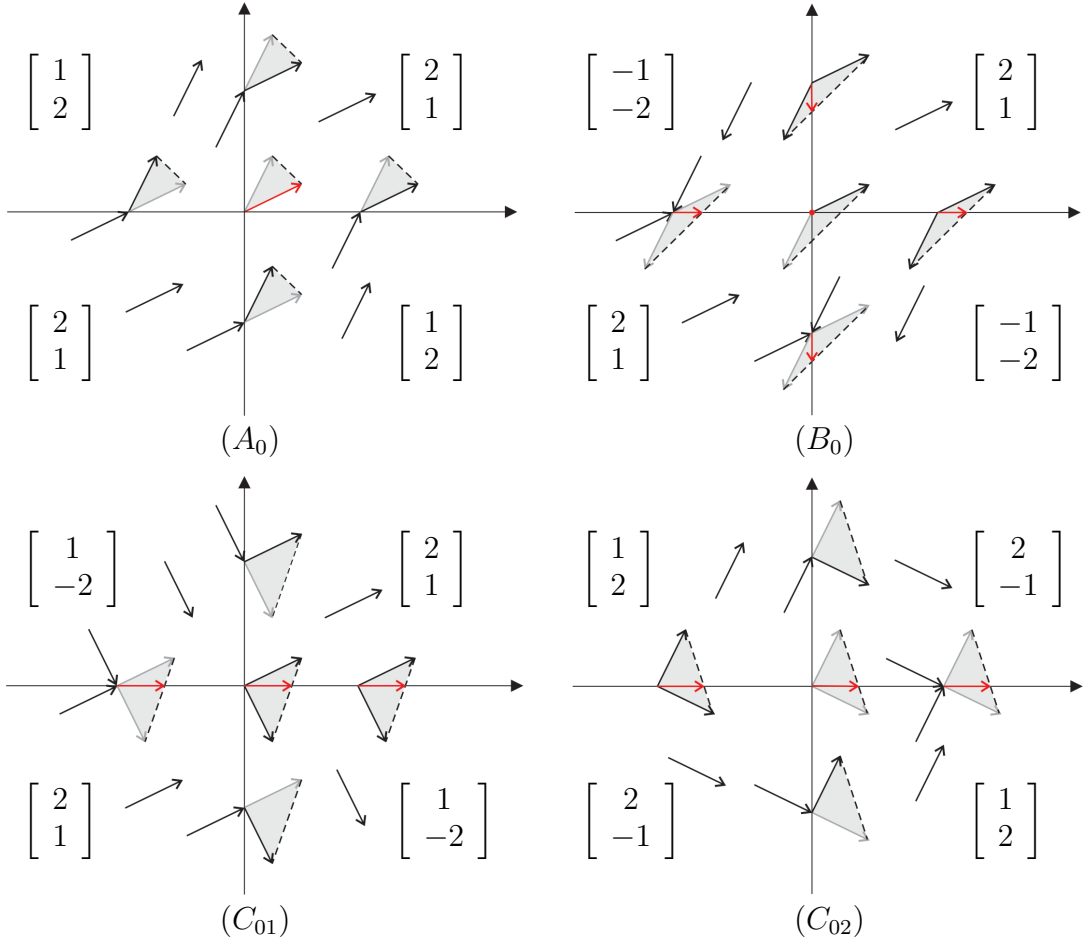

%A point $p \in \Sigma \setminus \{0\}$ is a tangency point if $Xf(p)=0$ or $Yf(p)=0$. It is easy to see that the tangency points occur at the borders of sewing, sliding, and escaping regions. In addition, the tangency point is a fold point when we have a quadratic tangency point, that is, $Xf(p)=0$ and $X(XF)(p)= \langle \bigtriangledown (Xf)(p), X(p) \rangle \neq 0$ or $Yf(p)=0$ and $Y(YF)(p)= \langle \bigtriangledown (Yf)(p), Y(p) \rangle \neq 0$. We are interested in low codimension singularities, so it is sufficient to consider only the fold points. To extend the concept of the fold point to the origin let us consider $\Sigma$ as the union of two regular manifolds and look at the tangency point in each of them considered separately.

A point $p \in \Sigma \setminus \{0\}$ is called a tangency point if $Xf(p)=0$ or $Yf(p)=0$. It is easy to see that tangency points occur at the boundaries between the crossing, sliding, and escaping regions. Moreover, a tangency point is called a fold point when the tangency is quadratic; that is, when either
\[
Xf(p)=0 \quad \text{and} \quad X(Xf)(p)= \langle \nabla (Xf)(p), X(p) \rangle \neq 0,
\]
or
\[
Yf(p)=0 \quad \text{and} \quad Y(Yf)(p)= \langle \nabla (Yf)(p), Y(p) \rangle \neq 0.
\]
Since we are interested in low-codimension singularities, it is sufficient to consider only fold points. To extend the concept of a fold point to the origin, we regard $\Sigma$ as the union of two regular manifolds and analyze the tangency points on each of them separately.

%
%\begin{defi}
%	A point $p \in \Sigma_i$ is a fold point of $X$ in $\Sigma_i$ if $X_i(p)=0$ and $X_j(p) \cdot \frac{\partial}{\partial x_j}X_i(p) \neq 0$ for $i,j=1,2$ and $i\neq j$. Moreover, $p \in \Sigma$ is a regular-fold of $X$ in $\Sigma_i$ if it is a fold point for $X$ in $\Sigma_i$ and $Y$ is transverse to $\Sigma$ at the origin, that is, $Y_k(0) \neq 0$ for $k=1,2$. Analogously, we define a fold point and a regular fold for $Y$.
%\end{defi}

\begin{defi}\label{defi2}
%	The point $p \in \Sigma_i $ is a fold point of $X$ \rm{(}$Y$\rm{)} in $\Sigma_i$ if $X_i(p)=0$ ($Y_i(p)=0$) and $X_j \, \frac{\partial }{\partial x_j} X_i(p) \neq 0$ 	$\left(Y_j \,\frac{\partial }{\partial x_j}Y_i(p) \neq 0 \right)$ for $i,j=1,2$ and $i \neq j$. Moreover, $p \in \Sigma$ is a {\it regular-fold} of $X$ \rm{(}$Y$\rm{)} in $\Sigma_i$ if it is a fold point for $X$ \rm{(}$Y$\rm{)} in $\Sigma_i$ and $Y$ \rm{(}$X$\rm{)} is transverse to $\Sigma$ at the origin, that is, $Y_k(0) \neq 0$ \rm{(}$X_k(0) \neq 0$\rm{)} for $k=1,2$. 
	A point $p \in \Sigma_i$ is called a \emph{fold point} of $X$ $(\text{resp. }  Y)$ in $\Sigma_i$ if 
	\[
	X_i(p) = 0 \quad (\text{resp. } Y_i(p) = 0) \quad \text{and} \quad 
	X_j \, (X_i)_{ x_j}(p) \neq 0 \quad (\text{resp. } 
	Y_j \, ( Y_i)_{x_j}(p) \neq 0),
	\] 
	for $i,j = 1,2$ with $i \neq j$. Moreover, a point $p \in \Sigma$ is called a \emph{regular-fold} of $X$ (resp. $Y$) in $\Sigma_i$ if it is a fold point for $X$ (resp. $Y$) in $\Sigma_i$ and the vector field $Y$ (resp. $X$) is transverse to $\Sigma$ at the origin; that is,
$
	Y_k(0) \neq 0$ $(\text{resp. }   X_k(0) \neq 0)$, $k = 1,2.
$

\end{defi}

\begin{figure}[!htb]

	\begin{center}
		
		\begin{overpic}[scale=0.75,tics=5]{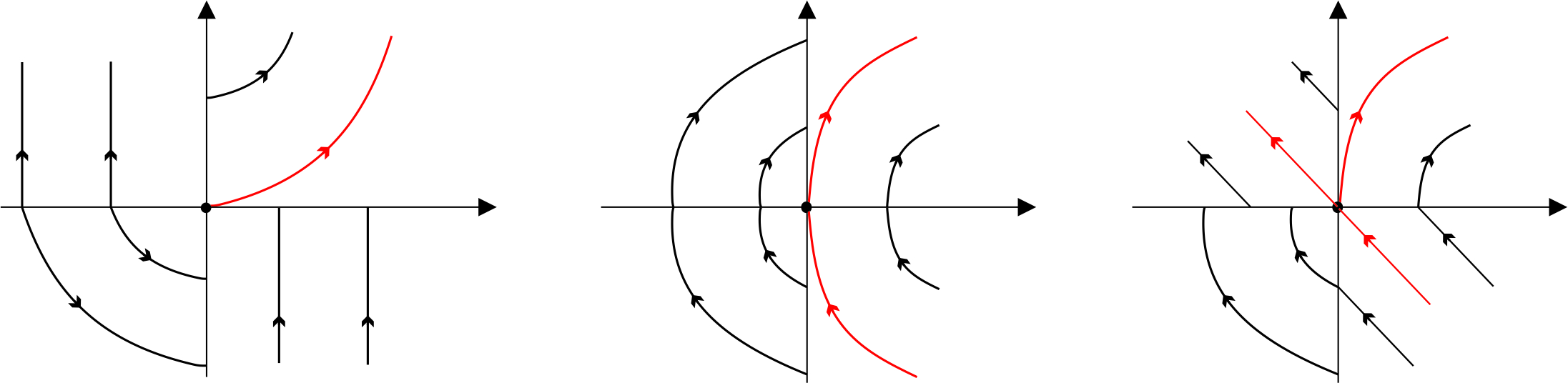}
			\vspace{0.2cm}
			
			%			\put(80,342){{(A)}}

			\put(45,-15){{(a)}}
			
			\put(198,-15){{(b)}}
			
			\put(333,-15){{(c)}}

			\put(45,99){{\parbox{0.55\linewidth}{$ 
						\Sigma_1$}}}
		
		\put(198,99){{\parbox{0.55\linewidth}{$ 
					\Sigma_1$}}}
		
		\put(333,99){{\parbox{0.55\linewidth}{$ 
					\Sigma_1$}}}
		
		\put(130,39){{\parbox{0.55\linewidth}{$ 
					\Sigma_2$}}}
				
			\put(265,39){{\parbox{0.55\linewidth}{$ 
						\Sigma_2$}}}
					
			\put(400,39){{\parbox{0.55\linewidth}{$ 
						\Sigma_2$}}}
							
					%campoXeY_1
						\put(105,80){{\parbox{0.55\linewidth}{$ 
								X$}}}
					
					\put(245,80){{\parbox{0.55\linewidth}{$ 
								X$}}}
							
								\put(105,5){{\parbox{0.55\linewidth}{$ 
										Y$}}}
							
							\put(245,5){{\parbox{0.55\linewidth}{$ 
										Y$}}}
				%campoXeY_2
			\put(05,80){{\parbox{0.55\linewidth}{$ 
						Y$}}}
			
			\put(160,80){{\parbox{0.55\linewidth}{$ 
						Y$}}}
			
			\put(05,5){{\parbox{0.55\linewidth}{$ 
						X$}}}
			
			\put(160,5){{\parbox{0.55\linewidth}{$ 
					X	$}}}

				%campoXeY_3
			\put(380,80){{\parbox{0.55\linewidth}{$ 
						X$}}}
			
			\put(295,80){{\parbox{0.55\linewidth}{$ 
						Y$}}}
			
			\put(295,5){{\parbox{0.55\linewidth}{$ 
						X$}}}
			
			\put(380,5){{\parbox{0.55\linewidth}{$ 
							Y$}}}
		\end{overpic}

	\end{center}
	
	$\;$
		\caption{The origin is a fold point of both $X$ and $Y$ in cases (a) and (b). In case (c), the origin is a regular-fold of $X$.
%			we have $X_2(0)=0$ with $X_1(p) \frac{\partial }{\partial x_1} X_2(p) > 0$ and  in (b) we have $X_1(0)=0$ with $X_2(p) \frac{\partial }{\partial x_2} X_1(p) > 0$. fazer os 4 casos lado a lado
}\label{dobranaorigem}
	
\end{figure}
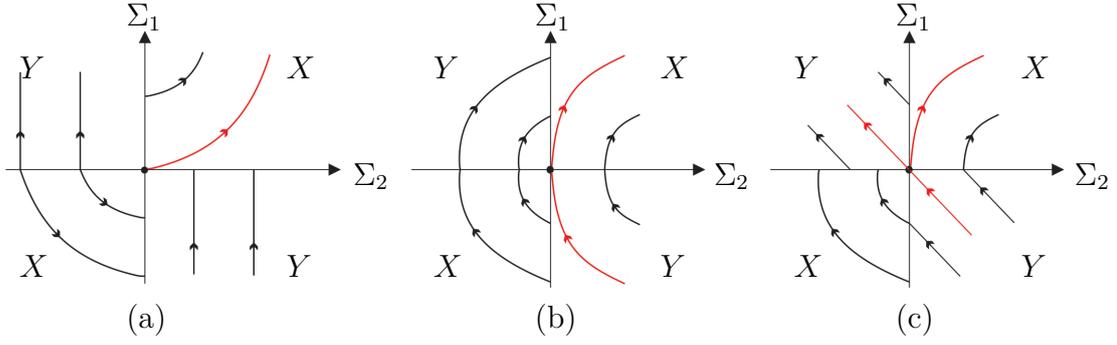

%The next step consists in to define the trajectories of $Z$ through a point of $\mathcal{U}$, following \cite{GuaSeaTeixeira2011} and respecting the Fillipov conditions \eqref{incdif1}. We denote by $\varphi_X(t;p)$ ($\varphi_X(t;p)$) the flow of a regular vector field $X$ ($Y$). Consider $p \in \mathcal{U}^\pm $, if the curves $\{ \varphi_{X}(t;p) : t \in \re \}$ or $\{ \varphi_{Y}(t;p) : t \in \re \}$ are transversal to $\Sigma^{s,e}$ then we let us assume that the trajectories of $X$ and $Y$ through $p$ do not reach $\Sigma^s \cup \Sigma^e$ in finite time, they are relatively open. Now
%we are going to define the trajectories of a point in $ \mathcal{U} \setminus \{ 0 \}$.

%Now denote by $\varphi_X(t;p)$ ($\varphi_Y(t;p)$) the flow of vector field $X$ ($Y$).  Consider $p \in \mathcal{U}^\pm $, if the curves $\{ \varphi_{X}(t;p) : t \in \re \}$ or $\{ \varphi_{Y}(t;p) : t \in \re \}$ are transversal to $\Sigma^{s,e}$ then we let us assume that the trajectories of $X$ and $Y$ through $p$ do not reach $\Sigma^s \cup \Sigma^e$ in finite time. Based on \cite{GuaSeaTeixeira2011} then the next step consists in to define the trajectories of $Z$ through a point of $ \mathcal{U} \setminus \{ 0 \}$. Denoting by $\partial$ the boundary of the set  and considering $i=1,2$ follows that:
Let $\varphi_X(t;p)$ (resp. $\varphi_Y(t;p)$) denote the flow of the vector field $X$ (resp. $Y$). Consider a point $p \in \mathcal{U}^\pm$. If the curves
\[
\{ \varphi_X(t;p) : t \in \mathbb{R} \} \quad \text{or} \quad 
\{ \varphi_Y(t;p) : t \in \mathbb{R} \}
\]
are transversal to $\Sigma^{s,e}$, we assume that the trajectories of $X$ and $Y$ through $p$ do not reach $\Sigma^s \cup \Sigma^e$ in finite time. Following \cite{GuaSeaTeixeira2011}, the next step is to define the trajectories of the piecewise smooth vector field $Z$ through points in $\mathcal{U} \setminus \{0\}$. Denoting by $\partial$ the boundary of a set and considering $i = 1,2$, we obtain:
%\begin{defi} 
%	Denoting by $\partial$ the boundary of the set, let us consider $p \in \mathcal{U} \setminus \{ 0 \} $ then its trajectory $\varphi_Z(t,p)$ is given by:
	\begin{itemize}
		
%		\item[(i)] If $p \in \mathcal{U}^+ \cup \mathcal{U}^-$ then its trajectory is given by the trajectory of $X$ or $Y$, respectively;
%		
%		\item[(ii)] If $p \in \Sigma^c = \Sigma_1^c \cup \Sigma_2^c $ then its trajectory is the concatenation of its respective trajectories in $\mathcal{U}^+$ and $\mathcal{U}^-$;
%		
%		\item[(iii)] If $p \in \Sigma_{i}^s \cup \Sigma_{i}^e  $ then its trajectory is given by $\varphi_{Z_i^s}(t;p)$ with $Z_i^s$ described in \eqref{eqdocampodeslizanteouescape};
%		
%		\item[(iv)] If $p \in \partial\Sigma_{i}^e \cup \partial \Sigma_{i}^s \cup \partial\Sigma_{i}^c$ 
%		such that $\displaystyle\lim_{q \rightarrow p^+} \varphi_Z(t;q) = \displaystyle\lim_{q \rightarrow p^-} \varphi_Z(t;q)$, then its trajectory is given by $\varphi_Z(t;p) = \displaystyle\lim_{q \rightarrow p} \varphi_Z(t;q) $. These points are called { \it regular tangency points};
%		
%		\item[(v)] If $p \in \partial\Sigma_{i}^e \cup \partial\Sigma_{i}^s \cup \partial\Sigma_{i}^c$ 
%		and does not satisfy the last condition, then $\varphi_Z(t;p) = \lbrace p \rbrace $. These points are called {\it  singular tangency points}.
		\item[(i)] If $p \in \mathcal{U}^+ \cup \mathcal{U}^-$, its trajectory is given by the trajectory of $X$ or $Y$, respectively;
		
		\item[(ii)] If $p \in \Sigma^c = \Sigma_1^c \cup \Sigma_2^c$, its trajectory is given by the concatenation of the corresponding trajectories in $\mathcal{U}^+$ and $\mathcal{U}^-$;
		
		\item[(iii)] If $p \in \Sigma_i^s \cup \Sigma_i^e$, its trajectory is given by $\varphi_{Z_i^s}(t;p)$, where $Z_i^s$ is defined in \eqref{eqdocampodeslizanteouescape};
		
		\item[(iv)] If $p \in \partial \Sigma_i^e \cup \partial \Sigma_i^s \cup \partial \Sigma_i^c$ and satisfies
		\[
		\lim_{q \to p^+} \varphi_Z(t;q) = \lim_{q \to p^-} \varphi_Z(t;q),
		\]
		then its trajectory is defined by
		\[
		\varphi_Z(t;p) = \lim_{q \to p} \varphi_Z(t;q).
		\]
		Such points are called \emph{regular tangency points};
		
		\item[(v)] If $p \in \partial \Sigma_i^e \cup \partial \Sigma_i^s \cup \partial \Sigma_i^c$ and does not satisfy the condition in item (iv), then
		$
		\varphi_Z(t;p) = \{p\}.
		$ Such points are called \emph{singular tangency points}.
		
		\vspace{0.25cm}
	\end{itemize}
%\end{defi}
% in this case, we follow  definitions indicated in \cite{GuaSeaTeixeira2011} for the case when $\Sigma \setminus \{0 \}$ is a regular manifold. 
\begin{figure}[!htb]\label{ciclos}
	
	\begin{center}
		
			\begin{overpic}[scale=0.55,tics=5]{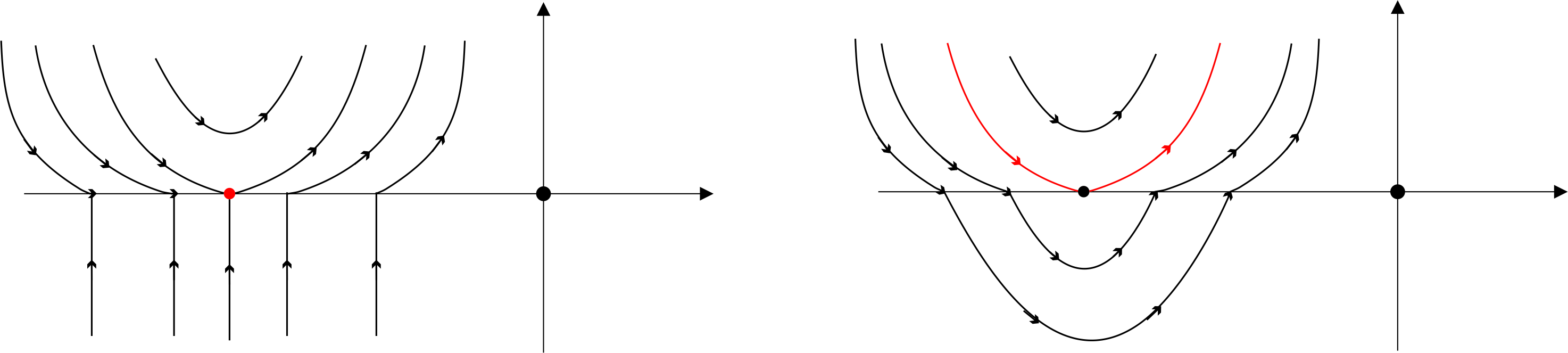}
			\vspace{0.2cm}
			
			%			\put(80,342){{(A)}}

			\put(49,83){{(a)}}
			
			\put(270,83){{(b)}}

			\put(130,90){{\parbox{0.55\linewidth}{$ 
						\Sigma_1$}}}
			
			\put(343,90){{\parbox{0.55\linewidth}{$ 
						\Sigma_1$}}}

			\put(180,36){{\parbox{0.55\linewidth}{$ 
						\Sigma_2$}}}
			
			\put(392,36){{\parbox{0.55\linewidth}{$ 
						\Sigma_2$}}}

		\end{overpic}

		\caption{%(a) Example of regular tangency trajectory; (b) Example of singular tangency trajectory .
		(a) Trajectory exhibiting a regular tangency; (b) trajectory exhibiting a singular tangency.
		}
		
	\end{center}
	
\end{figure}
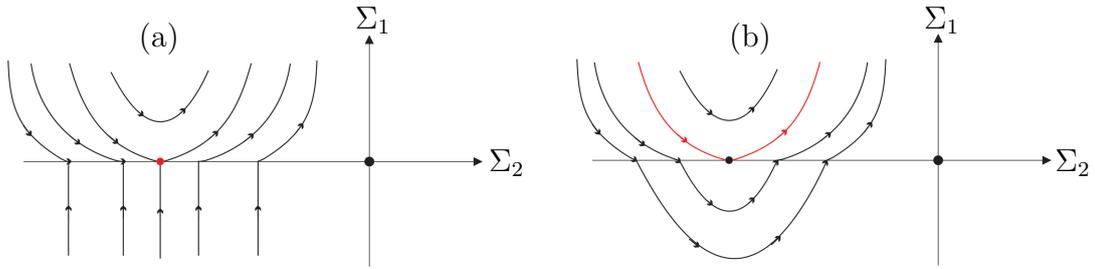
On the other hand, when the vector fields are transversal at the origin for
$i,j = 1,2$, the trajectory $\varphi_Z(t,0)$ is defined, following
\cite{LarSeaTei2021}, as:
%\begin{defi}\label{defitraj}
%	Consider $\lbrace 0 \rbrace \in \Sigma$, then  $\varphi_Z(t,0)$ will be given by:
%	
	\begin{itemize}
		
%		\item[($i_0$)] If $\lbrace 0 \rbrace \in \Sigma = \overline{\Sigma^c}=   \overline{\Sigma_1^c \cap \Sigma_2^c} $ then there is only trajectory of $X$ or $Y$  through the origin and we define  $\varphi_Z(t;0)$ as being this trajectory, see Figure \ref{exemplos_inclusao_filippov} ($A_0$);
%		
%		\item[($ii_0$)]  If $\lbrace 0 \rbrace \in  \overline{(\Sigma_{i}^e \cup \Sigma_{i}^s)} \cap \overline{\Sigma_{j}^c}$ with $i \neq j$, then   $\varphi_Z(t;0) = \varphi_{Z_i^s}(t;0)$ such that $Z_i^s$ is described in \eqref{eqdocampodeslizanteouescape}, see Figure \ref{exemplos_inclusao_filippov} ($C_{0i}$). In this case, the origin is a regular point of the Filippov sliding vector field \eqref{eqdocampodeslizanteouescape};
%		
%		\item[($iii_0$)] If $\lbrace 0 \rbrace \in \cap_{i=1} ^2  \overline{(\Sigma_{i}^e \cup \Sigma_{i}^s)}$  then $\varphi_Z(t;0) = \lbrace 0 \rbrace $, see Figure \ref{exemplos_inclusao_filippov} ($B_0$).

\item[($i_0$)]
If $\{0\} \in \Sigma = \overline{\Sigma^c}
= \overline{\Sigma_1^c \cap \Sigma_2^c}$,
then there is a unique trajectory of either $X$ or $Y$ passing through the
origin. In this case, $\varphi_Z(t;0)$ is defined as this trajectory; see
Figure~\ref{exemplos_inclusao_filippov} ($A_0$).

\item[($ii_0$)]
If $\{0\} \in \overline{\Sigma_i^e \cup \Sigma_i^s} \cap \overline{\Sigma_j^c}$
with $i \neq j$, then
$
\varphi_Z(t;0) = \varphi_{Z_i^s}(t;0),
$
where $Z_i^s$ is defined in~\eqref{eqdocampodeslizanteouescape}; see
Figure~\ref{exemplos_inclusao_filippov} ($C_{0i}$).
In this case, the origin is a regular point of the Filippov sliding vector field
defined in~\eqref{eqdocampodeslizanteouescape}.

\item[($iii_0$)]
If $\{0\} \in \bigcap_{i=1}^{2} \overline{\Sigma_i^e \cup \Sigma_i^s}$,
then
$
\varphi_Z(t;0) = \{0\};
$
see Figure~\ref{exemplos_inclusao_filippov} ($B_0$).

				\vspace{0.25cm}
				
	\end{itemize}

A {\it regular orbit} is an orbit
\[
\gamma = \{\varphi_Z(t,p) : t \in \mathbb{R}\}
\]
that is entirely contained in the region
$\mathcal{U}^+ \cup \mathcal{U}^- \cup \Sigma^c$. A {\it pseudo-cycle} (also called a {\it closed contour}) is defined as the closure
of a finite collection of regular orbits
$\gamma_1, \ldots, \gamma_n$ such that the endpoints (arrival and departure points)
of each orbit $\gamma_i$ coincide with endpoints of $\gamma_{i+1}$, including the
connection between $\gamma_n$ and $\gamma_1$.
These connections form a curve homeomorphic to
$\mathbb{S}^1$, in such a way that at least one point corresponds to either two
departure points or two arrival points. 

Pseudo-cycles are not genuine closed orbits; however, they are preserved under the
equivalence relation considered by Larrosa, M-Seara, and Teixeira
in~\cite{LarSeaTei2021}.
Moreover, pseudo-cycles constitute the only type of recurrent behavior containing
the origin that appears in codimension-one bifurcation unfoldings. For other notions of periodic behavior, such as {\it regular periodic orbits} and
{\it cycles}, see~\cite{GuaSeaTeixeira2011,LarSeaTei2021}.
Note that the existence of any type of periodic orbit requires transitions between
different quadrants or regions.
Therefore, it is necessary to introduce a precise definition of such transitions.

\begin{defi}
%	We say that  $Z \in \Omega$ is transient in $\mathcal{U}^{\pm}$ if for any $p \in \overline{\mathcal{U}^{\pm}}$ there is $t_i(p) \in \re$ such that $\varphi_Z(t_i(p);p) \in \Sigma_i$ for $i=1,2$ and $\varphi_Z(t;p) \in \mathcal{U}^{\pm}$ for all $t \in [min \lbrace t_1(p),  t_2(p) \rbrace , max \lbrace t_1(p), t_2(p) \rbrace]$.   Then $Z$ transient if it is transient in $\mathcal{U}^+$ and $\mathcal{U}^-$.
	
We say that $Z \in \Omega$ is {\it transient} in $\mathcal{U}^{\pm}$
if, for any $p \in \overline{\mathcal{U}^{\pm}}$, there exist times
$t_i(p) \in \mathbb{R}$, $i=1,2$, such that
\[
\varphi_Z(t_i(p);p) \in \Sigma_i,
\]
and
\[
\varphi_Z(t;p) \in \mathcal{U}^{\pm}
\quad \text{for all} \quad
t \in \big[ \min\{t_1(p),t_2(p)\}, \max\{t_1(p),t_2(p)\} \big].
\]
Moreover, $Z$ is {\it transient} if it is transient in both
$\mathcal{U}^+$ and $\mathcal{U}^-$.

\end{defi}
\begin{figure}[!htb]

\begin{center}
	
	$\;$

	\begin{overpic}[scale=0.9,tics=5]{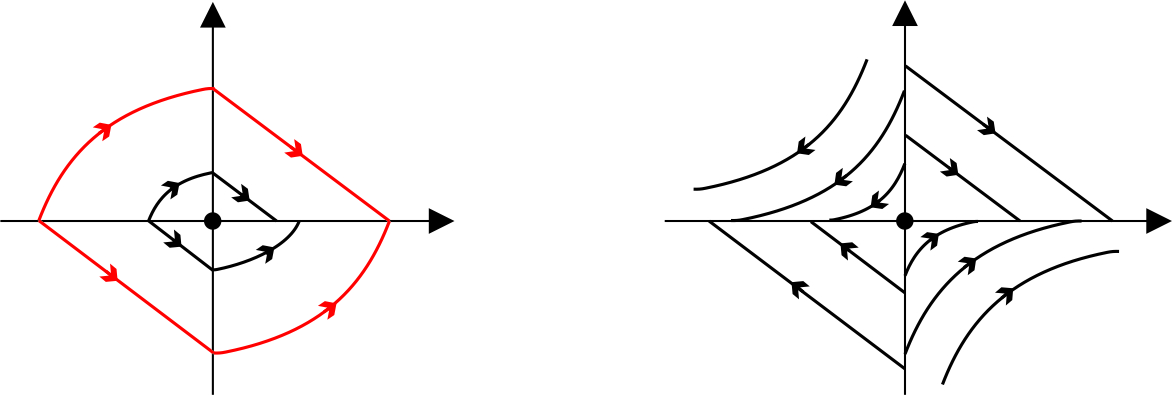}
		\vspace{0.2cm}
		
		%			\put(80,342){{(A)}}

		\put(40,-17){{(a)}}
		
		\put(190,-17){{(b)}}

		\put(41,88){{\parbox{0.55\linewidth}{$ 
					\Sigma_1$}}}
		
		\put(191,88){{\parbox{0.55\linewidth}{$ 
					\Sigma_1$}}}

		\put(100,33){{\parbox{0.55\linewidth}{$ 
					\Sigma_2$}}}
		
		\put(255,33){{\parbox{0.55\linewidth}{$ 
					\Sigma_2$}}}

			\put(160,65){{\parbox{0.55\linewidth}{$ 
					Y$}}}
				
			\put(230,65){{\parbox{0.55\linewidth}{$ 
						X$}}}
					
				\put(160,5){{\parbox{0.55\linewidth}{$ 
						X$}}}
			
			\put(230,5){{\parbox{0.55\linewidth}{$ 
						Y$}}}	
		
	\end{overpic}

	$\;$

		\caption{(a) Illustration of a {\it pseudo-cycle}; (b) Illustration of a {\it transient} vector field.
		 }\label{pseudo_ciclo_e_transiente}
	
\end{center}

\end{figure}
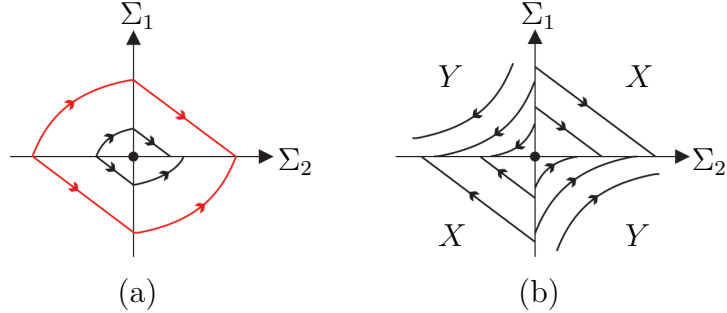

To analyze the stability of a certain type of periodic orbit, we need to define a type of Poincaré map, called the {\it first return map}. Let $Z \in \Omega$ be transient. Then, for each $p \in \Sigma_1$, there exists a unique $t_X(p) \in \mathbb{R}$ such that
\[
\varphi_X(p;t) \in \overline{\mathcal{U}^+} \quad \text{for all} \quad 
t \in [ \min\{0, t_X(p)\}, \max\{0, t_X(p)\} ],
\]
and satisfying $\varphi_X(t_X(p),p) \in \Sigma_2$. We then define
\[
\phi_X : p \in \Sigma_1 \mapsto \phi_X(p) = \varphi_X(p;t_X(p)) \in \Sigma_2,
\]
which is a diffeomorphism, since it is given by the flow of $X$ and the function $t: p \in \Sigma_1 \mapsto t_X(p) \in \mathbb{R}$ is differentiable by the Implicit Function Theorem. Proceeding analogously, we define
\[
\phi_Y : p \in \Sigma_2 \mapsto \phi_Y(p) = \varphi_Y(p;t_Y(p)) \in \Sigma_1.
\]
Finally, the {\it first return map} of $Z$ is defined as
\begin{equation}\label{aplicpoincare}
	\begin{aligned}
		\phi_Z : \; & \Sigma_2^- \rightarrow \Sigma_2^-, \\
		& p \mapsto \phi_Z(p) = \phi_X \circ \phi_Y \circ \phi_X \circ \phi_Y(p),
	\end{aligned}
\end{equation}
which is clearly a one-dimensional diffeomorphism. The flow directions of $X$ and $Y$ are selected to ensure that the composition is well-defined. In equation \eqref{aplicpoincare}, the domain of $\phi_Z$ is chosen as $\Sigma_2^-$; however, any $\Sigma_i^{\pm}$, with $i = 1,2$, could be used to define the corresponding {\it first return map}.

%\begin{defi}
%	Let $Z \in \Omega$ transient and let $\phi_Z$ be its {\it first return map} associated to $Z$ at origin. Then the origin is ``geometrically stable'' if $| \phi_Z^\prime(0)|<1$ and it is ``geometrically unstable'' if $|\phi_Z^\prime(0)|>1$.
%\end{defi}

%In this context, the {\it first return map} can not have a dynamical meaning, that is, the trajectories through  of the point $ p \in \Sigma_i ^ {\pm}, \; i = 1,2, $ may not reach the cross-section $ \Sigma_i ^ {\pm} $ again in the past or future, see Figure \ref{pseudo_ciclo_e_transiente}(b). And even in these cases, $\phi_Z$  is important to detect the existence of {\it pseudo-cycles}.
%Finally in \cite{LarTereTei2021}, the authors showed that if  $Z \in \Omega$ satisfy $X_i \, Y_i(0) >0, \; X_j \, Y_j(0) <0 \; \mbox{with} \; X_1 \, X_2(0)<0 \; \mbox{for} \; i,j = 1,2, \; i \neq j,$ then $Z$ is transient. Moreover, they also proved that the {\it firs return map} associated of $Z$ can be described as
In this context, the {\it first return map} may not have a direct dynamical interpretation; that is, the trajectories through a point $p \in \Sigma_i^{\pm}$, $i = 1,2$, may not intersect the cross-section $\Sigma_i^{\pm}$ again, either in the past or in the future (see Figure \ref{pseudo_ciclo_e_transiente}(b)). Even in such cases, $\phi_Z$ remains important for detecting the existence of {\it pseudo-cycles}. Finally, in \cite{LarSeaTei2021}, the authors showed that if $Z \in \Omega$ satisfies 
\[
X_i \, Y_i(0) > 0, \quad X_j \, Y_j(0) < 0, \quad \text{and} \quad X_1 \, X_2(0) < 0, \quad \text{for } i,j = 1,2, \; i \neq j,
\] 
then $Z$ is transient. Moreover, they proved that the {\it first return map} associated with $Z$ can be described as
%veand suppose the following expression for Poincaré map of fields $X$ and $Y$:
%\begin{equation}\label{poincareX}
%	\phi_X(x) = a_X x + b_X x^2 + c_X x^3 + \of(x^4) ,
%\end{equation}
%\begin{equation}\label{poincareY}
%	\phi_Y(y) = a_Y y + b_Y y^2 + c_Y y^3 + \of(y^4).
%\end{equation}
%Therefore by the composition in the proper way of \eqref{poincareX} and \eqref{poincareY} we obtain Poincaré map of $\phi_Z$ given by
\begin{equation}\label{poincareZ}
	\phi_Z(x) = \alpha_Z ^2 x + (\alpha_Z + \alpha_Z ^2) \beta_Z x^2 + \eta_Z x^3 + \mathcal{O}(x^4), \; x \in \Sigma_2^-
\end{equation}
%where $\alpha_Z=(X_1 \, Y_2/X_2 \, Y_1)(0)<0$. We do not specify the other constants, because they are will not useful for us in our discussions. For more details, see Section 4.2 in \cite{LarTereTei2021}.
where $\alpha_Z = (X_1 \, Y_2 / X_2 \, Y_1)(0) < 0$. The other constants are not specified, as they will not be relevant for our discussion. For more details, see Section 4.2 in \cite{LarSeaTei2021}.

%, \; \beta_Z = b_X \cdot a_Y^2 + a_X \cdot b_Y $ and 
%$$\eta_Z = -2 \left( (b_X \cdot a_Y)^2 + c_X \cdot a_Y ^3 + \left( \frac{b_Y}{a_Y} \right) ^2 + \left( \frac{c_Y}{a_Y} \right) \right).$$

%
%\begin{defi}
%	Consider $Z \in \Omega$ and $\phi_Z$ be its {\it first return map} associated to $Z$ at origin. Then the origin is geometrically stable if $0 < \phi_Z^\prime(0)<1$ and it is geometrically unstable if $\phi_Z^\prime(0)>1$.
%\end{defi}

%\section{Locally Structural Stability and Codimension one Bifurcations}\label{section3}

%In order to continue to the necessary tools to prove ours results, we will  
%enter in the conceps about stabylit considered for the family $\Omega$.

%It is well known that the set $\mathfrak{X} = \mathfrak{X}^r(\mathcal{U})$, $\overline{\mathcal{U}}$, of the germs of vector fields of class $\mathcal{C}^r$, $r \geq 1$ defined on a bounded neighborhood of the origin with the $\mathcal{C}^r$-topology is a Banach Space.  Therefore $\Omega = \mathfrak{X} \times \mathfrak{X}$ is also a Banach Space and consequently, it is a Banach Manifold.

%For our purposes, we must have a classification of structural stability and codimension one bifurcations for the vector fields. First, we will establish definitions of $\Sigma$-structural stability and codimension one bifurcations in the context of family $\Omega$. After that, we will exhibit the classification, of the last aspects, carried out by Larrosa, Teixeira, and M-Seara in \cite{LarTereTei2021}. 

For our purposes, it is necessary to have a classification of structural stability and codimension-one bifurcations for the vector fields. First, we establish the definitions of $\Sigma$-structural stability and codimension-one bifurcations within the context of the family $\Omega$. Then, we present the classification of these aspects as carried out by Larrosa, Teixeira, and M-Seara in \cite{LarSeaTei2021}.

%In the sequel, we need to do a little discussion about structural stability and bifurcations of one codimension for smooth vector fields, which will be useful in the discussion of regularization of the set of locally $\Sigma$-structurally stable vector fields $\Omega_0 \subset \Omega$ and regularization of the fields that have a generic codimension one bifurcations at origin $\Xi_1 \subset \Omega_1 = \Omega \setminus \Omega_0$. 

\begin{defi}
%Let $Z, \; \widehat{Z} \in \Omega$, defined in $\mathcal{U}$ and $\widehat{\mathcal{U}}$ limited neighborhoods of origin, with switching curve $\Sigma$ and $\widehat{\Sigma}$, respectively.  We say that $Z$ and $\widehat{Z}$ are locally $\Sigma$-equivalent if there is neighborhoods $\mathcal{U}_0$, $\widehat{\mathcal{U}_0}$ of the origin and a homeomorphism $ h : \mathcal{U}_0 \rightarrow \widehat{\mathcal{U}_0}$  preserving orientation which maps trajectories of $Z$ in trajectories of $\widehat{Z}$ and sends $\Sigma$ in $\widehat{\Sigma}$. 	We say that $Z \in \Omega$ is locally $\Sigma$-structurally stable at the origin if there is a neighborhood $\mathcal{V}_Z \subset \Omega$ of $Z$ such that if $\widehat{Z} \in \mathcal{V}_Z$ then $\widehat{Z}$ is locally $\Sigma$-equivalent to $Z$.

Let $Z, \; \widehat{Z} \in \Omega$ be defined in $\mathcal{U}$ and $\widehat{\mathcal{U}}$, which are bounded neighborhoods of the origin, with switching curves $\Sigma$ and $\widehat{\Sigma}$, respectively. We say that $Z$ and $\widehat{Z}$ are locally $\Sigma$-equivalent if there exist neighborhoods $\mathcal{U}_0 \subset \mathcal{U}$ and $\widehat{\mathcal{U}}_0 \subset \widehat{\mathcal{U}}$ of the origin and an orientation-preserving homeomorphism $h : \mathcal{U}_0 \rightarrow \widehat{\mathcal{U}}_0$ that maps trajectories of $Z$ to trajectories of $\widehat{Z}$ and sends $\Sigma$ to $\widehat{\Sigma}$. We say that $Z \in \Omega$ is locally $\Sigma$-structurally stable at the origin if there exists a neighborhood $\mathcal{V}_Z \subset \Omega$ of $Z$ such that, for any $\widehat{Z} \in \mathcal{V}_Z$, $\widehat{Z}$ is locally $\Sigma$-equivalent to $Z$.

\end{defi}

%\begin{defi}
%	We say that $Z \in \Omega$ is locally $\Sigma$-structurally stable at the origin if there is a neighborhood $\mathcal{V}_Z \subset \Omega$ of $Z$ such that if $\widehat{Z} \in \mathcal{V}_Z$ then $\widehat{Z}$ is locally $\Sigma$-equivalent to $Z$.
%\end{defi}

%It is worth noting that the last definition makes sense because it is well known that the set $\mathfrak{X} = \mathfrak{X}^r(\mathcal{U})$, $\overline{\mathcal{U}}$ compact,
% of the germs of vector fields of class $\mathcal{C}^r$, $r \geq 1$ defined on a bounded neighborhood of the origin with the $\mathcal{C}^r$-topology is a Banach Space.  Therefore $\Omega = \mathfrak{X} \times \mathfrak{X}$ is also a Banach Space. We will denote by $\Omega_0$ the set of all piecewise systems in $\Omega$ which are locally $\Sigma$-structurally
%stable.  When $Z$ is not locally $\Sigma$-structurally stable at the origin, we say that $Z$ belongs
%to the bifurcation set $\Omega_1 = \Omega \smallsetminus \Omega_0$.

It is worth noting that the previous definition is well-posed because it is well known that the set 
$\mathfrak{X} = \mathfrak{X}^r(\mathcal{U})$, $\overline{\mathcal{U}}$ compact,
of germs of $\mathcal{C}^r$ vector fields, $r \geq 1$, defined on a bounded neighborhood of the origin, endowed with the $\mathcal{C}^r$-topology, forms a Banach space. Therefore, $\Omega = \mathfrak{X} \times \mathfrak{X}$ is also a Banach space. 

We denote by $\Omega_0$ the set of all piecewise systems in $\Omega$ that are locally $\Sigma$-structurally stable. When $Z$ is not locally $\Sigma$-structurally stable at the origin, we say that $Z$ belongs to the bifurcation set 
$
\Omega_1 = \Omega \smallsetminus \Omega_0.
$

\begin{defi}
%	Let $Z \in \Omega$, then a one-parameter unfolding of $Z$ is smooth map $\gamma :   \delta \in (-\delta_0, \delta_0) \rightarrow Z_\delta \in \Omega$ with $\delta_0 \ll 1$ such that $\gamma(0) = Z_0 = Z$. We usually denote an unfolding of $Z$ by $Z_{\delta}$.
	
	Let $Z \in \Omega$. A one-parameter unfolding of $Z$ is a smooth map 
	\[
	\gamma : \delta \in (-\delta_0, \delta_0) \mapsto Z_\delta \in \Omega,
	\] 
	with $\delta_0 \ll 1$, such that $\gamma(0) = Z_0 = Z$. An unfolding of $Z$ is usually denoted by $Z_\delta$.
	
\end{defi}

\begin{defi}\label{defi15}
%	Let $Z, \; \widehat{Z} \in \Omega$. We say that two one-parameter unfoldings $Z_{\delta}$ of $Z$ and $ \widehat{Z}_{\widehat{\delta}}$ of $\widehat{Z}$ are locally weak equivalent if there is a homeomorphic change of parameters $\mu(\delta)$, such that, for each $\delta$ the vector fields $Z_{\delta}$ and $\widehat{Z}_{\mu(\delta)}$ are locally $\Sigma$-equivalent.  Moreover, given a one-parameter unfolding $Z_{\delta}$ of $Z$ it is said to be a versal one-parameter unfolding if every other unfolding $Z_{\alpha}$ of $Z$ is locally weak equivalent
%to $Z_{\delta}$. 

Let $Z, \; \widehat{Z} \in \Omega$. Two one-parameter unfoldings, $Z_{\delta}$ of $Z$ and $\widehat{Z}_{\widehat{\delta}}$ of $\widehat{Z}$, are said to be \emph{locally weakly equivalent} if there exists a homeomorphic change of parameters $\mu(\delta)$ such that, for each $\delta$, the vector fields $Z_{\delta}$ and $\widehat{Z}_{\mu(\delta)}$ are locally $\Sigma$-equivalent. Moreover, a one-parameter unfolding $Z_{\delta}$ of $Z$ is said to be a \emph{versal one-parameter unfolding} if every other unfolding $Z_{\alpha}$ of $Z$ is locally weakly equivalent to $Z_{\delta}$.

\end{defi}

%A bifurcation is a change in the topology of the phase portraits of a dynamical system when a parameter is varied, so we define a bifurcation simply in terms of the loss of structural stability as we vary a parameter. An unfolding  (or versal unfolding) of a bifurcation is a simplified system that for small values of $\delta$ we have all possible structurally stable phase portraits arising under small perturbations. The codimension of a bifurcation is the dimension of the parameter space needed to unfold the bifurcation.

A bifurcation is a change in the topology of the phase portrait of a dynamical system when a parameter is varied. Accordingly, a bifurcation can be defined in terms of the loss of structural stability as the parameter changes. An unfolding (or versal unfolding) of a bifurcation is a simplified system that captures all possible structurally stable phase portraits arising under small perturbations for sufficiently small values of $\delta$. The codimension of a bifurcation is defined as the dimension of the parameter space required to unfold it.

\begin{defi}
%A piecewise smooth vector field  $Z \in \Omega_1 $ has a singularity of codimension one at the origin if it is relatively $\Sigma$-structural stable in the induced topology of $\Omega_1 $.  That is, if there is an open set $\mathcal{V}_Z \subset \Omega$ with $Z \in \mathcal{V}_Z$ such that if $\widehat{Z} \in \mathcal{V}_Z \cap \Omega_1$ then $\widehat{Z}$ is locally $\Sigma$-equivalent
%to $Z$ and any one-parameter unfoldings of $Z$ and $\widehat{Z}$ are locally weak equivalent.  We denote by $\Xi_1$ the set of all bifurcation of codimension one in $\Omega$.

A piecewise smooth vector field $Z \in \Omega_1$ has a singularity of codimension one at the origin if it is relatively $\Sigma$-structurally stable in the induced topology of $\Omega_1$. That is, there exists an open set $\mathcal{V}_Z \subset \Omega$ with $Z \in \mathcal{V}_Z$ such that, for any $\widehat{Z} \in \mathcal{V}_Z \cap \Omega_1$, $\widehat{Z}$ is locally $\Sigma$-equivalent to $Z$, and any one-parameter unfoldings of $Z$ and $\widehat{Z}$ are locally weakly equivalent. We denote by $\Xi_1$ the set of all codimension-one bifurcations in $\Omega$.

\end{defi}
 
%Larrosa, M-Seara, and Teixeira in \cite{LarSeaTei2021}, using the previous definitions, classified the sets of the structurally stable piecewise smooth systems in $\Omega$ and $\Omega_1$ in the following results.
 
 Using the previous definitions, Larrosa, M-Seara, and Teixeira \cite{LarSeaTei2021} classified the sets of structurally stable piecewise smooth systems in $\Omega$ and $\Omega_1$ as summarized in the following results.

%So in \cite{Lar2015} has the following result about $\Sigma$-structural stability on  $\widehat{\Omega_{\pm}}$.

\begin{teo}[\cite{LarSeaTei2021}]\label{teocod0}

		$(\Sigma$-structural stability in $\Omega)$ Denote by $\Omega_0 \subset \Omega$ the set of all vector fields that are locally $\Sigma$-structurally stable near the origin. Let $Z \in \Omega$; then $Z \in \Omega_0$ if and only if $Z$  one of the following conditions holds:
	\begin{itemize}
		\item[($A_0$)] $X_i Y_i(0) > 0$, for $i=1,2$;
		
		\item[($B_0$)] $X_i Y_i(0) < 0$, for $i=1,2$, and
$
		\det Z(0) = (X_1 Y_2 - X_2 Y_1)(0) \neq 0;
$
		
		\item[($C_0$)] $X_i Y_i(0) > 0$ and $X_j Y_j(0) < 0$, for $i,j=1,2$, $i \neq j$. Moreover, if $Z$ is transient, that is, if $X_1 X_2(0) < 0$, then
$
		\alpha_Z^2 = (X_1 \, Y_2/X_2 \, Y_1)^2 (0)\neq 1.
$
	\end{itemize}
	Moreover, the subset $\Omega_0$ is open and dense in $\Omega$. Consequently, local $\Sigma$-structural stability is a generic property in $\Omega$.
\end{teo}

\begin{teo}[\cite{LarSeaTei2021}]\label{teocod1}
	$(\Sigma$-structural stability on $\Omega_1)$ Let $Z \in \Omega_1 = \Omega \setminus  \Omega_0$. Then $Z \in \Xi_1$, that is, $Z$ has a codimension-one singularity at the oring (see Definition \ref{defi15}) if and only if $Z$ one of the following conditions holds:
	\begin{itemize}
		
		\item[($A_1$)] (Regular--fold bifurcation) The origin is a regular fold of $Z$; see Definition \ref{defi2}.
		
		\item[($B_1$)] (Double pseudo-equilibrium bifurcation) $X_i \, Y_i(0) <0$, $\det Z(0)=0$, and $(\det Z)_ {x_i}(0)$ $ \neq 0$, for $i=1,2$.
			
		\item[($C_1$)] (\textit{Pseudo-Hopf bifurcation})
		$X_i Y_i(0) > 0$ and $X_j Y_j(0) < 0$, for $i,j=1,2$, $i \neq j$, together with $X_1 X_2(0) < 0$. In this case, $Z$ is transient and the coefficient $\alpha_Z$ of the \textit{first return map}~\eqref{poincareZ} satisfies $\alpha_Z = -1$. Moreover, the remaining coefficients satisfy $\beta_Z \neq 0$ and $\eta_Z \neq 0$.
	\end{itemize}
	Furthermore, the subset $\Xi_1$ is open and dense in $\Omega_1$. Consequently, local $\Sigma$-structural stability is a generic property in $\Omega_1$.

\end{teo}

%The authors relate a curiosity, for systems in family $\Omega$, the unfoldings of codimension one singularity do not have periodic orbits appearing near the singularity,  contrarily to what happens with the unfoldings of codimension one singularity of Filippov systems with regular switching manifold. In each case of $Z \in \Omega_0$, the authors construct the $\Sigma$-equivalence between $Z$ and its corresponding ``normal form'' which describes its dynamics. And for each case of codimension one, they show that for each $ Z \in \Xi_1$ any versal unfolding of $Z$ is locally weak equivalent to the one-parameter family, that is, there is a ``normal form'' too.

The authors point out an interesting feature of systems in the family $\Omega$: 
the unfoldings of codimension-one singularities do not give rise to periodic orbits in a neighborhood of the singularity, in contrast to what occurs for unfoldings of codimension-one singularities in Filippov systems with a regular switching manifold. For each case $Z \in \Omega_0$, the authors construct a $\Sigma$-equivalence between $Z$ and a corresponding \emph{normal form} that captures its local dynamics. Moreover, for each codimension-one case, they show that for any $Z \in \Xi_1$, every versal unfolding of $Z$ is locally weakly equivalent to a one-parameter family. In particular, a corresponding \emph{normal form} can also be defined in this setting.

%In the sequel, we need to do a little discussion about structural stability and bifurcations of one codimension for smooth vector fields, which will be useful in the discussion of regularization of the set of locally $\Sigma$-structurally stable vector fields $\Omega_0 \subset \Omega$ and regularization of the fields that have a generic codimension one bifurcations at origin $\Xi_1 \subset \Omega_1 = \Omega \setminus \Omega_0$. 

After presenting the classification of structurally stable piecewise smooth systems in $\Omega$ and $\Omega_1$, which will be used throughout this work, we now present several results and remarks that will be useful for the analysis of structural stability and codimension-one bifurcations of the smooth vector fields obtained after applying a smoothing process to the vector fields in $\Omega$ defined by equation~\eqref{sistemacruz}.

%O termo bifurcação foi introduzido por Poincaré em 1885 e é uma referência a mudança qualitativa na estrutura topológica de um retrato de fase de um sistema dinâmico, a medida que parâmetro(s) do sistema passam por valores críticos. 

%When, in a neighborhood, the solutions of two smooth differential systems are homeomorphic such that their orientations are preserved, these systems are locally topologically orbitally equivalent. This means that the solutions of one system can be continuously deformed until they become equal to the trajectories of the other system. We say that a smooth differential system is locally structurally stable if it is locally topologically orbitally equivalent to a sufficiently small perturbation in its equations. So, to conclude the local structural stability of smooth differential systems without equilibrium points in a neighborhood at the origin, we will use the classic Flow Box Theorem (see \cite{Soto1979}).
 
 Two smooth differential systems are said to be locally topologically orbitally equivalent if, in a neighborhood, there exists a homeomorphism that maps the trajectories of one system onto those of the other while preserving their orientations. In this sense, the solutions of one system can be continuously deformed into the trajectories of the other. A smooth differential system is said to be locally structurally stable if it is locally topologically orbitally equivalent to all sufficiently small perturbations of its defining equations. Therefore, in order to establish the local structural stability of smooth differential systems without equilibrium points in a neighborhood of the origin, we shall rely on the classical Flow Box Theorem (see \cite{Soto1979}).

As previously mentioned, at a bifurcation point the differential system is no longer structurally stable. 
Local bifurcations are those that can be detected by analyzing the behavior of a differential system in a neighborhood of an equilibrium point or a closed trajectory. 
In the present work, the codimension of a bifurcation is defined as the number of parameters that must be varied in order to produce the bifurcation. Therefore, we present below conditions for the occurrence of some codimension-one local bifurcations in the plane, which will be required in subsequent analyses. These conditions are stated in the Sotomayor Theorem (see \cite{Perko2013}). Before presenting this theorem, we assume that the map $g(u,\mu)$ is sufficiently smooth so that all derivatives appearing in the statement are continuous. We denote by $Dg$ the Jacobian matrix of $g$ with respect to the variable $u$, and by $g_{\mu}$ the vector of partial derivatives of $g$ with respect to the parameter $\mu$.

\begin{teo}[\cite{Perko2013}]\label{Teosotobiftransc}
	(Sotomayor Theorem) Consider a bidimensional differential system depending on a one-dimensional  parameter $ \mu $, that is, we are considering $\dot{u} = g(u, \mu)$ such that $u \in \re^2$, $\mu \in \re$ and $g$ is an map of $U \subset \re^{2+1}$ in $\re^2$.  We suppose that $g(u_0, \mu_0)=0$, for some $(u_0, \mu_0)$ in domain of $g$, and the matrix $2 \times 2$,  $A \equiv Dg(u_0, \mu_0)$, has a simple eigenvalue $\lambda = 0$ with eigenvector associated $ v $ and that $ A ^ T $ has an eigenvector $ w $ corresponding to what is also its eigenvalue $ \lambda = 0 $. In addition, suppose that 
	$A$ has another eigenvalue with either a negative or a positive real part, and that the following conditions are satisfied.
	\begin{equation}\label{terceiracond_selano}
		\begin{array}{l}
			w^T g_{\mu}(u_0, \mu_0) \neq 0 \; \mbox{and} \;
			w^T[D^2 g(u_0, \mu_0)(v,v)] \neq 0,
		\end{array} 
	\end{equation}
	then the sytem $ \dot {u} = g (u, \mu) $ suffer a saddle-node bifurcation at the  equilibrium point $ u_0 $ with the variation of  parameter $ \mu $ through the bifurcation value $ \mu = \mu_0 $. If the conditions in \eqref {terceiracond_selano} are changed to
	\begin{equation}\label{terceiracond}
		\begin{array}{l}
			w^T g_{\mu}(u_0, \mu_0) = 0, \; w^T[Dg_{\mu}(u_0, \mu_0)v] \neq 0, \; \mbox{e} \; w^T[D^2 g(u_0, \mu_0)(v,v)] \neq 0
		\end{array} 
	\end{equation}
	then the system $ \dot{u} = g (u, \mu) $ suffer a transcritical bifurcation at the equilibrium point $ u_0 $ with the variation of the parameter $ \mu $ through the bifurcation value $ \mu = \mu_0 $.
\end{teo}

%We would like to emphasize that the {\it saddle-node bifurcation}  is generic because perturbations in its equation do not cause changes in the behavior of its unfolding. However,  the {\it transcritical bifurcation} is not generic, because perturbations in its equation cause changes in the behavior of its unfolding, consequently in the dynamics of the differential equation.  Simple examples of the transcritical bifurcation and the saddle-node bifurcation, for two-dimensional differential systems with bifurcation point $(x,y,\mu)=(0,0,0)$, can be described, respectively, by
We emphasize that the \textit{saddle--node bifurcation} is generic, since small perturbations of its defining equation do not alter the qualitative behavior of its unfolding. In contrast, the \textit{transcritical bifurcation} is not generic, because arbitrarily small perturbations of its equation may change the qualitative structure of the unfolding and, consequently, the dynamics of the associated differential equation. Simple examples of the transcritical and saddle--node bifurcations for two-dimensional differential systems, with bifurcation point at $(x,y,\mu)=(0,0,0)$, are given, respectively, by  
\begin{equation}
%	\left\{\begin{array}{l}
			\dot{x_1} = \mu \, x_1 - x_1^2, \quad
			\dot{x_2} = -x_2
	%	\end{array} \right.
\;	\; \; \; \; \; \; \; \mbox{and} \; \; \; \; \; \; \; \; \; \;
%	\left\{ \begin{array}{l}
			\dot{x_1} = \mu - x_1^2, \quad
			\dot{x_2} = -x_2
	%	\end{array} \right. .
\end{equation}

\section{Proofs of Theorems \ref{teonep},  \ref{teo1}, \ref{teo2} and \ref{teo4}.}\label{section3}

%In this section, we will dedicate to proving Theorem 1.1. The proof will be an immediate
%consequence of Proposition 2.1 and from Propositions 4.1 and 4.2 that we will present the
%following.

\begin{proof}[Proof of Theorem \ref{teonep}.]
%We will only focus on one implication since the other will be immediate with such a discussion. Suppose that $Z^R_{\varepsilon, \eta}(p_0)=0$, first we are going to consider the case that 
%$\xi=0$ then with this
We will focus on only one implication, as the converse follows immediately from a similar argument. 
Suppose that $Z^R_{\varepsilon, \eta}(p_0) = 0$. 
First, we consider the case $\xi = 0$; then, with this assumption, 
	\begin{equation}\label{eqcondequi}
		Z^R_{\varepsilon, \eta}(p_0)=\left( \begin{array}{cc} 
			X_1(p_0) & Y_1(p_0)\\
			X_2(p_0) & Y_2(p_0)  
		\end{array} \right)
		\left( \begin{array}{l} 
			1 + \varphi_\varepsilon(p_{01}) \, \psi_\eta(p_{02}) \\
			1 - \varphi_\varepsilon(p_{01}) \, \psi_\eta(p_{02})
		\end{array} \right) =
		\left( \begin{array}{l} 
			0 \\
			0  
		\end{array} \right).
	\end{equation}	
%	If $det[ Z ] ( p_0) \neq0$ for any $p_0 \in \mathcal{U}$ then there is a neighborhood that preserves this property, and necessarily on it
%	
If $\det[Z](p_0) \neq 0$ for every $p_0 \in \mathcal{U}$, then there exists a neighborhood of $p_0$ in which this property is preserved; consequently, 
\begin{equation*}
		1 + \varphi_\varepsilon(x_1) \, \psi_\eta(x_2) = 0 \; \mbox{and} \; 1 -  \varphi_\varepsilon(x_1) \, \psi_\eta(x_2) = 0.
	\end{equation*}
%	causing an absurdity, that is, that there are no equilibrium points at $\mathcal{U}$ ,
%	 which contradicts the hypothesis. On the other hand, if there is $p_0 \in \mathcal{U}$  such that $det[ Z] ( p_0) =0$ then the lines of the $2 \times 2$ matrix  in \eqref{eqcondequi} are proportionals and it is sufficient to consider only one equation given by
	 This leads to a contradiction, since it implies that there are no equilibrium points in $\mathcal{U}$, 
	 which contradicts the hypothesis. On the other hand, if there exists $p_0 \in \mathcal{U}$ such that 
	 $\det[Z](p_0) = 0$, then the rows of the $2 \times 2$ matrix in \eqref{eqcondequi} are proportional, 
	 and it suffices to consider only one of the equations given by	 
	\begin{equation*}\label{eqZi}
		Z^{R_i}_{ \varepsilon , \eta }(p_0)=H_i(p_0)=	\varphi_\varepsilon(p_{10}) \, \psi_\eta(p_{20})  \,  (X_i - Y_i)(p_{10},p_{20}) + (X_i + Y_i)(p_{10},p_{20}).
	\end{equation*}
%	And, when $H_i(p_0)\neq 0$ for any $p_0$ so there is not equilibrium point,
%	and if $H_i(p_0)= 0$ then we have $p_0$ being a equilibrium point.
	% for $\xi$ sufficiently small. 
	If $H_i(p_0) \neq 0$ for any $p_0$, then there is no equilibrium point.  
	Conversely, if $H_i(p_0) = 0$, then $p_0$ is an equilibrium point.

%	Now we are going to analyze the case $\xi \neq 0$.
%	$G( \varphi_\varepsilon(p_{01}), \psi_\eta(p_{02}))\neq 0.$ 
%	Since for $\xi0$ we have the previous case, 
We now analyze the case $\xi \neq 0$.
%	Consequently, the cases without equilibrium points are preserved for $\xi$ sufficiently small. For the remaining case we have that there is $p_0 \in \mathcal{U}$  such that $det[Z( p_0)] = 0$ and $H_i(p_0)=0$. 	In this case,
%	$$F_i(x_1,x_2,\xi)=H_i(x_1,x_2)+\xi \, G_i( \varphi_\varepsilon(x_1), \psi_\eta(x_2)),$$
%	and follows that $F_i(p_{01},p_{02},\xi)=\xi \, G_i( \varphi_\varepsilon(p_{01}), \psi_\eta(p_{02}))$ which proves the result.
	Consequently, the cases without equilibrium points are preserved for sufficiently small $\xi$.  
	For the remaining case, there exists $p_0 \in \mathcal{U}$ such that $\det[Z(p_0)] = 0$ and $H_i(p_0)=0$.  
	In this case, we define
	\[
	F_i(x_1,x_2,\xi) = H_i(x_1,x_2) + \xi \, G_i\big(\varphi_\varepsilon(x_1), \psi_\eta(x_2)\big),
	\]
	and it follows that
$
	F_i(p_{01},p_{02},\xi) = \xi \, G_i\big(\varphi_\varepsilon(p_{01}), \psi_\eta(p_{02})\big),
$
	which proves the result.

\end{proof}

\begin{proof}[Proof of Theorem \ref{teo1}.]
	
%	We must analyze nonlinear double regularization for each of the cases in Theorem \ref{teocod0} when Z is structurally stable ($Z \in \Omega_0$). The idea is to show that each nonlinear double regularization does not have equilibrium points at the origin and then the Flow Box Theorem completes the proof. 
	
	We now analyze the nonlinear double regularization for each case in Theorem~\ref{teocod0}, assuming that $Z$ is structurally stable ($Z \in \Omega_0$). The goal is to show that, in each case, the nonlinear double regularization has no equilibrium points at the origin. The proof is then completed by applying the Flow Box Theorem.
	
	%	And let us use the Remark \ref{obsdet} for Sotomayor-Teixeira double regularization generalized plus the Tubular Flow Theorem to complete the proof.
	
	%In the case of the vector field $Z$ satisfying ($B_0$) in Theorem \ref{teocod0} then $detZ(0) \neq 0$. And, for a vector field $Z$ satisfying ($C_0$) in Theorem \ref{teocod0}, if
	%
	%$$X_i(0).Y_i(0) >0, \;X_j(0).Y_j(0) <0 \; \mbox{for} \; i,j = 1,2, \; i \neq j.$$ Supposing that $detZ(0) = 0$ then $X_1(0)/Y_1(0)=X_2(0)/Y_2(0)$, that is, 
%	In the case of the vector field $Z$ satisfying ($B_0$) then $detZ(0) \neq 0$. And, for a vector field $Z$ satisfying ($C_0$), if $detZ(0) = 0$ then
%	$$\sgn (X_1(0) \, Y_1(0))=\sgn(X_1(0)/Y_1(0)) = \sgn(X_2(0)/Y_2(0))=\sgn (X_2(0) \, Y_2(0)),$$ 
%	causing an absurdity. In this way, $det Z(0) \neq 0$, regardless if the vector field is transient or not.
	In the case where the vector field $Z$ satisfies condition ($B_0$), we have $\det Z(0) \neq 0$.  
	Moreover, for a vector field $Z$ satisfying condition ($C_0$), if we assume $\det Z(0) = 0$, then
$$\sgn (X_1(0) \, Y_1(0))=\sgn(X_1(0)/Y_1(0)) = \sgn(X_2(0)/Y_2(0))=\sgn (X_2(0) \, Y_2(0)),$$ 
	which leads to a contradiction. Therefore, we conclude that $\det Z(0) \neq 0$, regardless of whether the vector field is transient or not.
	%	
	%	No caso dos sistemas em $(A)$, ou seja, $Z \in \Omega_0^1$, por hipótese $X_i(0).Y_i(0) >0$ para $i=1,2$.  

		Finally, let $Z$ be a vector field satisfying condition ($A_0$). If $\det Z(0) = 0$, consider the expression $H_i(0)=\varphi_\varepsilon(0) \, \psi_\eta(0)  \,  (X_i - Y_i)(0,0) + (X_i + Y_i)(0,0)$ for some $i=1,2.$ If $\varphi_\varepsilon(0) \, \psi_\eta(0)  \,  (X_i - Y_i)(0,0) =0$, then $H_i(0)= (X_i + Y_i)(0,0) \neq0 $ since $X_i(0) \, Y_i(0)>0$. On the other hand, if $\varphi_\varepsilon(0) \, \psi_\eta(0)  \,  (X_i - Y_i)(0,0) \neq 0$, then  $H_i(0)=0$ if and only if $\varphi \, \psi(0)=\varphi_\varepsilon(0) \, \psi_\eta(0)  =-(X_i(0)+Y_i(0))/(X_i(0)-Y_i(0))$.
% which with the hypothesis implies that 
Under the hypotheses, this would imply	$X_i(0) \, (X_i-Y_i)^{-1}(0)>0$ and $Y_i(0) \, (X_i-Y_i)^{-1}(0)<0,$ 	which contradicts the assumption. Therefore, $H_i(0) \neq 0$.

%	In this way, for each one of the cases of vector fields belonging to the family $\Omega_0$ ($\Sigma$-structurally stable),  $det Z(0) \neq 0$ or $H_i(0)\neq 0$. Therefore Theorem \ref{teonep} guarantees the nonexistence of an equilibrium point at the origin for all cases. In addition, there is a sufficiently small neighborhood of the origin that does not contain equilibrium points. And, the Flow Box Theorem guarantees the structural stability in this neighborhood. The real numbers $ \varepsilon_0 $ and $ \eta_0 $ are such that nonlinear double regularization remains in this neighborhood. And, in all cases, we have zero codimension preservation of the bifurcation for any parameter $\xi$. 
	
	In this way, for each case of vector fields belonging to the family $\Omega_0$ (i.e., $\Sigma$-structurally stable), either $\det Z(0) \neq 0$ or $H_i(0) \neq 0$. Therefore, Theorem~\ref{teonep} guarantees the nonexistence of an equilibrium point at the origin in all cases. Moreover, there exists a sufficiently small neighborhood of the origin that contains no equilibrium points. By the Flow Box Theorem, the structural stability is ensured in this neighborhood. The real numbers $\varepsilon_0$ and $\eta_0$ are chosen so that the nonlinear double regularization remains within this neighborhood. Furthermore, in all cases, the bifurcation exhibits zero codimension preservation for any parameter $\xi$.

\end{proof}

\begin{proof}[Proof of Theorem \ref{teo2}.] The idea of the proof is the same as in the proof of Theorem~\ref{teo1}. Case $(C_1)$ is analogous to case $(C_0)$. In turn, for the regular–fold bifurcation, case $(A_1)$, by hypothesis, for $i,j = 1,2$ with $i \neq j$,
	%	
	%	$\;$
	%	
	%	, we must analyze each of the three cases in Theorem \ref{teocod1}. And also let us use the Remark \ref{obsdet} for Sotomayor-Teixeira double regularization generalized plus the Tubular Flow Theorem to complete the proof of the cases that are structural stable. For the cases that have codimension one, we need to use Sotomayor Theorem (Theorem \ref{Teosotobiftransc}) in the double regularization considered.
	\begin{itemize}
		%	\item [(i)] Here, we have the cases (A) Pseudo-Hopf bifurcation and (B) Regular-fold bifurcation of the Theorem \ref{teocod1}.
		
		%	The case of Pseudo-Hopf bifurcation is analogous to the vector fields satisfying $(C)$ in Theorem \ref{teocod0}. It is the same proof, then $detZ(0) \neq 0$.
		
		\item [($A_{1}^a$)] $X_i(0)=0$, $X_j(p)\,\dfrac{\partial X_i}{\partial x_j}(p) \neq 0$, and $Y_i(0) \neq 0$, for $i=1,2$, or
		
		\item [($A_{1}^b$)] $Y_i(0)=0$, $Y_j(p)\,\dfrac{\partial Y_i}{\partial x_j}(p) \neq 0$, and $X_i(0) \neq 0$, for $i=1,2$.

	\end{itemize}
	%	Observing that $det Z(0) = (X_1 \cdot Y_2 - X_2 \cdot Y_1)(0) = ((-1)^j X_i \cdot Y_j -  (-1)^j Y_i \cdot X_j)(0),$
	%	then considering ($A_{1}^a$) it follows that  $detZ(0)=-(-1)^jY_i \cdot X_j(0) \neq 0$. On the other hand if the vector field satisfies ($A_{1}^b$), then $detZ(0)=(-1)^jX_i \cdot Y_j(0) \neq 0$. 
	%	In turn, the case ($B_1$) is analogous to case ($A_0$), where here we have $(X_2-Y_2)(0)$ never vanishing, then by the hypothesis follows that $X_i(0) \, (X_i-Y_i)^{-1}(0)<0$ or $Y_i(0) \, (X_i-Y_i)^{-1}(0)>0$ which causes an absurd. Therefore in all three cases, there are no equilibrium points.
	
	Observing that
	\[
	\det Z(0) = (X_1 Y_2 - X_2 Y_1)(0)
	= \bigl((-1)^j X_i Y_j - (-1)^j Y_i X_j\bigr)(0),
	\]
	we consider case $(A_{1}^a)$. It follows that
	\[
	\det Z(0) = -(-1)^j Y_i X_j(0) \neq 0.
	\]
	On the other hand, if the vector field satisfies $(A_{1}^b)$, then
	\[
	\det Z(0) = (-1)^j X_i Y_j(0) \neq 0.
	\]
	
	In turn, case $(B_1)$ is analogous to case $(A_0)$. Here, $(X_2 - Y_2)(0)$ never vanishes, and by the hypothesis it follows that
	\[
	X_i(0)\,(X_i - Y_i)^{-1}(0) < 0
	\quad \text{or} \quad
	Y_i(0)\,(X_i - Y_i)^{-1}(0) > 0,
	\]
	which leads to a contradiction. 
	
	Therefore, in all three cases, there are no equilibrium points.

	%		$\;$
	%		
	%		In this way, each one of the cases of vector fields suffering Pseudo-Hopf bifurcations or Regular-fold bifurcations,
	%		are satisfying the property $det Z(0) \neq 0$. Therefore, there is a sufficiently small neighborhood of the origin that does not contain equilibrium points. The Tubular Flow Theorem guarantees structural stability in this neighborhood. The real numbers $ \varepsilon_0 $ and $ \eta_0 $ are such that Sotomayor-Teixeira double regularization generalized remains in this neighborhood. Furthermore, in these cases, we do have not the preservation of the bifurcation codimension and, the result is valid for any transition functions.
	
	%		\vspace{0.3cm}
	%		\item[(ii)] Here our attention is focused on the $(C)$ Double pseudo-equilibrium bifurcation.
	
	%	\end{itemize}

%	$\;$
%	
%	
%	($A_1$) is analogous to case ($A_0$) and, in turn, the case ($B_1$) is analogous to case ($C_0$), 
%	
%	
%	which here $(X_2-Y_2)(0)$ never is vanished. However, 	
\end{proof}

%$\;$
%
%que sao as questao de ser generico que devemos ver como vamos abordar
%antes de demonstrar precisaremos de alguns lemas
%
%$\;$

%\textcolor{blue}{\begin{lemma}
%		Consider the system
%		\begin{equation}\label{eqpsuedoequilibrioduplo2}
%			\widetilde{Z_{\alpha}} = \left\{ \begin{array}{l}
%				\widetilde{X_{\alpha}}(x,y) = \left( \begin{array}{c}
%					a - b c_2 x \\ 
%					b + a \alpha
%				\end{array}  \right) , \; xy >0 \\ 
%				\widetilde{Y}(x,y) = \left( \begin{array}{c}
%					-a \\ 
%					-b + a c_1 y
%				\end{array}  \right) , \; xy <0
%			\end{array} \right.,
%		\end{equation}
%		where $a = sgn ( X_1(0) ), \; b=  sgn ( X_2(0) ), \; c_i = sgn ( \frac{\partial}{\partial x_j} det Z(0) )$ with $i,j = 1,2$ and $i \neq j$. And consider its regularization 
%		\begin{equation}\label{eq2par2} 
%			Z^R_{\varepsilon, \eta,\alpha} = \phi \left( \dfrac{x}{\varepsilon} \right)\psi \left( \dfrac{y}{\eta} \right)  \left( \frac{\widetilde{X_\alpha} - \widetilde{Y}}{2} \right) +  \left( \frac{\widetilde{X_\alpha} + \widetilde{Y}}{2} \right).
%		\end{equation}
%		If the transition function $\phi$ satisfies $\phi(0)=0$ then there are positive values $\varepsilon_0>0$ and $\eta_0>0$ such that for all $0<\varepsilon<\varepsilon_0$ and $0<\eta<\eta_0$ there is a parameter value $\alpha_0=\alpha_0(\epsilon,\eta)$ such that family \eqref{eq2par2} undergoes a transcritical bifurcation.
%\end{lemma}}

%\textcolor{blue}{

\begin{proof}[Proof of Theorem \ref{teo4}.]
	\begin{itemize}
	\item[(i)]
	Consider the nonlinear double regularization applied to the family $\widetilde{Z_{\alpha}}$ give in~\eqref{eq:pseudoequilibrio_duplo3},
		\begin{equation}\label{eq2par2} 
			Z^R_{\varepsilon, \eta,\alpha}%(x_1,x_2)
			 = \phi \left( \dfrac{x_1}{\varepsilon} \right)\psi \left( \dfrac{x_2}{\eta} \right)  \left( \frac{\widetilde{X_\alpha} - \widetilde{Y}}{2} \right) +  \left( \frac{\widetilde{X_\alpha} + \widetilde{Y}}{2} \right).
		\end{equation}
%		Under the hypothesis $\phi(0)=0$ we have that $
%		(0,-\alpha/c_1)$ is a singular point of  \eqref{eq2par2} for all $\alpha$. The linearization of \eqref{eq2par2} at $
%		(0,-\alpha/c_1)$ is
%		\[ \left(\begin{array}{cc}  a\phi^\prime(0) \, \psi(-\alpha/c_1\eta)/\varepsilon -b \, c_2/2 & 0 \\ \ast & a \, c_1/2 \end{array}\right).\]
%		We choose $\alpha_0$ as the solution of the equation
%		
		Under the hypothesis $\phi(0)=0$, we have that $(0,-\alpha/c_1)$ is a singular point of~\eqref{eq2par2} for all $\alpha$. 
		The linearization of~\eqref{eq2par2} at $(0,-\alpha/c_1)$ is
		\[
		\begin{pmatrix}
			a\,\phi'(0)\,\psi(-\alpha/(c_1\eta))/\varepsilon - b\,c_2/2 & 0 \\
			\ast & a\,c_1/2
		\end{pmatrix}.
		\]
		We choose $\alpha_0$ as the solution of the equation
		\begin{equation} \label{eqalpha0}
			\psi\left(-\frac{\alpha}{c_1\eta}\right)= \frac{b \, c_2\varepsilon}{2 \, a \, \phi^\prime(0)}.
		\end{equation}
		Observe that equation \eqref{eqalpha0} has solution only if $0<\varepsilon<2\phi^\prime(0)$. We now perform the change of variables $x_1=r$, $x_2=s-(\mu+\alpha_0)/c_1$ and $\alpha=\mu+\alpha_0$. System ~\eqref{eq2par2} becomes
		\begin{equation}\label{eq2par3} 
			Z^R_{\varepsilon, \eta,\mu} = \left( \begin{array}{cc} A_1 \, \mu & 0 \\ A_2+A_3 \, \mu-A_1 \, \mu^2 & a \, c_1/2 \end{array} \right)\left( \begin{array}{c} r \\ s \end{array} \right) + \left( \begin{array}{c} A_4 \, r^2 +\cdots \\ \cdots \end{array} \right),
		\end{equation}
		where 
		\[\begin{aligned}
			A_1  =-a \, \phi^\prime(0) \, \psi^\prime(-\alpha_0/c_1\eta)/(c_1 \varepsilon \eta),\quad A_2 & =b \, c_2 \, (\alpha_0 \, a+b)/(2 \, a), \\
			A_3 =b \, c_2/2- \phi^\prime(0) \, \psi^\prime(-\alpha_0/c_1\eta) \, (\alpha_0a+b)/(c_1\varepsilon\eta), \quad A_4 &=-a/4+b \, c_2 \, \phi^{\prime\prime}(0)/(4 \, \phi^\prime(0) \, \varepsilon).
		\end{aligned}\]
		In order to put the linear part of $Z^R_{\varepsilon, \eta,0}$ at $(r,s)=(0,0)$ into its Jordan normal form we apply the linear change of coordinates $r=-a \, c_1 \, x_1/2$ and $s=(c_2/(2 \, a)+b \, c_2 \, \alpha_0/2) \, x_1+x_2$. 	System~\eqref{eq2par3} then becomes
		\begin{equation}\label{eq2par4} 
			Z^R_{\varepsilon, \eta,\mu} = g(x_1,x_2,\mu)= \left( \begin{array}{l} A_5 \, x_1 \mu + A_6 \, x_1^2 + \cdots \\  a \, c_1 \, x_2/2 + \cdots \end{array} \right) ,
		\end{equation}
		where 
		\[\begin{aligned}
		A_5 &=-a \, \phi^\prime(0) \, \psi^\prime(-\alpha_0/c_1\eta)/(c_1 \, \varepsilon \, \eta) \phi^{\prime\prime}(0)/(4 \, \phi^\prime(0) \, \varepsilon), \\
		A_6 & =c_1/8+  c_2 \, \phi^\prime(0) \, \psi^\prime(-\alpha_0/(c_1 \, \eta)) (1+a \, b \, \alpha_0)/(2 \, \varepsilon \, \eta) -a \, b \, c_1 \, c_2 \, \phi^{\prime\prime}(0)/(8 \, \phi^\prime(0) \, \varepsilon)
		\end{aligned}\]
	%	$A_5=-a\phi^\prime(0)\psi^\prime(-\alpha_0/c_1\eta)/c_1\varepsilon\eta$ and $A_6=c_1/8+  c_2\phi^\prime(0)\psi^\prime(-\alpha_0/c_1\eta) (1+ab\alpha_0)/2\varepsilon\eta -abc_1c_2\phi^{\prime\prime}(0)/8\phi^\prime(0)\varepsilon$.
		In the notation of Theorem~\ref{Teosotobiftransc}, the eigenvector associated with the eigenvalue $0$ of the linear part of~\eqref{eq2par4} is $v^T=w^T=(1,0)$. It follows that
		\begin{equation}\label{TC}
			\begin{aligned}
				&w^Tg_\mu(0,0,0) =0,\\
				& w^TDg_\mu(0,0,0)(v) = A_5\neq0,\\
				&w^T[D^2g(0,0,0)(v,v)] = A_6\neq0,
			\end{aligned}
		\end{equation}
		for $\varepsilon$ and $\eta$ sufficiently small, that is, there exist $0<\varepsilon_0<2 \, \phi^\prime(0)$ and $\eta_0>0$ such that~\eqref{TC} are satisfied for all $0<\varepsilon<\varepsilon_0$ and $0<\eta<\eta_0$. From Theorem~\ref{Teosotobiftransc}, we conclude that system~\eqref{eq2par4} undergoes a transcritical bifurcation at the point $(x_1,x_2,\mu)=(0,0,0)$. Consequently, system~\eqref{eq2par2} undergoes a transcritical bifurcation at the point $(x_1,x_2,\alpha)=(0,-\alpha_0/c_1,\alpha_0)$.

\item[(ii)] Now consider the nonlinear double regularization applied to the family $\widehat{Z}_{\beta}$ given in~\eqref{eq:pseudoequilibrio_duplo4}. 
Under the additional hypothesis $\psi(0) = 0$, the proof is analogous to the previous one. 
Hence, under the conditions of Theorem~\ref{Teosotobiftransc}, we have
\begin{equation*}\label{TC1}
	\begin{aligned}
		&w^{T} g_{\mu}(u_0, \mu_0) = \frac{a}{2} \neq 0, \\[3pt]
		&w^{T}\!\big[D^{2} g(u_0, \mu_0)(v,v)\big] 
		= \frac{-2 b\, c_2\, \varphi'(0)\, \psi'(0)}{c_1\, \varepsilon\, \eta} \neq 0.
	\end{aligned}
\end{equation*}

	\end{itemize}
\end{proof}
%}

%Now using Theorem \ref{teocod1} we can observe three classes of the systems $\Sigma$-structural stability on bifurcation set $\Omega_1= \Omega \setminus \Omega_0$. So we going to analyze each one of those classes in order to prove our Theorem \ref{teo2}.

	\section{Aspects of an equilibrium point at the origin}\label{sectioneqorigem}

%Since it is interesting to analyze the double nonlinear regularization sufficiently close to the piecewise smooth system, it is interesting to analyze the equilibrium points that persist for all small $\eta$ and $\varepsilon$. 

Since it is important to analyze the double nonlinear regularization sufficiently close to the piecewise-smooth system, we focus on the equilibrium points that persist for all small $\eta$ and $\varepsilon$.

%From now on we will analyze cases in which the origin is an equilibrium point of the nonlinear double regularization  \eqref{eq2par1}. 
%Therefore, by Theorem \ref{teonep} we must have $det[Z](0)=0$ and $H_2(0)=0$. 
%Furthermore, the proof of Theorems \ref{teo1} and \ref{teo2} show that the only possibility of the origin being an equilibrium point of the nonlinear double regularization consists of when the piecewise smooth vector field $Z$ belongs to the subclass ($A_0$) $\subset \Omega_0$ or it is an element of the subclass ($B_1$) $\subset \Xi_1 \subset \Omega_1 = \Omega \setminus \Omega_0$. 

%Consequently, from $det[Z](0)=0$ and
%the conditons which $Z$ must to satisfy, follows that $(X_1(0),Y_1(0))= f(0) \, (X_2(0),Y_2(0))$ with $f(0)\neq0$. 
%And also, if $Z \in$ ($A_0$) or $Z \in$ ($A_1$) and the origin is an equilibrium point of $Z^R_{\varepsilon, \eta}$ then necessarily $\varphi_\varepsilon(0) \, \psi_\eta(0) \, (X_2- Y_2)(0,0) \neq 0$ or $(X_2 - Y_2)(0,0) \neq 0$, respectively. And, we also have $\varphi \, \psi(0)=\varphi_\varepsilon(0) \, \psi_\eta(0)  =-(X_2(0)+Y_2(0))/(X_2(0)-Y_2(0))$. Remember that, as $Z=(X,Y)$ then we can denote
%$$Z^T= \left( \begin{array}{ccc} 
%	X_1 &  X_2\\
%	Y_1&  Y_2
%\end{array} \right).$$

From now on, we analyze the cases in which the origin is an equilibrium point of the nonlinear double regularization~\eqref{eq2par1}. 
Therefore, by Theorem~\ref{teonep}, we must have $ \det  = 0$ and $H_2(0) = 0$. Furthermore, the proofs of Theorems~\ref{teo1} and~\ref{teo2} show that the only possibility for the origin to be an equilibrium point of the nonlinear double regularization occurs when the piecewise-smooth vector field $Z$ belongs to the subclass $(A_0) \subset \Omega_0$, or when it belongs to the subclass $(B_1) \subset \Xi_1 \subset \Omega_1 = \Omega \setminus \Omega_0$. Consequently, from $\det[Z](0) = 0$ and the conditions that $Z$ must satisfy, it follows that
\[
(X_1(0),Y_1(0)) = f(0)\,(X_2(0),Y_2(0)) \quad \text{with } f(0)\neq 0.
\]
Moreover, if $Z \in (A_0)$ or $Z \in (A_1)$ and the origin is an equilibrium point of $Z^R_{\varepsilon, \eta}$, then necessarily
\[
\varphi_\varepsilon(0)\,\psi_\eta(0)\,(X_2 - Y_2)(0,0) \neq 0 \quad \text{or} \quad (X_2 - Y_2)(0,0) \neq 0,
\]
respectively. In addition, we have
\[
\varphi \,\psi(0) = \varphi_\varepsilon(0)\,\psi_\eta(0) = -\frac{X_2(0)+Y_2(0)}{X_2(0)-Y_2(0)}.
\]
Finally, recall that, for $Z=(X,Y)$, we can denote
\[
Z^T = \begin{pmatrix} 
	X_1 & X_2 \\
	Y_1 & Y_2
\end{pmatrix}.
\]

%In this way, considering $G(W, T)=G( \varphi_\varepsilon(x_1), \psi_\eta(x_2))$,  $\triangledown \, G_i$ the gradient vector in relation to the variables $(W,T)$ for each $i=1,2$, and if $v=(v_1,v_2)$ is a vector then its orthogonal vector given by $v^\perp=(-v_2,v_1)$. So we calculate the expressions for the determinant and trace of the Jacobian matrix of $Z^R_{\varepsilon, \eta}$ in the equilibrium point at the origin, respectively, given by

In this way, considering $G(W,T)=G(\varphi_\varepsilon(x_1),\psi_\eta(x_2))$, denoting by $\nabla G_i$ the gradient with respect to the variables $(W,T)$ for each $i=1,2$, and letting $v=(v_1,v_2)$ be a vector with orthogonal vector $v^\perp = (-v_2,v_1)$, we compute the expressions for the determinant and the trace of the Jacobian matrix of $Z^R_{\varepsilon,\eta}$ at the equilibrium point at the origin, which are given by
%
%In order to find conditions for the hyperbolicity of the equilibrium point at the origin, considering $G(W, T)=G( \varphi_\varepsilon(x_1), \psi_\eta(x_2))$, we compute the expressions for the determinant and trace of the Jacobian matrix of $Z^R_{\varepsilon, \eta}(0)$ at the origin, respectively, given by
\begin{equation}
	\begin{aligned}
		det [ D Z^R_{\varepsilon, \eta}]&=   \frac{\partial \, det [ Z]}{\partial \, x_2} \,  \frac{\varphi^\prime\, \psi }{2 \, \varepsilon} - \frac{\partial  \, det [ Z]}{\partial \, x_1} \, \frac{\varphi \, \psi^\prime}{2 \, \eta}+det_{0}+det_1 \, \xi + det[DG] \,  \frac{\varphi^\prime \, \psi^\prime}{\varepsilon \, \eta} \, \xi^2 ,
	\end{aligned}
\end{equation}
such that
\begin{equation*}
	\begin{aligned}
		det_{0} & = \dfrac{Y_2^2 \, det[DX] + Y_2 \, X_2 \, (det[DX]-det[D(X+Y)] +det[DY]) + X_2^2 \, det[DY] }{(X_2-Y_2)^2}, \\
		%+ \dfrac{X_2^2}{(X_2-Y_2)^2}  \,  det[DY]  +\\
		%& \phantom{det_{00}}+ \dfrac{Y_2 \, X_2}{(X_2-Y_2)^2}  \,(det[DX]-det[D(X+Y)] +det[DY]);\\
		det_1 & =  (X_2-Y_2) \, det_{11}  \, \dfrac{\varphi^\prime \, \psi^\prime }{2 \, \varepsilon \, \eta}  +\dfrac{1}{(X_2-Y_2)} \, det_{12}   \, \dfrac{ \psi^\prime}{\varepsilon} -  \dfrac{1}{(X_2-Y_2)} \, det_{13} \, \dfrac{ \varphi^\prime}{\eta},
	\end{aligned}
\end{equation*}
with
\begin{equation*}
	\begin{aligned}
		%& \left(   f  \, \dfrac{\partial \, G_2}{\partial \, T} - \dfrac{\partial \, G_1}{\partial \, T} \right) \,  \psi-  \left(   f  \, \dfrac{\partial \, G_2}{\partial \, W} -  \dfrac{\partial \, G_1}{\partial \, W} \right) \,   \varphi, \\
		%d_{12}  \, \varphi^\prime(0) \, \psi^\prime(0) \, \varphi(0) 
		& det_{12} = \left\langle \dfrac{\partial \, G}{\partial \, W} \, , \,  \left( (Y_2 ,- X_2)  \; \; \dfrac{\partial \, Z^T}{\partial \, x_1} \right)^\perp \right\rangle, \qquad
		%&\left\langle \dfrac{\partial \, G}{\partial \, W} \, , Y_2 \, \left( \dfrac{\partial \, X}{\partial \, x_1} \right)^\perp  - X_2 \,  \left( \dfrac{\partial \, Y}{\partial \, x_1} \right)^\perp \right\rangle, \\
		det_{13} = \left\langle \dfrac{\partial \, G}{\partial \, T} \, , \,  \left( (Y_2 ,- X_2)  \; \; \dfrac{\partial \, Z^T}{\partial \, x_2} \right)^\perp \right\rangle, \\
		%& det_{13} =\left\langle \dfrac{\partial \, G}{\partial \, T} \, , Y_2 \, \left( \dfrac{\partial \, X}{\partial \, x_2} \right)^\perp  - X_2 \,  \left( \dfrac{\partial \, Y}{\partial \, x_2} \right)^\perp \right\rangle, \\
		%& \left(Y_2 \, \left\langle \dfrac{\partial \, G}{\partial \, T}, \left( \dfrac{\partial \, X}{\partial \, x_1} \right)^\perp \right\rangle- X_2 \, \left\langle \dfrac{\partial \, G}{\partial \, T}, \left( \dfrac{\partial \, Y}{\partial \, x_1} \right)^\perp \right\rangle \right)  ,
		%&+\dfrac{1}{(X_2-Y_2)} \, \dfrac{ \varphi^\prime}{\eta}, \\
		%& det_2(0) = d_{21}  \, \varphi^\prime(0) \, \psi^\prime(0).\\
		&  det_{11} = \left\langle f \,  \triangledown \, G_2 -  \triangledown \, G_1 \, , \, (\psi,\varphi)^\perp  \right\rangle, \quad
	\end{aligned}
\end{equation*}
and
\begin{equation}
	\begin{aligned}
		tr[  D Z^R_{\varepsilon, \eta}]&= f \, (X_2-Y_2)  \, \frac{\varphi^\prime\, \psi }{2 \, \varepsilon} + (X_2-Y_2)\,  \frac{\varphi \, \psi^\prime}{2 \, \eta} +tr_{0}+tr_1\, \xi,
	\end{aligned}
\end{equation}
with	
\begin{equation*}
	\begin{aligned}
		%	& tr_0= \dfrac{1}{2 \eta} \, (X_2-Y_2)\,  \varphi \, \psi^\prime +f \, (X_2-Y_2)  \,\frac{\varphi^\prime\, \psi }{2 \, \varepsilon} \\
		& tr_{0}  =  \dfrac{X_2 \, tr[DY] - Y_2  \, tr[DX] }{(X_2-Y_2)},  \qquad \qquad
		%-\dfrac{Y_2}{(X_2-Y_2)}  \, tr[DX]\\
		tr_1= \dfrac{\partial \, G_1}{\partial \, W} \,  \dfrac{\varphi^\prime}{\varepsilon} + \dfrac{\partial \, G_2}{\partial \, T}   \,  \dfrac{\psi^\prime}{\eta}.
	\end{aligned}
\end{equation*}

\vspace{0.1cm}

%The origin will be a hyperbolic equilibrium point when the real part of the eigenvalues relative to its Jacobian matrix is not zero. Given the Jacobian matriz at the origin,
%then considering its determinant ``$det [ D Z^R_{\varepsilon, \eta}](0)$'' and its trace ``$tr[ D Z^R_{\varepsilon, \eta}](0)$'' then its respective characteristic polynomial will be given by

The origin is a hyperbolic equilibrium point when the real parts of the eigenvalues of its Jacobian matrix are nonzero. 
Given the Jacobian matrix at the origin, by considering its determinant ``$det [ D Z^R_{\varepsilon, \eta}](0)$'' and its trace ``$tr[ D Z^R_{\varepsilon, \eta}](0)$'', the corresponding characteristic polynomial is given by
\begin{equation}
	\lambda^2 - tr[ D Z^R_{\varepsilon, \eta}](0) \, \lambda + det [ D Z^R_{\varepsilon, \eta}](0).
\end{equation}
%Therefore, direct calculations and the use Tables \ref{table1}-\ref{table4} prove the following result.
%The term (a.e.p) in each table means that almost every point, that is, there is a hypersurface in the three-dimensional space of parameters $(\varepsilon,\eta,\xi)$ such that the respective condition is not valid.
Therefore, direct calculations together with the use of Tables~\ref{table1}--\ref{table4} yield the following result.
The term (a.e.p.) in each table means \emph{almost every point}, that is, there exists a hypersurface in the three-dimensional parameter space $(\varepsilon,\eta,\xi)$ for which the corresponding condition does not hold.

\begin{prop}\label{propcaseA0B1}
%	Consider $Z \in$ ($A_0$) $\subset \Omega_0$ such that $det[Z](0)=0$ and $(X_2-Y_2)(0)\neq 0$ or $Z \in$ ($B_1$) $\subset \Xi_1 \subset \Omega_1 = \Omega \setminus \Omega_0$. If 
%	$\varphi \, \psi(0)
%	\notin [-1,1]$ or, respectively, $\varphi \, \psi(0)
%	\in (-1,1)$  then the  nonlinear double regularization $Z^R_{\varepsilon, \eta}$ \eqref{eq2par1} has the following properties:
%	
	Consider $Z \in (A_0) \subset \Omega_0$ such that $det[Z](0)=0$ and $(X_2 - Y_2)(0) \neq 0$, or $Z \in (B_1) \subset \Xi_1 \subset \Omega_1 = \Omega \setminus \Omega_0$. 
	If $\varphi\,\psi(0) \notin [-1,1]$ or, respectively, $\varphi\,\psi(0) \in (-1,1)$, then the nonlinear double regularization $Z^R_{\varepsilon,\eta}$ in~\eqref{eq2par1} has the following properties:
	\begin{itemize}
%		\item[(i)] the origin is a hyperbolic equilibrium point  when the next statements occurs:
%		
%		\begin{itemize}
%			\item[($i_1$)] if $D Z^R_{\varepsilon, \eta}(0)$ has real eigenvalues and satisfies
%			$(2_d)$ or  $(3_d)$ in  Table \ref{table1}, or even
%			$(1.2_d)$or  $(1.3_d)$  in Table \ref{table3} (a.e.p.);
%			
%			\item[($i_1$)]  if $D Z^R_{\varepsilon, \eta}(0)$ has complex eigenvalues and satisfies
%			$(2_t)$ or  $(3_t)$ in  Table \ref{table2}, or even
%			$(1.2_t)$or  $(1.3_t)$  in Table \ref{table4} (a.e.p.);
%			
%		\end{itemize}
			\item[(i)] The origin is a hyperbolic equilibrium point when the following statements hold:
		\begin{itemize}
			\item[($i_1$)] If $D Z^R_{\varepsilon, \eta}(0)$ has real eigenvalues and satisfies
			$(2_d)$ or $(3_d)$ in Table~\ref{table1}, or $(1.2_d)$ or $(1.3_d)$ in Table~\ref{table3} (a.e.p.);
			
			\item[($i_2$)] If $D Z^R_{\varepsilon, \eta}(0)$ has complex eigenvalues and satisfies
			$(2_t)$ or $(3_t)$ in Table~\ref{table2}, or $(1.2_t)$ or $(1.3_t)$ in Table~\ref{table4} (a.e.p.).
		\end{itemize}

%		\item[(ii)] the origin is a non-hyperbolic equilibrium point  when the next statements occurs:
%		
%		\begin{itemize}
%			\item[($ii_1$)] if $D Z^R_{\varepsilon, \eta}(0)$ has real eigenvalues and satisfies 
%			$(1_d)$ in  Table \ref{table1} with $\xi =0$ or
%			$(1.1_d)$ in Table \ref{table3} with $ det[D G](0) \, \varphi^\prime \, \psi^\prime =0$;
%			
%			\item[($ii_2$)] if $D Z^R_{\varepsilon, \eta}(0)$ has complex eigenvalues and satisfies 
%			$(1_t)$ in Table \ref{table2} with $\xi =0$ or
%			$(1.1_t)$ in  Table \ref{table4}.
%			
%		\end{itemize}
		
		\item[(ii)] 
%		The origin is a non-hyperbolic equilibrium point when the following conditions occur:
	The origin is a non-hyperbolic equilibrium point under the following conditions:
		\begin{itemize}
			\item[($ii_1$)] If $D Z^R_{\varepsilon,\eta}(0)$ has real eigenvalues and satisfies 
			$(1_d)$ in Table~\ref{table1} with $\xi = 0$, or 
			$(1.1_d)$ in Table~\ref{table3} with $ det[D G](0) \, \varphi^\prime \, \psi^\prime =0$;
			
			\item[($ii_2$)] If $D Z^R_{\varepsilon,\eta}(0)$ has complex eigenvalues and satisfies 
			$(1_t)$ in Table~\ref{table2} with $\xi = 0$, or 
			$(1.1_t)$ in Table~\ref{table4}.
		\end{itemize}
		
	\end{itemize}
	
\end{prop}

\begin{table}[]
	\begin{tabular}{|c|lc|c|c|}
		\hline
		\multirow{6}{*} {\footnotesize$\dfrac{\partial \, det [Z] }{\partial \, x_1}   \, \varphi \, \psi^\prime  = 0$} & \multicolumn{1}{l|}{\multirow{4}{*}{\footnotesize  $\dfrac{\partial \, det [Z] }{\partial \, x_2}   \, \varphi^\prime \, \psi   = 0$}} & \multirow{2}{*}{$det_{0}=0$}& \multirow{2}{*}{\footnotesize $	det [ D Z^R_{\varepsilon, \eta}] =  \mathcal{O} ( \xi )$} & \multirow{2}{*}{$(1_d)$} \\
		& \multicolumn{1}{l|}{}                  &                   &                   &                   \\ \cline{3-5} 
		& \multicolumn{1}{l|}{}                  & \multirow{2}{*}{$det_{0} \neq 0$} & \multirow{8}{*}{\footnotesize $	det [ D Z^R_{\varepsilon, \eta}] \neq 0$} & \multirow{8}{*}{$(2_d)$} \\
		& \multicolumn{1}{l|}{}                  &                   &                   &                   \\ \cline{2-3}
		& \multicolumn{2}{l|}{\multirow{2}{*}{\footnotesize  $\dfrac{\partial \, det [Z] }{\partial \, x_2}   \, \varphi^\prime \, \psi   \neq 0$}}                     &                   &                   \\
		& \multicolumn{2}{l|}{}                                      &                   &                   \\ \cline{1-3}
		\multirow{6}{*}{ \footnotesize$\dfrac{\partial \, det [Z] }{\partial \, x_1}   \, \varphi \, \psi^\prime  \neq 0$} & \multicolumn{2}{l|}{\multirow{2}{*}{\footnotesize  $\dfrac{\partial \, det [Z] }{\partial \, x_2}   \, \varphi^\prime \, \psi   = 0$}}                     &                   &                   \\
		& \multicolumn{2}{l|}{}                                      &                   &                   \\ \cline{2-3}
		& \multicolumn{1}{l|}{\multirow{4}{*}{\footnotesize  $\dfrac{\partial \, det [Z] }{\partial \, x_2}   \, \varphi^\prime \, \psi   \neq 0$}} & \multirow{2}{*}{\footnotesize $\frac{\partial \, det [Z] }{\partial \, x_1}   \, \varphi \, \psi^\prime \, \frac{\partial \, det [Z] }{\partial \, x_2}  \, \varphi^\prime \, \psi  <0$} &                   &                   \\
		& \multicolumn{1}{l|}{}                  &                   &                   &                   \\ \cline{3-5} 
		& \multicolumn{1}{l|}{}                  & \multirow{2}{*}{\footnotesize $\frac{\partial \, det [Z] }{\partial \, x_1}   \, \varphi \, \psi^\prime \, \frac{\partial \, det [Z] }{\partial \, x_2}  \, \varphi^\prime \, \psi  >0$} & \multirow{2}{*}{\footnotesize $	det [ D Z^R_{\varepsilon, \eta}] \neq 0$ (a.e.p)} & \multirow{2}{*}{$(3_d)$} \\
		& \multicolumn{1}{l|}{}                  &                   &                   &                   \\ \hline
	\end{tabular}
	\vspace{0.3cm}
	\caption{%Cases of vanishing or not of the determinant.
	Cases of vanishing or non-vanishing of the determinant.}\label{table1}
\end{table}

\begin{table}[]
	\begin{tabular}{|c|lc|c|c|}
		\hline
		\multirow{6}{*}{$  \varphi \, \psi^\prime  = 0$} & \multicolumn{1}{l|}{\multirow{4}{*}{ $\varphi^\prime \, \psi   = 0$}} & \multirow{2}{*}{$X_2 \, tr[DY]-Y_2 \, tr[DX]=0$} & \multirow{2}{*}{ $tr[ D Z^R_{\varepsilon, \eta}] =  \mathcal{O} ( \xi )$} & \multirow{2}{*}{$(1_t)$} \\
		& \multicolumn{1}{l|}{}                  &                   &                   &                   \\ \cline{3-5} 
		& \multicolumn{1}{l|}{}                  & \multirow{2}{*}{$X_2 \, tr[DY]-Y_2 \, tr[DX]\neq0$} & \multirow{8}{*}{ $tr[ D Z^R_{\varepsilon, \eta}] \neq  0$} & \multirow{8}{*}{$(2_t)$} \\
		& \multicolumn{1}{l|}{}                  &                   &                   &                   \\ \cline{2-3}
		& \multicolumn{2}{l|}{\multirow{2}{*}{ $\varphi^\prime \, \psi   \neq 0$}}                     &                   &                   \\
		& \multicolumn{2}{l|}{}                                      &                   &                   \\ \cline{1-3}
		\multirow{6}{*}{$  \varphi \, \psi^\prime   \neq 0$} & \multicolumn{2}{l|}{\multirow{2}{*}{ $\varphi^\prime \, \psi   = 0$}}                     &                   &                   \\
		& \multicolumn{2}{l|}{}                                      &                   &                   \\ \cline{2-3}
		& \multicolumn{1}{l|}{\multirow{4}{*}{$\varphi^\prime \, \psi   \neq 0$}} & \multirow{2}{*}{  $  \varphi \, \psi^\prime \,  \varphi^\prime \, \psi  < 0$} &                   &                   \\
		& \multicolumn{1}{l|}{}                  &                   &                   &                   \\ \cline{3-5} 
		& \multicolumn{1}{l|}{}                  & \multirow{2}{*}{ $  \varphi \, \psi^\prime \,  \varphi^\prime \, \psi  > 0$} & \multirow{2}{*}{ $tr[ D Z^R_{\varepsilon, \eta}] \neq  0$ (a.e.p)} & \multirow{2}{*}{$(3_t)$} \\
		& \multicolumn{1}{l|}{}                  &                   &                   &                   \\ \hline
	\end{tabular}
	\vspace{0.3cm}
	\caption{%Cases of vanishing or not of the trace.
	Cases of vanishing or nonvanishing trace.}\label{table2}
\end{table}
% \usepackage{multirow}
\begin{table}[]
	\begin{tabular}{|llll|c|c|}
		\hline
		\multicolumn{1}{|l|}{\multirow{8}{*}{\footnotesize $det_{12}   \,  \psi^\prime=0$}} & \multicolumn{1}{l|}{\multirow{4}{*}{\footnotesize $det_{13}   \,  \varphi^\prime=0$}} & \multicolumn{2}{l|}{\multirow{2}{*}{\footnotesize $det_{11}    \,  \varphi^\prime \,  \psi^\prime=0$}}                     & \multirow{2}{*}{\footnotesize $	det [ D Z^R_{\varepsilon, \eta}] =  \mathcal{O} ( \xi^2 )$}  & \multirow{2}{*}{\footnotesize $(1.1_d)$}  \\
		\multicolumn{1}{|l|}{}                  & \multicolumn{1}{l|}{}                  & \multicolumn{2}{l|}{}                                      &                    &                    \\ \cline{3-6} 
		\multicolumn{1}{|l|}{}                  & \multicolumn{1}{l|}{}                  & \multicolumn{2}{l|}{\multirow{2}{*}{\footnotesize $det_{11}   \,  \varphi^\prime \,  \psi^\prime \neq 0$}}                     & \multirow{14}{*}{\footnotesize $	det [ D Z^R_{\varepsilon, \eta}] \neq 0 $} & \multirow{14}{*}{$(1.2_d)$} \\
		\multicolumn{1}{|l|}{}                  & \multicolumn{1}{l|}{}                  & \multicolumn{2}{l|}{}                                      &                    &                    \\ \cline{2-4}
		\multicolumn{1}{|l|}{}                  & \multicolumn{1}{l|}{\multirow{4}{*}{\footnotesize $det_{13}   \,  \varphi^\prime \neq 0$}} & \multicolumn{2}{l|}{\multirow{2}{*}{\footnotesize $det_{11}    \,  \varphi^\prime \,  \psi^\prime=0$}}                     &                    &                    \\
		\multicolumn{1}{|l|}{}                  & \multicolumn{1}{l|}{}                  & \multicolumn{2}{l|}{}                                      &                    &                    \\ \cline{3-4}
		\multicolumn{1}{|l|}{}                  & \multicolumn{1}{l|}{}                  & \multicolumn{2}{l|}{\multirow{2}{*}{\footnotesize $det_{11}    \,  \varphi^\prime \,  \psi^\prime \, det_{13}   \,  \varphi^\prime <0$}}                     &                    &                    \\
		\multicolumn{1}{|l|}{}                  & \multicolumn{1}{l|}{}                  & \multicolumn{2}{l|}{}                                      &                    &                    \\ \cline{1-4}
		\multicolumn{1}{|l|}{\multirow{4}{*}{\footnotesize $det_{12}   \,  \psi^\prime \neq 0$}} & \multicolumn{1}{l|}{\multirow{4}{*}{\footnotesize $det_{13}   \,  \varphi^\prime = 0$}} & \multicolumn{2}{l|}{\multirow{2}{*}{\footnotesize $det_{11}    \,  \varphi^\prime \,  \psi^\prime=0$}}                     &                    &                    \\
		\multicolumn{1}{|l|}{}                  & \multicolumn{1}{l|}{}                  & \multicolumn{2}{l|}{}                                      &                    &                    \\ \cline{3-4}
		\multicolumn{1}{|l|}{}                  & \multicolumn{1}{l|}{}                  & \multicolumn{2}{l|}{\multirow{2}{*}{\footnotesize $det_{11}    \,  \varphi^\prime \,  \psi^\prime \, det_{12}   \,  \psi^\prime > 0$}}                     &                    &                    \\
		\multicolumn{1}{|l|}{}                  & \multicolumn{1}{l|}{}                  & \multicolumn{2}{l|}{}                                      &                    &                    \\ \cline{1-4}
		\multicolumn{2}{|l|}{\multirow{6}{*}{\footnotesize $det_{12}   \,  \psi^\prime \, det_{13}   \,  \varphi^\prime <0$ }}                                          & \multicolumn{2}{l|}{\multirow{2}{*}{\footnotesize $det_{11}    \,  \varphi^\prime \,  \psi^\prime=0$}}                     &                    &                    \\
		\multicolumn{2}{|l|}{}                                                           & \multicolumn{2}{l|}{}                                      &                    &                    \\ \cline{3-4}
		\multicolumn{2}{|l|}{}                                                           & \multicolumn{1}{l|}{\multirow{4}{*}{\footnotesize $det_{11}    \,  \varphi^\prime \,  \psi^\prime \neq 0$}} & \multirow{2}{*}{\footnotesize $det_{11}    \,  \varphi^\prime \,  \psi^\prime \, det_{12}   \,  \psi^\prime > 0$} &                    &                    \\
		\multicolumn{2}{|l|}{}                                                           & \multicolumn{1}{l|}{}                  &                   &                    &                    \\ \cline{4-6} 
		\multicolumn{2}{|l|}{}                                                           & \multicolumn{1}{l|}{}                  & \multirow{2}{*}{\footnotesize $det_{11}    \,  \varphi^\prime \,  \psi^\prime \, det_{12}   \,  \psi^\prime < 0$} & \multirow{8}{*}{}  & \multirow{8}{*}{}  \\
		\multicolumn{2}{|l|}{}                                                           & \multicolumn{1}{l|}{}                  &                   &                    &                    \\ \cline{1-4}
		\multicolumn{4}{|l|}{\multirow{2}{*}{\footnotesize $det_{12}   \,  \psi^\prime \, det_{13}   \,  \varphi^\prime >0$}}                                                                                                       &                    &                    \\
		\multicolumn{4}{|l|}{}                                                                                                                        &                    &                    \\ \cline{1-4}
		\multicolumn{1}{|l|}{\multirow{2}{*}{\footnotesize $det_{12}   \,  \psi^\prime\neq 0$}} & \multicolumn{1}{l|}{\multirow{2}{*}{\footnotesize $det_{13}   \,  \varphi^\prime=0$}} & \multicolumn{2}{l|}{\multirow{2}{*}{\footnotesize $det_{11}    \,  \varphi^\prime \,  \psi^\prime \, det_{12}   \,  \psi^\prime < 0$}}                     &     \footnotesize    $	det [ D Z^R_{\varepsilon, \eta}] \neq 0$ (a.e.p)           &   \footnotesize    $(1.3_d)$             \\
		\multicolumn{1}{|l|}{}                  & \multicolumn{1}{l|}{}                  & \multicolumn{2}{l|}{}                                      &                    &                    \\ \cline{1-4}
		\multicolumn{1}{|l|}{\multirow{2}{*}{\footnotesize $det_{12}   \,  \psi^\prime=0$}} & \multicolumn{1}{l|}{\multirow{2}{*}{\footnotesize $det_{13}   \,  \varphi^\prime \neq 0$}} & \multicolumn{2}{l|}{\multirow{2}{*}{\footnotesize $det_{11}    \,  \varphi^\prime \,  \psi^\prime \, det_{13}   \,  \varphi^\prime > 0$}}                     &                    &                    \\
		\multicolumn{1}{|l|}{}                  & \multicolumn{1}{l|}{}                  & \multicolumn{2}{l|}{}                                      &                    &                    \\ \hline
	\end{tabular}
	\vspace{0.3cm}
	\caption{%Cases of vanishing or not of the determinant which depends on $\xi$.
	Cases in which the determinant, which depends on $\xi$, vanishes or does not vanish.
	}\label{table3}
\end{table}
% \usepackage{multirow}
\begin{table}[]
	\begin{tabular}{|l|ll|c|c|}
		\hline
		\multirow{4}{*}{$\dfrac{\partial \, G_1}{\partial \, W} \,  \varphi^\prime=0$} & \multicolumn{2}{l|}{\multirow{2}{*}{\footnotesize$\dfrac{\partial \, G_2}{\partial \, T}   \,  \psi^\prime = 0 $}}                     & \multirow{2}{*}{$	tr [ D Z^R_{\varepsilon, \eta}] =0$} & \multirow{2}{*}{$(1.1_t)$} \\
		& \multicolumn{2}{l|}{}                                      &                   &                   \\ \cline{2-5} 
		& \multicolumn{2}{l|}{\multirow{2}{*}{\footnotesize $\dfrac{\partial \, G_2}{\partial \, T}   \,  \psi^\prime \neq 0 $}}                     & \multirow{6}{*}{} & \multirow{6}{*}{} \\
		& \multicolumn{2}{l|}{}                                      &                   &                   \\ \cline{1-3}
		\multirow{6}{*}{$ \dfrac{\partial \, G_1}{\partial \, W} \,  \varphi^\prime \neq 0$} & \multicolumn{2}{l|}{\multirow{2}{*}{\footnotesize$\dfrac{\partial \, G_2}{\partial \, T}   \,  \psi^\prime = 0 $}}                     &      $	tr[ D Z^R_{\varepsilon, \eta}] \neq 0$       &     $(1.2_t)$              \\
		& \multicolumn{2}{l|}{}                                      &                   &                   \\ \cline{2-3}
		& \multicolumn{1}{l|}{\multirow{4}{*}{\footnotesize$\dfrac{\partial \, G_2}{\partial \, T}   \,  \psi^\prime \neq 0 $}} & \multirow{2}{*}{\footnotesize $\dfrac{\partial \, G_1}{\partial \, W} \,  \varphi^\prime \, \dfrac{\partial \, G_2}{\partial \, T}   \,  \psi^\prime >0$} &                   &                   \\
		& \multicolumn{1}{l|}{}                  &                   &                   &                   \\ \cline{3-5} 
		& \multicolumn{1}{l|}{}                  & \multirow{2}{*}{\footnotesize $\dfrac{\partial \, G_1}{\partial \, W} \,  \varphi^\prime \, \dfrac{\partial \, G_2}{\partial \, T}   \,  \psi^\prime <0$} & \multirow{2}{*}{$	tr [ D Z^R_{\varepsilon, \eta}] \neq 0$ (a.e.p)} & \multirow{2}{*}{$(1.3_t)$} \\
		& \multicolumn{1}{l|}{}                  &                   &                   &                   \\ \hline
	\end{tabular}
	\vspace{0.3cm}
	\caption{%Cases of vanishing or not of the trace which depends on $\xi$.
	Cases in which the trace, which depends on $\xi$, vanishes or does not vanish.
	}\label{table4}
\end{table}

\begin{remark}
%	Observe that in the cases of Table \ref{table3} and \ref{table4} in the Proposition \ref{propcaseA0B1} the parameter $\xi$ must be null. Moreover, in the cases for (a.e.p) is always possible to choose $\xi$ that does not vanish the respective expression.
%	
Observe that, in the cases of Tables~\ref{table3} and~\ref{table4} in Proposition~\ref{propcaseA0B1}, the parameter $\xi$ must be zero. Moreover, in the cases of (a.e.p.), it is always possible to choose $\xi$ so that the respective expression does not vanish.
	
\end{remark}

\section{Examples of regularization  %of $\Sigma$-structurally stable vector fields
}\label{section4}

%At first, we will make some remarks.  Using Theorem \ref{teocod0} we have four classes of the systems $\Sigma$-structurally stable on $\Omega$,  consequently, so we had to analyze each one of those classes to prove our Theorem \ref{teo1}.

We begin with some remarks. 
Using Theorem~\ref{teocod0}, we identify four classes of systems that are $\Sigma$-structurally stable on $\Omega$. 
Consequently, we need to analyze each of these classes in order to prove Theorem~\ref{teo1}.

There are no restrictions on the transition functions, as long as they satisfy their defining conditions; that is, they must be at least continuous and satisfy the conditions in~\eqref{condfuncaotransi}. 
Thus, our next step is to apply the nonlinear double regularization~\eqref{eq2par1} to the vector fields in Figure~\ref{exemplos_inclusao_filippov}, always highlighting the trajectory through the origin. 
We will use different types of transition functions to illustrate Theorem~\ref{teo1} and to understand how this smoothing process occurs. 
First, let us consider two transition functions of Sotomayor–Teixeira for applying the nonlinear double regularization in case~(A) of Figure~\ref{exemplos_inclusao_filippov}. 
More precisely, consider the following functions:
%So our next step will be to apply the Sotomayor-Teixeira double regularization generalized 
\begin{equation}\label{casoAcod0}
	\varphi(x) = \left\{\begin{array}{r}
		\sgn(x), \; \; \mbox{if} \; \;  |x| \geq 1, \\[6pt]
		\; \; \; \; \; \; \; \; x, \; \; \mbox{if} \; \; |x| \leq 1,
	\end{array} \right.
	\; \; \mbox{and} \; \; \psi(y) = \left\{\begin{array}{r}
	   \; \; \sgn(y), \; \; \mbox{if} \; \;  |y| \geq 1, \\[6pt]
		 3 \frac{y}{2}-\frac{y^3}{2}, \; \; \mbox{if} \; \; |y| \leq 1.
	\end{array} \right.
\end{equation}
%Here $\varphi$ is a $C^0$ or continuous function with the graph in Figure \ref{transicoes} (a) being smooth in a neighborhood of the origin, and $\psi$ is a $C^1$ function with the graph in Figure \ref{transicoes} (b). In this regularization, we have a concrete structurally stable example without equilibrium points of what happens in Figure \ref{figuraregdulpa}, see the next Figure \ref{regcod0A}.
Here, $\varphi$ is a $C^0$ (continuous) function whose graph in Figure~\ref{transicoes}(a) is smooth in a neighborhood of the origin, 
and $\psi$ is a $C^1$ function with the graph shown in Figure~\ref{transicoes}(b). 
This regularization provides a concrete, structurally stable example without equilibrium points, illustrating the situation depicted in Figure~\ref{figuraregdulpa}, see also Figure~\ref{regcod0A}.
\begin{figure}[!htb]

	\begin{center}
		
		\begin{overpic}[scale=0.3,tics=5]{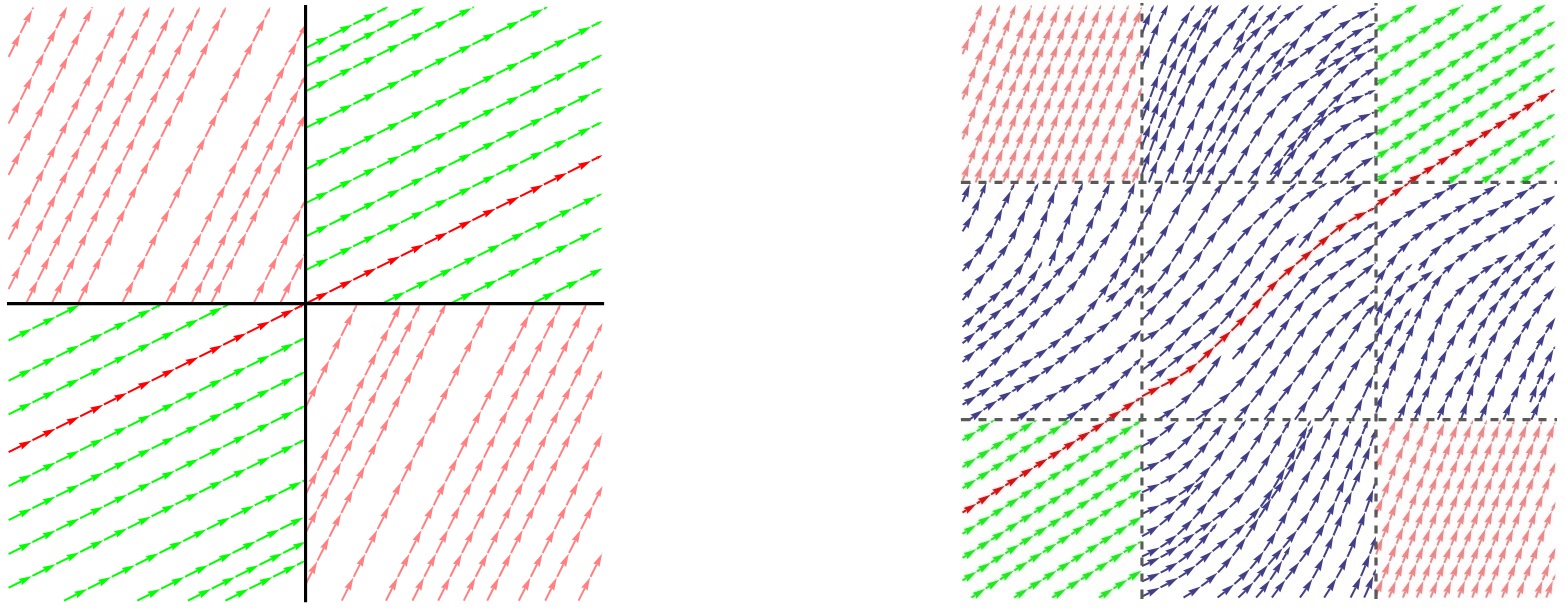}
			\vspace{0.2cm}
			
			%			\put(80,342){{(A)}}
%			\put(80,168){{(A)}}
%			
%			\put(295,168){{(B)}}
%			
%			\put(80,-13){{(C1)}}
%			
%			\put(295,-13){{(C2)}}
			%campo1
%			\put(6,300){{\parbox{0.55\linewidth}{$ 
%						\left[ \begin{array}{cc}
%							1 \\ 
%							2  \\ 
%						\end{array} \right]$}}}
%			\put(145,300){{\parbox{0.5\linewidth}{$\left[ \begin{array}{cc}
%							2 \\ 
%							1  \\ 
%						\end{array} \right]$}}}
%			\put(6,205){{\parbox{0.5\linewidth}{$\left[ \begin{array}{cc}
%							2 \\ 
%							1  \\ 
%						\end{array} \right]$}}}
%			\put(145,205){{\parbox{0.5\linewidth}{$ \left[ \begin{array}{cc}
%							1 \\ 
%							2  \\ 
%						\end{array} \right]$}}}
%

\put(165,75){{$\longrightarrow$}}		\put(165,90){{$Z_{\varepsilon, \eta}^R$}}	
		\end{overpic}

	\end{center}
	
	$\;$
	\caption{
%		Vector Field $Z$ given in Figure \ref{exemplos_inclusao_filippov} (A) and its nonlinear double regularization $Z_{\varepsilon, \eta}^R$ for $\varepsilon=0.015$, $\eta=0.012$, $\xi=0$ and transition functions in \eqref{casoAcod0}.  
	Vector field $Z$ is shown in Figure~\ref{exemplos_inclusao_filippov}(A), together with its nonlinear double regularization $Z_{\varepsilon, \eta}^R$ for $\varepsilon = 0.015$, $\eta = 0.012$, $\xi = 0$, and the transition functions given in~\eqref{casoAcod0}.
	}\label{regcod0A}
	
\end{figure}

%To exemplify other situations that can occur in the regularization process considered, we choose two transition functions that are not Sotomayor-Teixeira type and apply the nonlinear double regularization \eqref{eq2par1}  in case (B) of Figure \ref{exemplos_inclusao_filippov}. Let us consider a monotone transition function and another transition function that is not monotone, respectively, given by

To illustrate other situations that may occur in the considered regularization process, we choose two transition functions that are not of Sotomayor–Teixeira type and apply the nonlinear double regularization~\eqref{eq2par1} in case (B) of Figure~\ref{exemplos_inclusao_filippov}. 
Specifically, we consider one monotone transition function and another that is not monotone, respectively, given by
\begin{equation}\label{casoBcod0}
	\varphi_B(x) = \left\{\begin{array}{l}
		-\frac{1}{x}-1, \; \; \mbox{if} \; \;  x \leq -1/2, \\[6pt]
		\phantom{-\frac{1}{x}-} \; \; 1, \; \; \mbox{if} \; \; x \geq -1/2,
	\end{array} \right.
	\; \; \mbox{and} \; \; \psi(y) = \left\{\begin{array}{r}
		 	\sgn(y), \; \; \mbox{if} \; \;  |y| \geq 1, \\[6pt]
		 \; \widetilde{\psi}(y), \; \; \mbox{if} \; \; |y| \leq 1,
	\end{array} \right.
\end{equation}
%where $\widetilde{\psi}(y)=1/10 - (349/576) \, y - y^2/5 + (1069/288) \, y^3 + y^4/10 - (1213/576) \, y^5$. The graph of function $\varphi_B$ can be seen in Figure \ref{regcod0B} and the graph of $\psi$ in Figure \ref{transicoes} (c), where both functions are smooth in a neighborhood of the origin. For this case, the regularized vector field is structurally stable  without equilibrium points and it is an example that the vector fields $X_{\pm,\pm}$, appear only for points with a big norm, see Figure \ref{regcod0B}, that is,
where 
\[
\widetilde{\psi}(y) = \frac{1}{10} - \frac{349}{576}\,y - \frac{y^2}{5} + \frac{1069}{288}\,y^3 + \frac{y^4}{10} - \frac{1213}{576}\,y^5.
\] 
The graph of the function $\varphi_B$ is shown in Figure~\ref{regcod0B}, and the graph of $\psi$ is shown in Figure~\ref{transicoes}(c), where both functions are smooth in a neighborhood of the origin. 
In this case, the regularized vector field is structurally stable and has no equilibrium points. 
It also provides an example in which the vector fields $X_{\pm,\pm}$ appear only at points with big norm; see Figure~\ref{regcod0B}, that is,
%For this case, we have an example without equilibrium points that the vector fields $X_{\pm,\pm}$, sometimes, appear only for points with a big norm, see Figure \ref{regcod0B}, that is,
\begin{equation*}
\begin{array}{rrr}
		Z^R_{\varepsilon,+} \; \mbox{if} \; x_1 \leq -\varepsilon/2, \;  x_2 \geq \eta; & Z^R_{\varepsilon,\eta} \; \mbox{if} \; x_1 \leq -\varepsilon/2, \; |x_2| \leq \eta; & Z^R_{\varepsilon,-} \; \mbox{if} \; x_1 \leq -\varepsilon/2, \; x_2 \leq - \eta; \\[6pt]
		X \; \mbox{if} \; x_1 \geq -\varepsilon/2, \;  x_2 \geq \eta; & Z^R_{+,\eta} \; \mbox{if} \; x_1 \geq -\varepsilon/2, \; |x_2| \leq \eta; &Y \; \mbox{if} \; x_1 \geq -\varepsilon/2, \; x_2 \leq - \eta. 
	\end{array}
	\end{equation*}
\begin{figure}[!htb]

	\begin{center}
		
		\begin{overpic}[scale=0.31,tics=5]{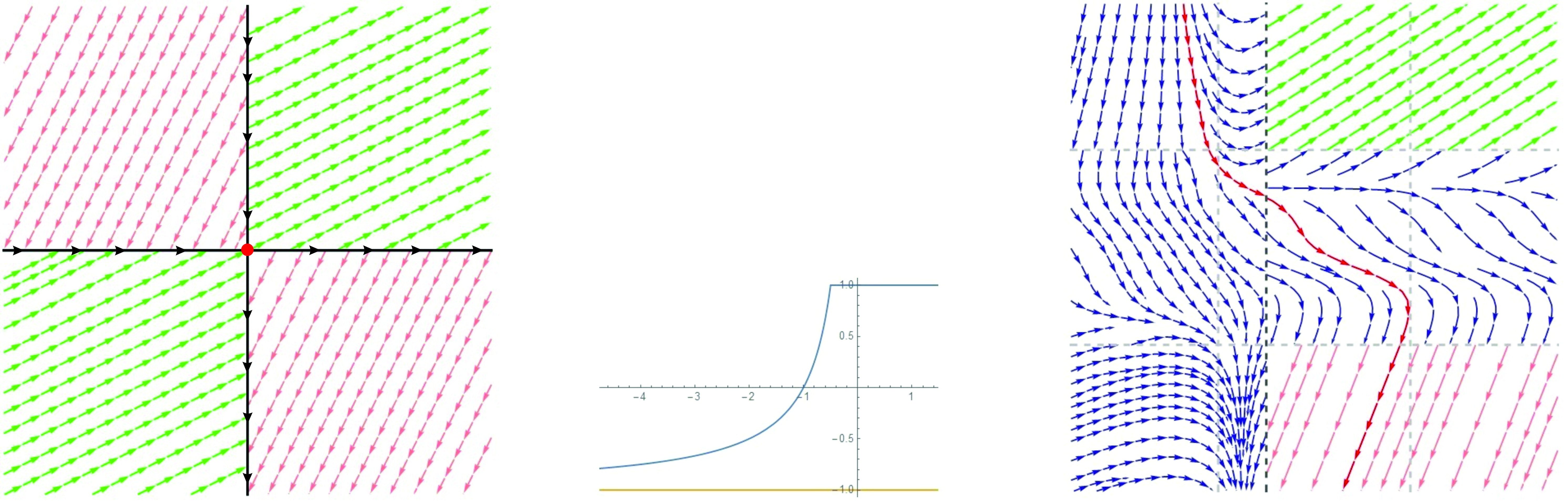}
			\vspace{0.2cm}
			
			%			\put(80,342){{(A)}}
			%			\put(80,168){{(A)}}
			%			
			%			\put(295,168){{(B)}}
			%			
			%			\put(80,-13){{(C1)}}
			%			
			%			\put(295,-13){{(C2)}}
			%campo1
			%			\put(6,300){{\parbox{0.55\linewidth}{$ 
			%						\left[ \begin{array}{cc}
			%							1 \\ 
			%							2  \\ 
			%						\end{array} \right]$}}}
			%			\put(145,300){{\parbox{0.5\linewidth}{$\left[ \begin{array}{cc}
			%							2 \\ 
			%							1  \\ 
			%						\end{array} \right]$}}}
			%			\put(6,205){{\parbox{0.5\linewidth}{$\left[ \begin{array}{cc}
			%							2 \\ 
			%							1  \\ 
			%						\end{array} \right]$}}}
			%			\put(145,205){{\parbox{0.5\linewidth}{$ \left[ \begin{array}{cc}
			%							1 \\ 
			%							2  \\ 
			%						\end{array} \right]$}}}
			%
			
			\put(210,95){{$\longrightarrow$}}		\put(210,110){{$Z_{\varepsilon, \eta}^R$}}
			
			\put(205,45){{$\varphi_B$}}	
		\end{overpic}

	\end{center}
	
	$\;$
	\caption{
%		Vector Field $Z$ given in Figure \ref{exemplos_inclusao_filippov} (B),  and its nonlinear double regularization $Z_{\varepsilon, \eta}^R$ for $\varepsilon=0.015$, $\eta=0.012$, $\xi=0$  and transition functions in \eqref{casoBcod0}.
%	
	Vector field $Z$ is shown in Figure~\ref{exemplos_inclusao_filippov}(B), together with its nonlinear double regularization $Z_{\varepsilon, \eta}^R$ for $\varepsilon = 0.015$, $\eta = 0.012$, $\xi = 0$, and the transition functions given in~\eqref{casoBcod0}.
	}\label{regcod0B}
	
\end{figure}

%For the nonlinear double regularization of Figure \ref{exemplos_inclusao_filippov} ($C_{01}$), we will choose transition functions that the vector fields $X_{\pm,\pm}$ appear only for points with a big norm. Consider the transition functions

For the nonlinear double regularization shown in Figure~\ref{exemplos_inclusao_filippov} ($C_{01}$), we choose transition functions such that the vector fields $X_{\pm,\pm}$ appear only at points with big norm. Consider the transition functions
\begin{equation}\label{casoC1cod0}
	\varphi(x) = \frac{x}{\sqrt{x^2+1}},
\; \mbox{and} \;  \psi_{C1}(y) = \left\{\begin{array}{r}
	 	\sgn(y), \; \mbox{if} \;  |y| \geq 1, \\[6pt]
	\widetilde{\psi}_{C_{01}}(y), \; \mbox{if} \; y \in [0,1], \\[6pt]
 y, \; \mbox{if} \;  y \in [-1,0], 
			\end{array} \right.
\end{equation}
%where $\widetilde{\psi}_{C_{01}}(y)= (3/2)\, y -y^3/2$. We can see the graph of $\varphi(x)$ in Figure \ref{transicoes} (d) and the graph of $\psi_{C1}(y)$ in Figure \ref{regcod0C1}. We also have regularized vector field is structurally stable  without equilibrium points. In addition, as said previously, we will only have the following vector fields in this nonlinear double regularization:
where $\widetilde{\psi}_{C_{01}}(y)= (3/2)\, y -y^3/2$.  The graph of $\varphi(x)$ is shown in Figure~\ref{transicoes}(d), and the graph of $\psi_{C_{01}}(y)$ is shown in Figure~\ref{regcod0C1}. The resulting regularized vector field is structurally stable and has no equilibrium points. In addition, as mentioned previously, only the following vector fields appear in this nonlinear double regularization:
\begin{equation*}
	\begin{array}{rrr}
		Z^R_{\varepsilon,+} \; \mbox{if} \; x_2 \geq \eta; & Z^R_{\varepsilon,\eta} \; \mbox{if} \; |x_2| \leq \eta; & Z^R_{\varepsilon,-} \; \mbox{if} \;  x_2 \leq - \eta. 
	\end{array}
\end{equation*}
\begin{figure}[!htb]

	\begin{center}
		
		\begin{overpic}[scale=0.31,tics=5]{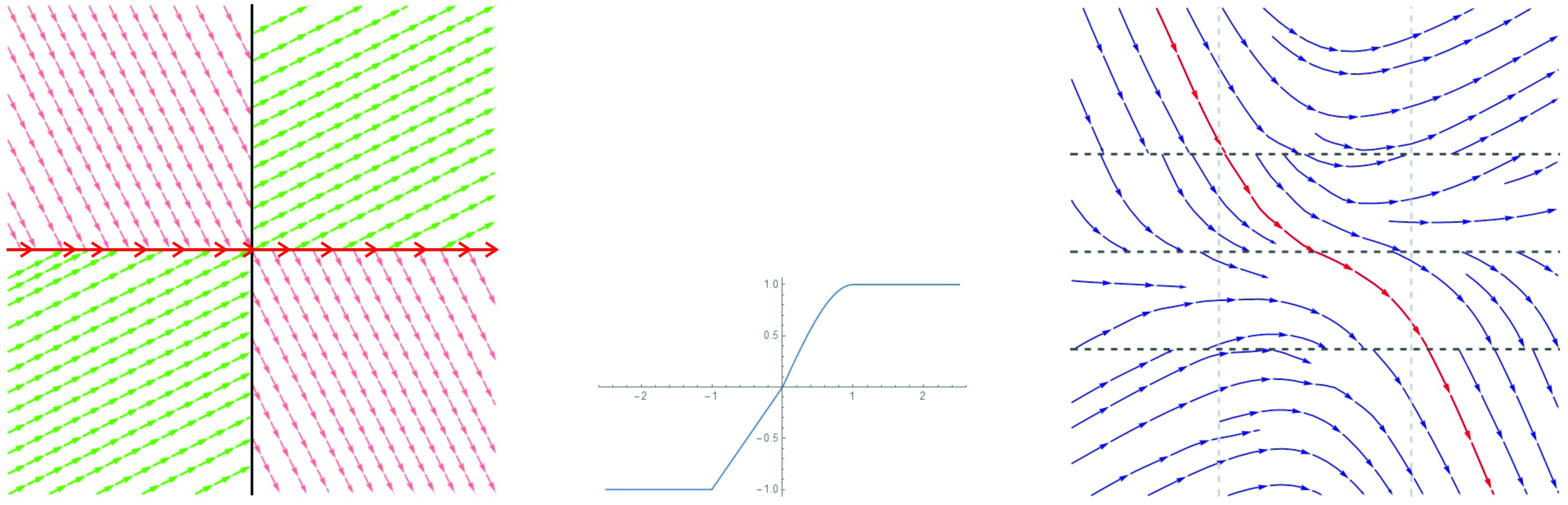}
			\vspace{0.2cm}
			
			%			\put(80,342){{(A)}}
			%			\put(80,168){{(A)}}
			%			
			%			\put(295,168){{(B)}}
			%			
			%			\put(80,-13){{(C1)}}
			%			
			%			\put(295,-13){{(C2)}}
			%campo1
			%			\put(6,300){{\parbox{0.55\linewidth}{$ 
			%						\left[ \begin{array}{cc}
			%							1 \\ 
			%							2  \\ 
			%						\end{array} \right]$}}}
			%			\put(145,300){{\parbox{0.5\linewidth}{$\left[ \begin{array}{cc}
			%							2 \\ 
			%							1  \\ 
			%						\end{array} \right]$}}}
			%			\put(6,205){{\parbox{0.5\linewidth}{$\left[ \begin{array}{cc}
			%							2 \\ 
			%							1  \\ 
			%						\end{array} \right]$}}}
			%			\put(145,205){{\parbox{0.5\linewidth}{$ \left[ \begin{array}{cc}
			%							1 \\ 
			%							2  \\ 
			%						\end{array} \right]$}}}
			%
			
			\put(215,95){{$\longrightarrow$}}		\put(215,110){{$Z_{\varepsilon, \eta}^R$}}
			
			\put(235,40){{$\psi_{C_{01}}$}}	
		\end{overpic}

	\end{center}
	
	$\;$
	\caption{
%		Vector Field $Z$ given in Figure \ref{exemplos_inclusao_filippov} ($C_{01}$),  and its nonlinear double regularization $Z_{\varepsilon, \eta}^R$ for $\varepsilon=0.015$, $\eta=0.012$, $\xi=0$  and transition functions in \eqref{casoC1cod0}.
%	
	Vector field $Z$ shown in Figure~\ref{exemplos_inclusao_filippov} ($C_{01}$) and its nonlinear double regularization $Z_{\varepsilon,\eta}^R$ with $\varepsilon=0.015$, $\eta=0.012$, $\xi=0$, and transition functions given in~\eqref{casoC1cod0}.
	}\label{regcod0C1}
	
\end{figure}

%Finishing the cases structurally stable in $\Omega$ (codimension zero) of Figure \ref{exemplos_inclusao_filippov} which preserve the codimension we will show the double regularization of the case ($C_{02}$) of Figure \ref{exemplos_inclusao_filippov}. We choose transition functions in a way that we will only have the vector field $Z^R_{\varepsilon,\eta}$, that is, we will always have the two transition functions in all regions considered in the plane, see Figure \ref{regcod0C2}. Therefore, for this purpose, we choose the following transition functions:

%Finishing the cases structurally stable in $\Omega$ (codimension zero) of Figure \ref{exemplos_inclusao_filippov} which preserve the codimension, we will show the double regularization of the case ($C_{02}$) of Figure \ref{exemplos_inclusao_filippov}. We choose transition functions in a way that we will only have the vector field $Z^R_{\varepsilon,\eta}$, that is, we will always have the vector field in all regions considered in the plane, see Figure \ref{regcod0C2}. Therefore, for this purpose, we will choose the following transition functions:

To conclude the structurally stable cases in $\Omega$ (codimension zero) shown in Figure~\ref{exemplos_inclusao_filippov} that preserve the codimension, we present the double regularization of case ($C_{02}$) in Figure~\ref{exemplos_inclusao_filippov}. We choose the transition functions so that only the vector field $Z^R_{\varepsilon,\eta}$ appears; that is, the same vector field is defined in all regions of the plane (see Figure~\ref{regcod0C2}). Therefore, for this purpose, we choose the following transition functions:
\begin{equation}\label{casoC2cod0}
	\varphi(x) = \frac{x}{\sqrt{x^2+1}},
	\; \; \mbox{and} \; \; \psi_{C_{02}}(y)= -\sgn(y) e^{-\sgn(y) y}+ \sgn(y).
\end{equation}
%The graph of $\psi_{C_{02}}(y)$ can be seen in Figure \ref{regcod0C2}.
The graph of $\psi_{C_{02}}(y)$ is shown in Figure~\ref{regcod0C2}.
\begin{figure}[!htb]

	\begin{center}
		
		\begin{overpic}[scale=0.31,tics=5]{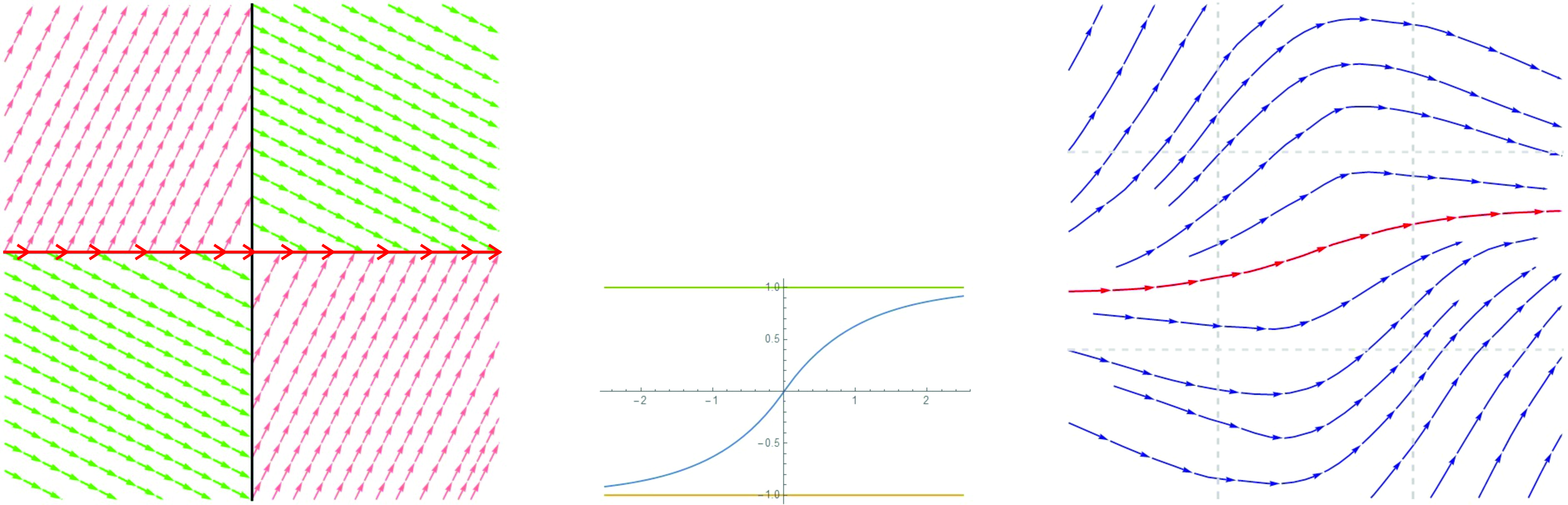}
			\vspace{0.2cm}
			
			%			\put(80,342){{(A)}}
			%			\put(80,168){{(A)}}
			%			
			%			\put(295,168){{(B)}}
			%			
			%			\put(80,-13){{(C1)}}
			%			
			%			\put(295,-13){{(C2)}}
			%campo1
			%			\put(6,300){{\parbox{0.55\linewidth}{$ 
			%						\left[ \begin{array}{cc}
			%							1 \\ 
			%							2  \\ 
			%						\end{array} \right]$}}}
			%			\put(145,300){{\parbox{0.5\linewidth}{$\left[ \begin{array}{cc}
			%							2 \\ 
			%							1  \\ 
			%						\end{array} \right]$}}}
			%			\put(6,205){{\parbox{0.5\linewidth}{$\left[ \begin{array}{cc}
			%							2 \\ 
			%							1  \\ 
			%						\end{array} \right]$}}}
			%			\put(145,205){{\parbox{0.5\linewidth}{$ \left[ \begin{array}{cc}
			%							1 \\ 
			%							2  \\ 
			%						\end{array} \right]$}}}
			%
			
			\put(215,95){{$\longrightarrow$}}		\put(215,110){{$Z_{\varepsilon, \eta}^R$}}
			
			\put(240,37){{$\psi_{C_{02}}$}}	
		\end{overpic}

	\end{center}
	
	$\;$
	\caption{
%		Vector Field $Z$ given in Figure \ref{exemplos_inclusao_filippov} ($C_{02}$),  and its nonlinear double regularization $Z_{\varepsilon, \eta}^R$ for $\varepsilon=0.015$, $\eta=0.012$, $\xi=0$  and transition functions in \eqref{casoC2cod0}. 
	Vector field $Z$ shown in Figure~\ref{exemplos_inclusao_filippov} ($C_{02}$) and its nonlinear double regularization $Z_{\varepsilon,\eta}^R$ with $\varepsilon=0.015$, $\eta=0.012$, $\xi=0$, and transition functions given in~\eqref{casoC2cod0}.
	}\label{regcod0C2}
	
\end{figure}

%Now we will show some examples in which the origin is an equilibrium point of the nonlinear double regularization of a structurally stable vector field $Z \in$ ($A_0$) $\subset \Omega_0$. The first case exhibit consists of an example of the Hopf bifurcation with codimension at least one \cite{MarMcC2012,Russ1998}. The Hopf bifurcation occurs when a complex conjugate pair of eigenvalues of the linearized system at an equilibrium point becomes purely imaginary occurring a birth of at least one limit cycle. The limit cycle birth is due to the variation of the trace of the matrix of the linearized system from zero to a number small enough.  

Now we present some examples in which the origin is an equilibrium point of the nonlinear double regularization of a structurally stable vector field $Z \in A_0 \subset \Omega_0$. The first example illustrates a Hopf bifurcation of codimension at least one \cite{MarMcC2012,Russ1998}. A Hopf bifurcation occurs when a complex conjugate pair of eigenvalues of the linearized system at an equilibrium point becomes purely imaginary, leading to the birth of at least one limit cycle. The limit cycle arises as the trace of the linearized system's matrix changes from zero to a small nonzero value.

%Inspired by the idea of the van der Pol oscillator and the discussions in Section \ref{sectioneqorigem} let us consider the piecewise vector field $Z=(X,Y)$ \eqref{sistemacruz} such that for $\mu \in \re$,

Inspired by the idea of the van der Pol oscillator and the discussions in Section \ref{sectioneqorigem}, let us consider the piecewise vector field $Z=(X,Y)$ in \eqref{sistemacruz}, for $\mu \in \mathbb{R}$,
\begin{equation}\label{sistemahopf}
X(x_1,x_2)= \left( \begin{array}{c}
-x_2 + (\mu - x_2^2) \,x_1 + 5/9\\ 
	\phantom{-}x_1 + 5/9 
\end{array}  \right) \qquad \mbox{and} \qquad
Y(x_1,x_2)= \left( \begin{array}{c}
	1 \\
	1
\end{array}  \right).
\end{equation}
% Considering the Proposition \ref{propcaseA0B1} then taking $\xi=0$, to the nonlinear double regularization of \eqref{sistemahopf} suffer a Hopf bifurcation at the origin in $\mu$ we will take the following transition functions
 Considering Proposition \ref{propcaseA0B1}, and taking $\xi=0$, the nonlinear double regularization of \eqref{sistemahopf} undergoes a Hopf bifurcation at the origin with respect to $\mu$. We will take the following transition functions:
 \begin{equation}\label{casoC1cod0}
 \varphi_1(x) = \left\{\begin{array}{r}
 \frac{1}{x}+1, \; \mbox{if} \;  x \geq 1, \\[6pt]
	\widetilde{\varphi}_1(x), \; \mbox{if} \; |x| \leq 1, \\[6pt]
 	-\frac{1}{x}-1, \; \mbox{if} \;  x \leq -1, 
 	 	\end{array} \right.
 	\; \; \mbox{and} \; \;  \psi_1(y) = \left\{\begin{array}{r}
 		1, \; \mbox{if} \;  y \geq 1, \\[6pt]
 		\widetilde{\psi}_1(y), \; \mbox{if} \; |y| \leq 1, \\[6pt]
 		-\frac{1}{y}-1, \; \mbox{if} \;  y\leq -1, 
 	\end{array} \right.
 \end{equation}
% such that its graphs are given in Figure \ref{graftranhopf},  $\widetilde{\varphi}_1(x)=2 - (3/2) \, x^2 + (5/2 ) \, x^3+ x^4/2 - (3 /2)\, x^5$ and $	\widetilde{\psi}_1(y)=7/4 - (9/4) \,  y^2 + y^3 + y^4 - y^5/2$. In this way, for $\xi = 0$, it follows that $Z^R_{\varepsilon, \eta}$ has an equilibrium point at the origin, whose eigenvalues of the linearized system are given by $9 (\mu \pm \sqrt{-4 + \mu^2})/8$. For $\mu = 0$, using the classical approach developed in~ \cite[Chap IX]{AndLeoGorMai1973}, it is possible to compute the first Lyapunov coefficient, which is given by
  such that their graphs are shown in Figure \ref{graftranhopf}, with
 $\widetilde{\varphi}_1(x)=2 - (3/2) \, x^2 + (5/2 ) \, x^3+ x^4/2 - (3 /2)\, x^5$ and $	\widetilde{\psi}_1(y)=7/4 - (9/4) \,  y^2 + y^3 + y^4 - y^5/2$.  
 In this way, for $\xi = 0$, it follows that $Z^R_{\varepsilon, \eta}$ has an equilibrium point at the origin, whose eigenvalues of the linearized system are given by
 \[
 \frac{9}{8} \bigl(\mu \pm \sqrt{-4 + \mu^2}\bigr).
 \]  
 For $\mu = 0$, using the classical approach developed in~\cite[Chap.~IX]{AndroLeontGor1973}, the first Lyapunov coefficient can be computed as
 \[
 u_3(2\pi) = \frac{\pi}{2916} \left( 
 -\frac{98 + 945 \varepsilon}{\varepsilon^4}
 - \frac{9 \, (-32 + 48 \eta + 81 \eta^4)}{\eta^4}
 \right).
 \]
%This coefficient does not vanish except along a curve in the space of regularization parameters. In other words, $u_3(2\pi) \neq 0$ for almost every point (a.e.p.), which implies that the origin is a weak focus. Therefore, for sufficiently small values of $\mu$, a Hopf bifurcation takes place (a.e.p), as illustrated in Figure~\ref{figbifhopf}.
This coefficient does not vanish except along a curve in the space of regularization parameters. In other words, $u_3(2\pi) \neq 0$ for almost every point (a.e.p.), implying that the origin is a weak focus. Therefore, for sufficiently small values of $\mu$, a Hopf bifurcation occurs for almost every point, as illustrated in Figure~\ref{figbifhopf}.

\begin{figure}[!htb]

	\begin{center}
		
		\begin{overpic}[scale=0.6,tics=5]{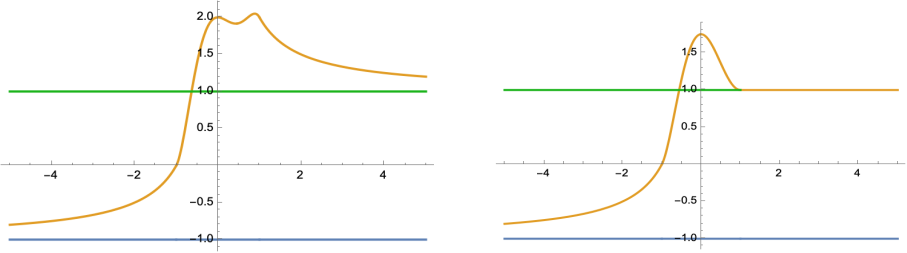}
			\vspace{0.2cm}
			
			%			\put(80,342){{(A)}}
			%			\put(80,168){{(A)}}
			%			
			%			\put(295,168){{(B)}}
			%			
			%			\put(80,-13){{(C1)}}
			%			
			%			\put(295,-13){{(C2)}}
			%campo1
			%			\put(6,300){{\parbox{0.55\linewidth}{$ 
						%						\left[ \begin{array}{cc}
							%							1 \\ 
							%							2  \\ 
							%						\end{array} \right]$}}}
			%			\put(145,300){{\parbox{0.5\linewidth}{$\left[ \begin{array}{cc}
							%							2 \\ 
							%							1  \\ 
							%						\end{array} \right]$}}}
			%			\put(6,205){{\parbox{0.5\linewidth}{$\left[ \begin{array}{cc}
							%							2 \\ 
							%							1  \\ 
							%						\end{array} \right]$}}}
			%			\put(145,205){{\parbox{0.5\linewidth}{$ \left[ \begin{array}{cc}
							%							1 \\ 
							%							2  \\ 
							%						\end{array} \right]$}}}
			%
			
			\put(350,90){{$\psi_1(y)$}}	
			
			\put(50,90){{$\varphi_1(x)$}}	
		\end{overpic}

	\end{center}
	
	$\;$
	\caption{
%		Nonmonotonous transition functions for the nonlinear double regularization of \eqref{sistemahopf} suffering a Hopf bifurcation.
	Nonmonotone transition functions for the nonlinear double regularization of \eqref{sistemahopf} undergoing a Hopf bifurcation.
	}\label{graftranhopf}
	
\end{figure}

\begin{figure}
	
	\begin{center}
		
		\begin{overpic}[scale=0.8,tics=1]{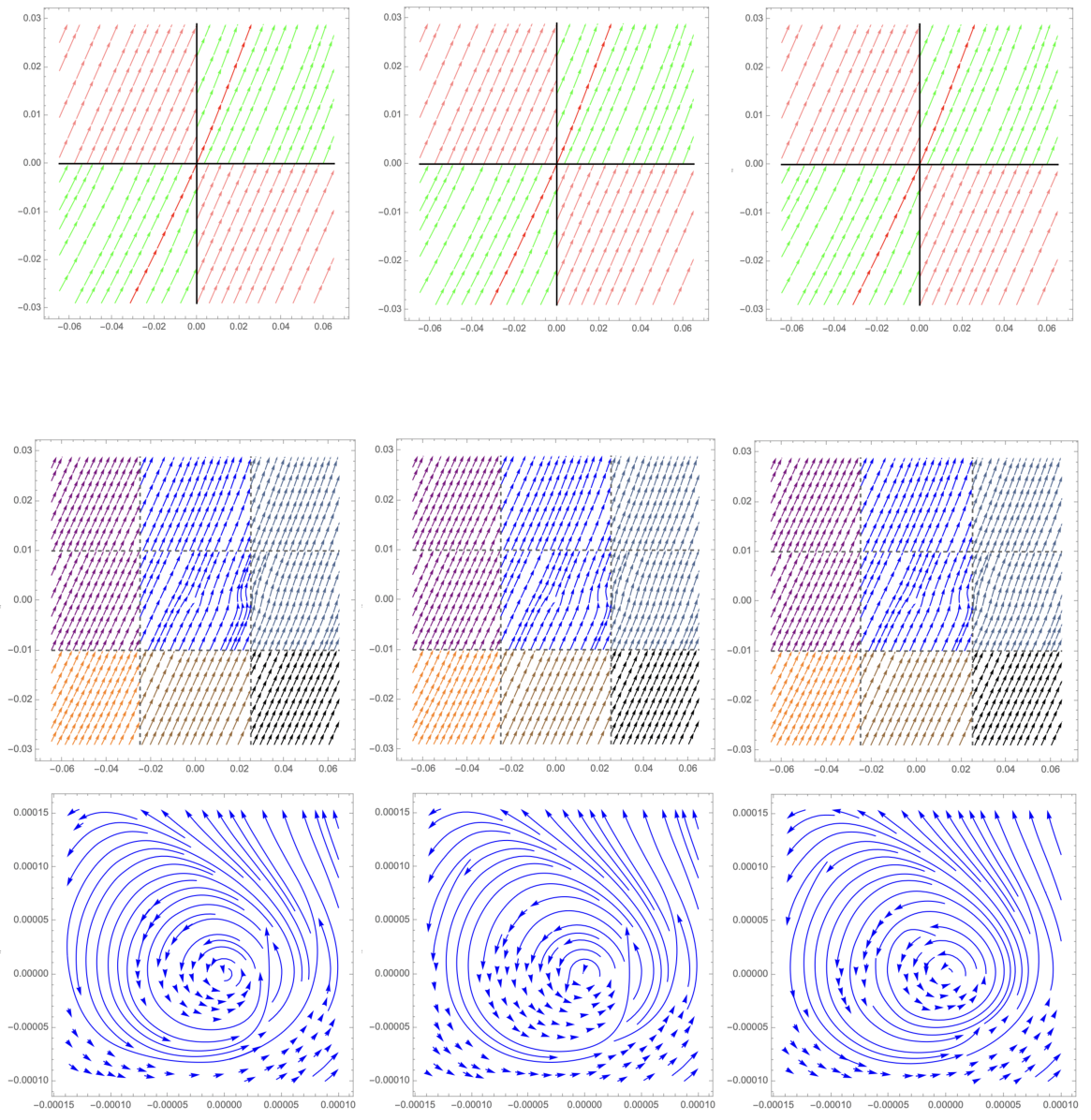}
			\vspace{0.2cm}
			
			%			\put(80,342){{(A)}}
			%			\put(80,168){{(A)}}
			%			
			%			\put(295,168){{(B)}}
			%			
			%			\put(80,-13){{(C1)}}
			%			
			%			\put(295,-13){{(C2)}}
			%campo1
			%			\put(6,300){{\parbox{0.55\linewidth}{$ 
						%						\left[ \begin{array}{cc}
							%							1 \\ 
							%							2  \\ 
							%						\end{array} \right]$}}}
			%			\put(145,300){{\parbox{0.5\linewidth}{$\left[ \begin{array}{cc}
							%							2 \\ 
							%							1  \\ 
							%						\end{array} \right]$}}}
			%			\put(6,205){{\parbox{0.5\linewidth}{$\left[ \begin{array}{cc}
							%							2 \\ 
							%							1  \\ 
							%						\end{array} \right]$}}}
			%			\put(145,205){{\parbox{0.5\linewidth}{$ \left[ \begin{array}{cc}
							%							1 \\ 
							%							2  \\ 
							%						\end{array} \right]$}}}
			%
			
			\put(227,300){{$\downarrow 	\; Z_{\varepsilon, \eta}^R$}}		
			
			\put(79,300){{$\downarrow 	\; Z_{\varepsilon, \eta}^R$}}		
			
			\put(376,300){{$\downarrow 	\; Z_{\varepsilon, \eta}^R$}}		
			
			\put(217,-12){{$\mu=0$}}		
			
			\put(42,-12){{$\mu=-1/100<0$}}

			\put(350,-12){{$\mu=1/100>0$}}		
			%	\put(240,37){{$\psi_{C_{02}}$}}	
		\end{overpic}

	\end{center}
	
	$\;$
	
	\caption{
%		Vector Field $Z$ given in \eqref{sistemahopf}  and its nonlinear double regularization $Z_{\varepsilon, \eta}^R$ suffering a Hopf bifurcation at origin in $\mu=0$ for $\varepsilon=0.025$, $\eta=0.01$, $\xi=0$  and transition functions in \eqref{graftranhopf}.
	Vector field $Z$ given in \eqref{sistemahopf} and its nonlinear double regularization $Z_{\varepsilon, \eta}^R$ undergoing a Hopf bifurcation at the origin for $\mu = 0$, with $\varepsilon = 0.025$, $\eta = 0.01$, $\xi = 0$, and transition functions as in \eqref{graftranhopf}.
	}\label{figbifhopf}
	
\end{figure}

% Once the Hopf bifurcation has codimension at least one, this example shows that structural stability in $\Omega$ is not necessarily preserved under the nonlinear double regularization \eqref{eq2par1}. 
% Furthermore, even for $\xi = 0$, equilibrium curves may exist and persist for all sufficiently small positive values of $\varepsilon$ and $\eta$. Consider the vector fields 
 
 Since the Hopf bifurcation has codimension at least one, this example shows that structural stability in $\Omega$ is not necessarily preserved under the nonlinear double regularization \eqref{eq2par1}.  
 Furthermore, even for $\xi = 0$, equilibrium curves may exist and persist for all sufficiently small positive values of $\varepsilon$ and $\eta$. Consider the vector fields
\begin{equation}\label{sistemacurvaeq}
	X(x_1,x_2)= \left( \begin{array}{c}
	\frac{1}{3}\\ [6pt]
\frac{1}{3}
	\end{array}  \right) \qquad \mbox{and} \qquad
	Y(x_1,x_2)= \left( \begin{array}{c}
		1 \\[6pt]
		1
	\end{array}  \right),
\end{equation}
%together with the transition function $\varphi_B(x)$ defined in \eqref{casoBcod0}, and
together with the transition function $\varphi_B(x)$ as defined in \eqref{casoBcod0}, and
 \begin{equation}\label{casocurva}
 	\psi_2(y) =
 	\begin{cases}
 		-\frac{1}{y} - 1, & \text{if } y \ge 1, \\[6pt]
 		-y^2 + x + 2, & \text{if } |y| \le 1, \\[6pt]
 		\frac{1}{y} + 1, & \text{if } y \le -1.
 	\end{cases}
 \end{equation}
%  Under this configuration, additional equilibrium curves may emerge during the regularization process, such as $x = 1/40$, 
% while others remain invariant for any smoothing parameter, for instance
% 
Under this configuration, additional equilibrium curves may emerge during the regularization process, such as $x = 1/40$,  
while others remain invariant for any smoothing parameter, for instance,
\begin{equation}\label{curvasdeeq}
 x = 0
 \quad \text{and} \quad
 y = \frac{1 + 20x - 800x^2}{100\,(-1 - 10x + 400x^2)},
\end{equation}
 as illustrated in Figure~\ref{figcurveq}.

 \begin{figure}[htbp]
	\centering
	\setlength{\tabcolsep}{4pt} % espaçamento horizontal entre colunas
	\renewcommand{\arraystretch}{1.0} % espaçamento vertical entre linhas
	
	\begin{tabular}{ccc}
		% Coluna 1 (abaixada)
		\begin{minipage}[t]{0.32\textwidth}
			\vspace{-3cm} % ↓ Ajuste vertical (aumente ou reduza conforme necessário)
			\begin{overpic}[width=\textwidth]{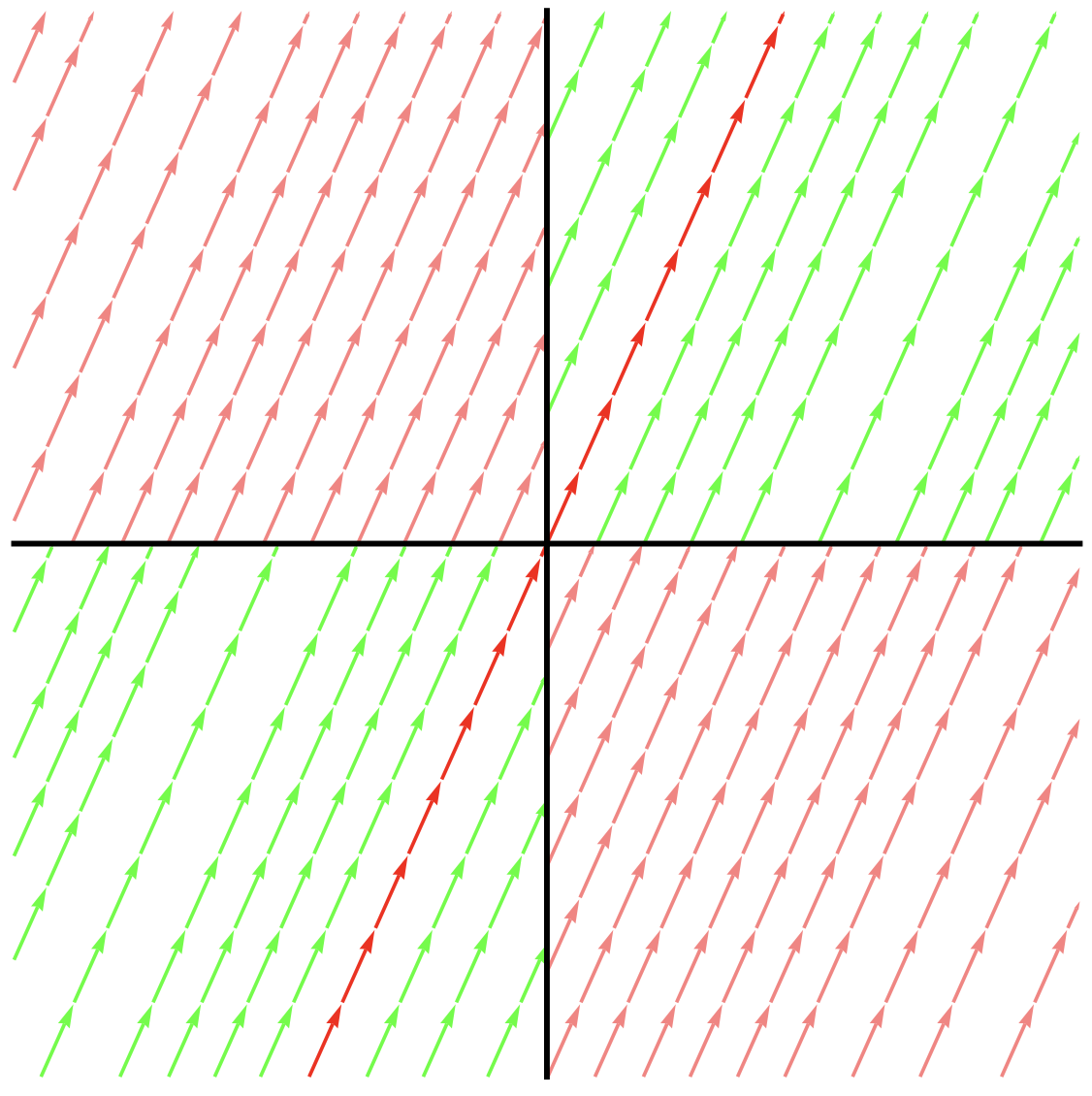}
				%	\put(95,-12){\small (a)}
			\end{overpic}
		\end{minipage} &
		% Coluna 2 (abaixada)
		\begin{minipage}[t]{0.25\textwidth}
			\vspace{-0.4cm} % ↓ Ajuste vertical
			\begin{overpic}[width=\textwidth]{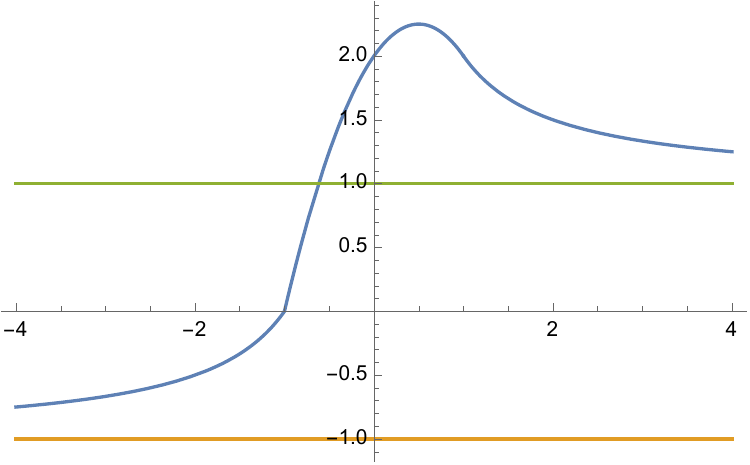}
				
				%	\put(95,-12){\small (b)}
				\put(45,90){{$\longrightarrow$}}	
				\put(45,105){{$Z_{\varepsilon, \eta}^R$}}
			\end{overpic}
		\end{minipage} &
		% Coluna 3 (duas figuras empilhadas)
		\begin{minipage}[t]{0.35\textwidth}
			\begin{overpic}[width=\textwidth]{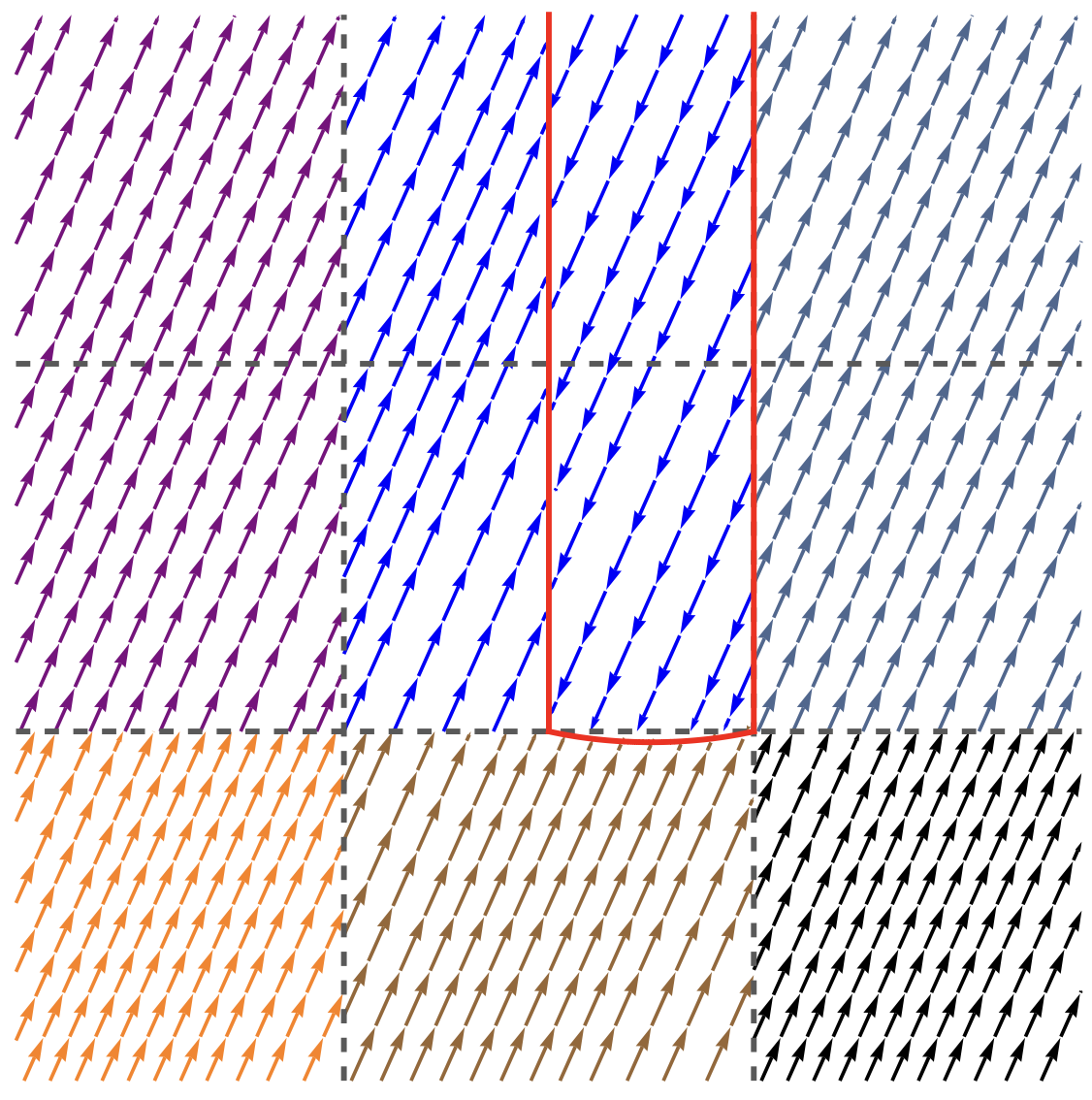}
				%	\put(90,90){\small (c)}
			\end{overpic}
			\vspace{0.6cm}
			\begin{overpic}[width=\textwidth]{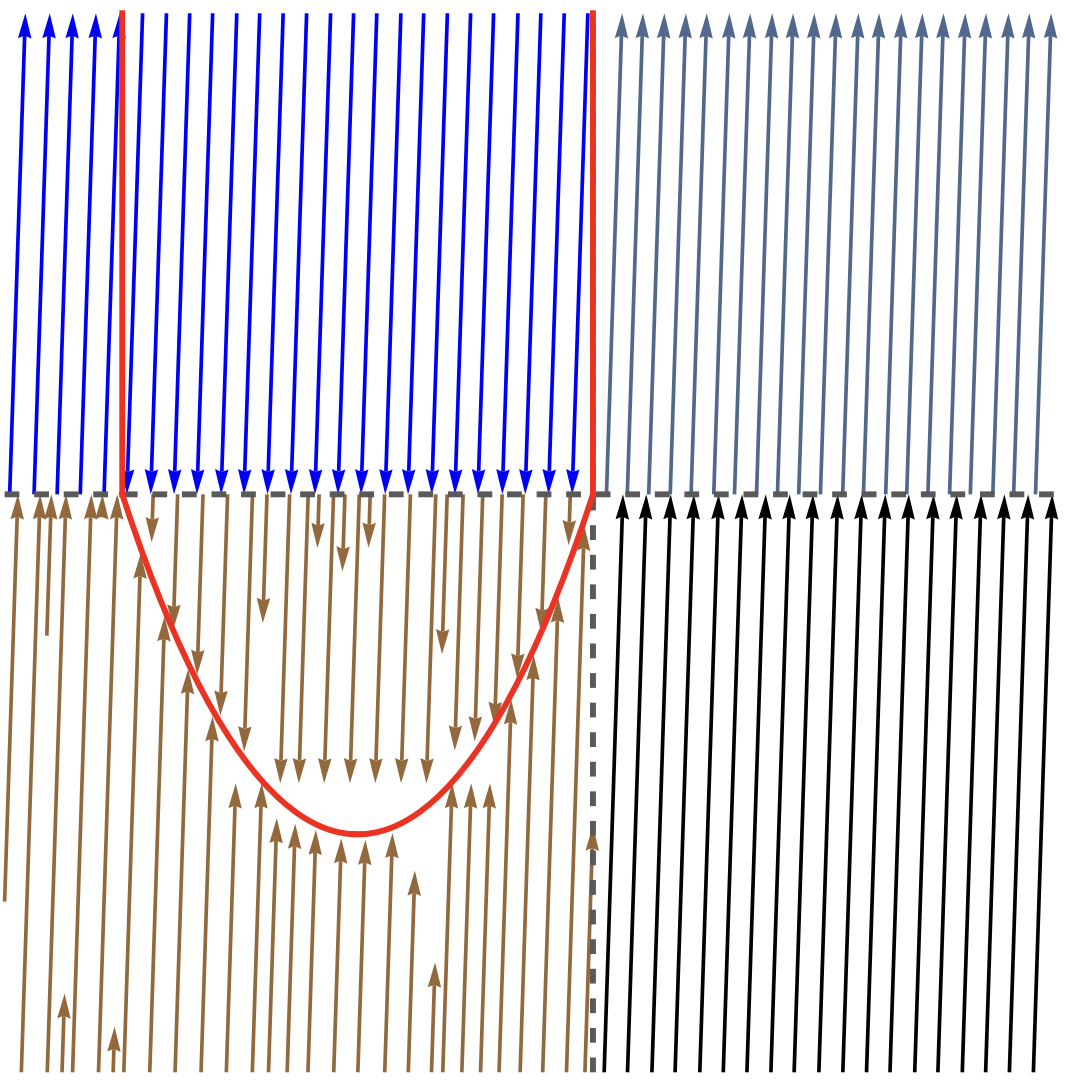}
				%	\put(90,90){\small (d)}
			\end{overpic}
		\end{minipage}
	\end{tabular}
	
	% $\;$
	
	\caption{
%		Vector field $Z$ given in \eqref{sistemacurvaeq}, and its nonlinear double regularization $Z_{\varepsilon, \eta}^R$ for $\varepsilon=0.025$, $\eta=0.01$, $\xi=0$, with the transition functions in \eqref{casocurva} and $\varphi_B$ given in \eqref{casoBcod0}, showing the equilibrium curves in \eqref{curvasdeeq} and $x = 1/40$, as well as the graph of the transition function in \eqref{casocurva}.
%		
	Vector field $Z$ given in \eqref{sistemacurvaeq}, and its nonlinear double regularization $Z_{\varepsilon, \eta}^R$ for $\varepsilon = 0.025$, $\eta = 0.01$, $\xi = 0$, with the transition functions in \eqref{casocurva} and $\varphi_B$ as given in \eqref{casoBcod0}, showing the equilibrium curves in \eqref{curvasdeeq} and at $x = 1/40$, as well as the graph of the transition function in \eqref{casocurva}.
	}\label{figcurveq}
\end{figure}

%So far, our examples have been constructed for $\xi = 0$. 
%We now turn our attention to cases with $\xi \neq 0$. 
%In this context, we consider a nonlinear double regularization 
%associated with a function $G$, defined by a homogeneous 
%polynomial in each coordinate, of the form

So far, our examples have been constructed for $\xi = 0$.  
We now turn our attention to cases with $\xi \neq 0$.  
In this context, we consider a nonlinear double regularization 
associated with a function $G$, defined as a homogeneous 
polynomial in each coordinate, of the form
$$ G(r,s) =\left( \begin{array}{l}
\displaystyle\sum_{i=0}^{m_{1}} a_{m_{1}-i, i} \, r^{m_{1}-i} \, s^{i} \\ [15pt]
\displaystyle\sum_{i=0}^{m_{2}}  b_{m_{2}-i,i}\, r^{\, m_{2}-i} \,  s^{ i}
\end{array}  \right), 
$$
%\[
%G(r,s) = \left(
%\sum_{i=0}^{m_{1}} a_{m_{1}-i, i} \, r^{m_{1}-i} \, s^{i} \; , \;
%\sum_{i=0}^{m_{2}}  b_{m_{2}-i,i}\, r^{\, m_{2}-i} \, s^{ i}
%\right),
%\]
with coefficients satisfying
\begin{align*}
	a_{m_1, 0} &= - \sum_{i=1}^{m_1} a_{m_1-i, i}, &
	a_{m_1-1, 1} &= - \frac{1}{2} \sum_{i=2}^{m_1} \left(1 + (-1)^{-i-1}\right) a_{m_1-i, i}, \\
	b_{m_2, 0} &= - \sum_{i=1}^{m_2} b_{m_2-i, i}, &
	b_{m_2-1, 1} &= - \frac{1}{2} \sum_{i=2}^{m_2} \left(1 + (-1)^{-i-1}\right) b_{m_2-i, i}.
\end{align*}
%For example, for $m_1 = 2$ and $m_2 = 3$,  
%$G(r,s) = (-r^2 + s^2, -r^2 s + s^3)$ satisfies the coefficient conditions. The nonlinear double regularization of \eqref{sistemahopf}, 
%with the transition functions in \eqref{casoC1cod0}, 
%then admits an equilibrium at the origin for each sufficiently small $\xi$, which also undergoes a Hopf bifurcation at $\mu=0$, as before. In the previous idea, we considered a homogeneous polynomial in each coordinate. 
%Since any polynomial of a given degree can be decomposed into a sum of homogeneous polynomials, 
%this approach can be extended to arbitrary polynomials with appropriate care.
%In the previous idea, we considered a homogeneous polynomial in each coordinate. 
%Since any polynomial of a given degree can be decomposed as a sum of homogeneous polynomials, 
%this approach can therefore be extended to arbitrary polynomials, provided that appropriate care is taken in accordance with \eqref{eq2par1}. For example, consider the continuous function
%
In the previous approach, we considered a homogeneous polynomial in each coordinate.  
Since any polynomial of a given degree can be decomposed as a sum of homogeneous polynomials,  
this approach can be extended to arbitrary polynomials, provided that appropriate care is taken in accordance with \eqref{eq2par1}.  
For example, consider the continuous function
\begin{equation}\label{eqG}
	G(r,s) =\left( \begin{array}{l}
		\displaystyle\sum_{i=0}^{2} a_{2-i,i} \, r^{2-i} s^{i} + 2 + a_{3,1} \, r^{3} s \\ [15pt]
		\displaystyle\sum_{i=0}^{3} b_{m_{2}-i,i} \, r^{3-i} s^{i} + \frac{7}{4} + b_{2,6} \, r^{2} s^{6}
%	\end{array}  \right
%	G(r,s) = \left(
%	\sum_{i=0}^{2} a_{2-i,i} \, r^{2-i} s^{i} + 2 + a_{3,1} \, r^{3} s \; , \;
%	\sum_{i=0}^{3} b_{m_{2}-i,i} \, r^{3-i} s^{i} + \frac{7}{4} + b_{2,6} \, r^{2} s^{6}
%	\right),
		\end{array}  \right),
\end{equation}
%which must satisfy the following coefficient conditions:
which satisfies the following conditions on its coefficients:
\begin{align*}
a_{2,0} &= -a_{1,1} - a_{0,2} - 2 - a_{3,1},   \qquad \qquad \; \,
a_{1,1} = -a_{3,1}, \\
b_{3,0} &= -b_{0,3} - b_{1,2} - b_{2,1} - \frac{7}{4} - b_{2,6}, \qquad
b_{2,1} = -b_{0,3}, \qquad
b_{2,6} = -\frac{7}{4}.
\end{align*}
%With this, the nonlinear double regularization of \eqref{sistemahopf}, 
%using the transition functions in \eqref{casoC1cod0} and the function $G$ given in \eqref{eqG}, together with the conditions
%
With this, the nonlinear double regularization of \eqref{sistemahopf},  
using the transition functions in \eqref{casoC1cod0} and the function $G$ given in \eqref{eqG}, together with the following conditions,
\begin{align*}
	a_{0,2} &= \frac{8}{5} \, \left(-4 + 7 \, a_{3,1}\right), \qquad
	b_{0,3} = \frac{1}{448} \, \left(-54425 - 512 \, b_{1,2}\right),
\end{align*}
admits an equilibrium at the origin for sufficiently small $\xi$, 
which also undergoes a Hopf bifurcation at $\mu = 0$, as discussed previously.

The function $G = (G_1, G_2)$ can also be defined in different configurations, 
for instance, by considering the coordinates as permutations of the following forms:
\begin{equation}
	\begin{aligned}
		G_i(r,s) &=
		\sum_{j=0}^{m_1} a_j \, |r|^{\alpha_j} \, |s|^{\beta_j}
		\quad \text{or} \quad
		G_i(r,s) =
		\sum_{j=0}^{m_2} b_j \, \operatorname{sgn}(r) \, |r|^{\alpha_j} \,
		\operatorname{sgn}(s) \, |s|^{\beta_j},
	\end{aligned}
\end{equation}
%where $\alpha_j$ and $\beta_j$ are negative real numbers.  
%In this case, the coefficients necessarily satisfy
where $\alpha_j$ and $\beta_j$ are negative real numbers.  
In this case, the coefficients must necessarily satisfy
\begin{align*}
	a_{0} &= - \sum_{i=1}^{m_1} a_{i}, \qquad
	b_{0} = - \sum_{i=1}^{m_2} b_{i}.
\end{align*}
%Since this case can be regarded as a generalization of the polynomial one, 
%we can, with appropriate care, also consider sums involving both 
%polynomials and functions of the latter type. 
%It is important to note that the function $G$ may not be differentiable at the origin. 
%Nevertheless, it is still possible to verify the existence of an equilibrium point. 
%For instance, by considering
Since this case can be regarded as a generalization of the polynomial case, 
we can, with appropriate care, also consider sums involving both 
polynomials and functions of the last type. 
It is important to note that the function $G$ may not be differentiable at the origin.
Nevertheless, it is still possible to verify the existence of an equilibrium point. 
For instance, by considering
\begin{equation}\label{eqG1}
G(r,s) = 
\left( \begin{array}{c}
	a_0 \,  |r|^{1/2} + a_1 \, |s| ^{7/2} + a_2 \, |r|^{10/3} + 2 + a_3 \, r^3 \, s \\[6pt]
b_0\,\mathrm{sgn}(r) \, |r|^{4} \,\mathrm{sgn}(s)\,|s|^{7/2}
	+ b_1\,\mathrm{sgn}(r)\,  |r|^{4} \, \mathrm{sgn}(s)\, |s|^{6}
	+ \dfrac{7}{4} + b_2 \,r^2 \, s^6
\end{array} \right),
\end{equation}
%so that the coefficients satisfy the following conditions:
such that the coefficients satisfy the following conditions:
\begin{align*}
	a_{0} = -a_{1} - a_{2} - 2 -a_3,   \quad
	a_{3} = 0,\quad
	b_{0} = -b_{1} - \frac{7}{4} - b_{2}, \quad
	b_{2} = -\frac{7}{4}.
\end{align*}
%In this way, the nonlinear double regularization of \eqref{sistemahopf}, 
%constructed using the transition functions in \eqref{casoC1cod0} and the function $G$ given in \eqref{eqG1}, with coefficients such that
In this way, the nonlinear double regularization of \eqref{sistemahopf},  
constructed using the transition functions in \eqref{casoC1cod0} and the function $G$ given in \eqref{eqG1}, with coefficients chosen such that
\begin{align*}
	a_1 =
	-\frac{128 \left( 2 \, (-1 + \sqrt{2}) + 2^{1/3} \, \big(-8 + 2^{1/6}\big) \, a_2 \right)}{128 \, \sqrt{2} - 343 \, \sqrt{7}}
	, \qquad
	b_1 = -\frac{116625}{784 \, \big(-343 + 32 \, \sqrt{7}\big)},
\end{align*}
%admits an equilibrium at the origin for sufficiently small $\xi$. 
%If $a_0 \neq 0$, then $G$ is not differentiable at the origin. 
%On the other hand, if $a_0 = 0$, that is,
%\[
%a_2 = \frac{-256 + 686\sqrt{7}}{1024\, 2^{1/3} - 343\sqrt{7}},
%\]
%the origin remains an equilibrium point for sufficiently small $\xi$, 
%and the system undergoes a Hopf bifurcation at $\mu = 0$, 
%as previously discussed, with $G$ being a mapping of class $C^3$ at the origin.
admits an equilibrium at the origin for sufficiently small $\xi$.  
If $a_0 \neq 0$, then $G$ is not differentiable at the origin.  
On the other hand, if $a_0 = 0$, that is,
\[
a_2 = \frac{-256 + 686\sqrt{7}}{1024 \, 2^{1/3} - 343\sqrt{7}},
\]
the origin remains an equilibrium point for sufficiently small $\xi$,  
and the system undergoes a Hopf bifurcation at $\mu = 0$,  
as previously discussed, with $G$ being a $C^3$ mapping at the origin.

%To conclude this section, we highlight another phenomenon that may arise during the smoothing process, in addition to the emergence of equilibrium curves discussed earlier. In the two-parameter regularization of piecewise-smooth systems, one may encounter situations where a regularization parameter is fixed, sufficiently small, and dependent on the bifurcation parameter. However, the dynamics of the regularized vector field may not be fully representative for understanding or comparing with the original piecewise-smooth field, in the sense that the regularization approaches the piecewise-smooth system only when both parameters simultaneously tend to zero. In the present case, only one parameter tends to zero. In this case, although the regularized vector field remains sufficiently close to the piecewise-smooth one, pointwise convergence with respect to the regularization parameters does not occur.

To conclude this section, we highlight another phenomenon that may arise during the smoothing process, in addition to the emergence of equilibrium curves discussed above. In the two-parameter regularization of piecewise-smooth systems, one may encounter situations in which a regularization parameter is fixed, sufficiently small, and dependent on the bifurcation parameter. However, the dynamics of the regularized vector field may not fully capture the behavior of the original piecewise-smooth system, in the sense that the regularization approaches the piecewise-smooth system only when both parameters simultaneously tend to zero. In the present case, only one parameter tends to zero. Consequently, although the regularized vector field remains sufficiently close to the piecewise-smooth one, pointwise convergence with respect to the regularization parameters does not occur.

%To illustrate the previous remark, in the regularization of the family~(\ref{eq2par2}), a bifurcation distinct from that described in Theorem~\ref{teo4} may occur for a fixed value $\eta = \eta_0$ is considered. However, this type of bifurcation is of limited interest for the purposes of our analysis, since the regularization approaches the piecewise-smooth system only when both $\varepsilon$ and $\eta$ simultaneously tend to zero, whereas in this case only $\varepsilon$ tends to zero. In a certain sense, this implies that such a bifurcation is ``far'' from the piecewise-smooth system. Hence, for $\xi=0$, the nonlinear double regularization associated with the family~(\ref{eq2par2}) can be expressed as

Illustrating the previous remark, consider the regularization of the family~(\ref{eq2par2}) for a fixed value $\eta = \eta_0$.
%, in which a bifurcation distinct from that described in Theorem~\ref{teo4} may occur. 
However, this type of bifurcation is of limited interest for the purposes of our analysis, since the regularization approaches the piecewise-smooth system only when both $\varepsilon$ and $\eta$ simultaneously tend to zero, whereas in this case only $\varepsilon$ tends to zero. In a certain sense, this implies that such a bifurcation is ``far'' from the piecewise-smooth system. Hence, for $\xi = 0$, the nonlinear double regularization associated with the family~(\ref{eq2par2}) can be expressed as
	\begin{equation}\label{eq2par3} 
	Z^R_{\varepsilon, \eta,\alpha} = \phi \left( \dfrac{x_1}{\varepsilon} \right)\psi \left( \dfrac{x_2}{\eta} \right)  \left( \frac{\widetilde{X_\alpha} - \widetilde{Y}}{2} \right) +  \left( \frac{\widetilde{X_\alpha} + \widetilde{Y}}{2} \right).
\end{equation}
%Assuming that the transition functions in~\eqref{eq2par3} are of the Sotomayor--Teixeira type and satisfy $\varphi(0) \neq 0$, it follows that there exists $p_0 \in (-\eta, \eta)$ such that $\psi(p_0) = 0$. Thus, we have that 
%$\bigl(0, -\alpha_0 / c_1\bigr) = (0, \eta \, p_0)$ is an equilibrium of system~\eqref{eq2par2} for the parameter value $\alpha = \alpha_0 = -\eta \, c_1 \, p_0$. We now perform the change of variables
Assuming that the transition functions in~\eqref{eq2par3} are of the Sotomayor--Teixeira type and satisfy $\varphi(0) \neq 0$, it follows that there exists $p_0 \in (-\eta, \eta)$ such that $\psi(p_0) = 0$. Consequently, 
$\bigl(0, -\alpha_0 / c_1\bigr) = (0, \eta \, p_0)$ is an equilibrium of system~\eqref{eq2par2} for the parameter value $\alpha = \alpha_0 = -\eta \, c_1 \, p_0$.  
We then perform the change of variables

\[
x_1 = r, \qquad x_2 = s + \eta p_0, \qquad \alpha = \mu - \eta \, c_1 \, p_0.
\]
%Under this transformation, system~\eqref{eq2par3} becomes
Under this transformation, system~\eqref{eq2par3} is transformed into
\begin{equation}\label{partelinearselano}
	Z^R_{\varepsilon, \eta,\mu} =
	\begin{pmatrix} 0 \\ a \mu / 2 \end{pmatrix} +
	\begin{pmatrix} -b \, c_2 / 2 & B_1 / \eta \\[2mm] 0 & (B_1 \, B_2 + c_1 \, B_3 \, \eta) / (2 \, \eta) \end{pmatrix} 
	\begin{pmatrix} r \\ s \end{pmatrix} +
	\begin{pmatrix} \cdots \\ B_4 \, s^2 + \cdots \end{pmatrix},
\end{equation}
where
\[
\begin{aligned}
	B_1 &= \phi(0) \,  \psi^\prime(p_0) \, a, &\quad B_2 &= \mu + 2 \, b \, a, \\
	B_3 &= a - 2 \, B_1 \, p_0, &\quad B_4 &= \frac{\phi(0) \, \psi^{\prime\prime}(p_0) \, (b - a c_1 \, p_0 \, \eta)}{\eta^2} - \frac{B_1 \, c_1}{2 \eta}.
\end{aligned}
\]
%Since we have fixed $\mu=0$, we choose $\eta_0$ as the solution, in $\eta$, of the following equation

Since we have fixed $\mu = 0$, we choose $\eta_0$ as the solution in $\eta$ of the following equation:
\begin{equation*} 
	(B_1 \, B_2 + c_1 \, B_3 \, \eta)_{|_{\mu=0}} = 0,
\end{equation*}
that is,
\begin{equation} \label{eqeta0}
	\eta_0 = \left(\frac{-B_1 \, B_2}{c_1 \, B_3}\right)_{\big|_{\mu=0}},
\end{equation}
%which vanishes the determinant of the linear part of \eqref{partelinearselano} for $\mu=0$, i.e., the entry in the second row and second column is zero for $\mu=0$. Note that $\eta_0$ is small and its denominator is nonzero, since $\phi(0) \in (-\varepsilon,\varepsilon)$ and $\mu$ is also small. To bring the linear part of $Z^R_{\varepsilon,\eta_0,\mu}$ at $(r,s)=(0,0)$ into its Jordan normal form, we apply the change of coordinates
which makes the determinant of the linear part of \eqref{partelinearselano} vanish for $\mu = 0$, i.e., the entry in the second row and second column is zero for $\mu = 0$. Note that $\eta_0$ is small and its denominator is nonzero, since $\phi(0) \in (-\varepsilon, \varepsilon)$ and $\mu$ is also small.  
To bring the linear part of $Z^R_{\varepsilon, \eta_0, \mu}$ at $(r, s) = (0, 0)$ into its Jordan normal form, we apply the change of coordinates
\[
r = x_1 + \frac{2 \, c_1 \,  (B_1 \, p_0 - 1)}{c_2 \, (2 \, a + b \, \mu)} \, x_2, \qquad s = x_2.
\]
%Under this transformation, the system \eqref{partelinearselano} becomes
Under this transformation, system~\eqref{partelinearselano} is transformed into
\begin{equation}\label{eq2par41} 
	Z^R_{\varepsilon,\eta,\mu} = g(x_1,x_2,\mu) =
	\begin{pmatrix}
		B_0(\mu) - \frac{b \, c_2}{2} x + \cdots \\[2mm]
		\frac{a \, \mu}{2} + B_5(\mu) \, x_2^2 + \cdots
	\end{pmatrix},
\end{equation}
where
\[
\begin{aligned}
B_0(\mu) & = \frac{a \, c_1 \, B_3 \, \mu}{b \, c_2 \, B_2}, \\
B_5(\mu) &= \frac{B_3}{2 \, B_2} - \frac{a \, \phi^\prime(0) \, B_3^2}{c_2 \, b \, \phi(0) \, B_2 \varepsilon} + \frac{\phi^\prime(0) \, \psi^\prime(p_0) \, p_0 \, B_3}{c_2 \, \varepsilon} - \frac{b \, \psi^{\prime\prime}(p_0) \, B_3^2}{\phi(0) \, (\psi^{\prime\prime}(p_0))^2 \, B_2^2} + \frac{\psi^{\prime\prime}(p_0) \, B_3}{\psi^\prime(p_0) \, B_2}.
\end{aligned}
\]

%In the notation of Theorem~\ref{Teosotobiftransc}, the eigenvector associated with the zero eigenvalue of the linear part of \eqref{eq2par41} is $v^T = w^T = (0,1)$. Therefore, we have
In the notation of Theorem~\ref{Teosotobiftransc}, the eigenvector associated with the zero eigenvalue of the linear part of \eqref{eq2par41} is $v^T = w^T = (0,1)$. Consequently, we have
\begin{equation}\label{TC1}
	\begin{aligned}
		&w^T g_\mu(0,0,0) = a/2 \neq 0,\\
		&w^T \big[D^2 g(0,0,0)(v,v)\big] = B_5(0) \neq 0,
	\end{aligned}
\end{equation}
%for $\varepsilon$ sufficiently small. That is, there exists $0 < \varepsilon_0$ such that \eqref{TC1} is satisfied for all $0 < \varepsilon < \varepsilon_0$ and $\eta = \eta_0$. By Theorem~\ref{Teosotobiftransc}, we conclude that the system \eqref{eq2par41} undergoes a saddle-node bifurcation at $(x_1,x_2,\mu) = (0,0,0)$, and consequently, the system \eqref{eq2par2} undergoes a saddle-node bifurcation at $(x_1,x_2,\alpha) = (0, -\alpha_0/c_1, \alpha_0)$ with $\alpha_0 = -\eta_0 c_1 p_0$, when $\eta = \eta_0$. In this case, the codimension and genericity are preserved.
for sufficiently small $\varepsilon$. That is, there exists $0 < \varepsilon_0$ such that \eqref{TC1} is satisfied for all $0 < \varepsilon < \varepsilon_0$ and $\eta = \eta_0$.  
By Theorem~\ref{Teosotobiftransc}, we conclude that system~\eqref{eq2par41} undergoes a saddle-node bifurcation at $(x_1, x_2, \mu) = (0, 0, 0)$. Consequently, system~\eqref{eq2par2} undergoes a saddle-node bifurcation at $(x_1, x_2, \alpha) = (0, -\alpha_0/c_1, \alpha_0)$ with $\alpha_0 = -\eta_0 c_1 p_0$, when $\eta = \eta_0$. In this case, both the codimension and genericity are preserved.

\section*{Acknowledgements}

%This work has been realized thanks to the Brazilian S˜ao Paulo Research Foundation (FAPESP) grant 2024/15612-6 and Conselho Nacional de Desenvolvimento
%Cient´ıfico e Tecnol´ogico (CNPq) grant 301878/2025-0; the Catalan AGAUR Agency
%grant 2021 SGR 00113; and the Spanish Ministerio de Ci´encia, Innovaci´on y Universidades, Agencia Estatal de Investigaci´on grants PID2022-136613NB-I00 and
%CEX2020-001084-M.

%This work was supported by the São Paulo Research Foundation (FAPESP) under grants 2016/00242-2, 2018/05098-2,2021/14695-7, 2022/03800-7 and the Spanish Ministry of Science, Innovation and Universities and the State Research Agency under Grant  PID2022-136613NB-I00 (or PK619449 PID2022-136613NB-I00).

%This work was supported by the São Paulo Research Foundation (FAPESP) under grants 2016/00242-2, 2018/05098-2, 2021/14695-7, and 2022/03800-7, and by the Spanish Ministry of Science, Innovation and Universities and the State Research Agency under grant PK619449 PID2022-136613NB-I00.
%(or PID2022-136613NB-I00).

This work was supported by the São Paulo Research Foundation (FAPESP) under Grants 2016/00242-2, 2018/05098-2, 2021/14695-7, 2022/03800-7, 2023/02959-5, and 2024/15612-6; by the Brazilian Federal Agency for Support and Evaluation of Graduate Education (CAPES) under Grant 88881.179491/2025-01; by the Brazilian National Council for Scientific and Technological Development (CNPq) under Grants 401974/2025-1 and 307706/2023-0; by the Spanish State Research Agency (Agencia Estatal de Investigación) under Grant PID2022-136613NB-I00; and by the French National Research Agency (Agence Nationale de la Recherche, ANR) under Grant ANR-23-CE40-0028.


\begin{thebibliography}{10}
	
	\bibitem{AlexSei1998}
	{\sc Alexander, J.~C., and Seidman, T.~I.}
	\newblock Sliding modes in intersecting switching surfaces, i: Blending.
	\newblock {\em Houston J. Math 24}, 3 (1998), 545--569.
	
	\bibitem{AndroLeontGor1973}
	{\sc Andronov, A.~A., Leontovich, E.~A., Gordon, I.~I., and Ma\u{\i}er, A.~G.}
	\newblock {\em Theory of bifurcations of dynamic systems on a plane}.
	\newblock Halsted Press [A division of John Wiley \& Sons], New York-Toronto,
	Ont.; Israel Program for Scientific Translations, Jerusalem-London, 1973.
	\newblock Translated from the Russian.
	
	\bibitem{AndVitKha1987}
	{\sc Andronov, A.~A., Vitt, A.~A., and Kha\u\i{}kin, S.~E.}
	\newblock {\em Theory of oscillators}.
	\newblock Dover Publications, Inc., New York, 1987.
	\newblock Translated from the Russian by F. Immirzi, Reprint of the 1966
	translation.
	
	\bibitem{Anosov1959}
	{\sc Anosov, D.~V.}
	\newblock On stability of equilibrium states of relay systems.
	\newblock {\em Avtomatika i Telemehanika 20\/} (1959), 135--149.
	
	\bibitem{AprBanVanBruBouFaIrRadRinVee2012}
	{\sc Apri, M., Banagaaya, N., Van Den~Berg, J., Brussee, R., Bourne, D.,
		Fatima, T., Irzal, F., Rademacher, J., Rink, B., Veerman, F., et~al.}
	\newblock Analysis of a model for ship maneuvering.
	\newblock In {\em 79th European Study Group with Industry, January 24-28, 2011,
		Amsterdam, The Netherlands\/} (2012), Vrije Universiteit Amsterdam,
	pp.~83--116.
	
	\bibitem{Barbashin1970}
	{\sc Barbashin, E.~A.}
	\newblock {\em Introduction to the theory of stability}.
	\newblock Translated from the Russian by Transcripta Service, London. Edited by
	T. Lukes. Wolters-Noordhoff Publishing, Groningen, 1970.
	
	\bibitem{BonLarTere2018}
	{\sc Bonet-Reves, C., Larrosa, J., Tere, M., et~al.}
	\newblock Regularization around a generic codimension one fold-fold
	singularity.
	\newblock {\em Journal of Differential Equations 265}, 5 (2018), 1761--1838.
	
	\bibitem{BonTere2016}
	{\sc Bonet-Rev\'{e}s, C., and M-Seara, T.}
	\newblock Regularization of sliding global bifurcations derived from the local
	fold singularity of {F}ilippov systems.
	\newblock {\em Discrete Contin. Dyn. Syst. 36}, 7 (2016), 3545--3601.
	
	\bibitem{BruPuSim2001}
	{\sc Broucke, M.~E., Pugh, C., and Simic, S.~N.}
	\newblock Structural stability of piecewise smooth systems.
	\newblock {\em Computational and applied mathematics 20}, 1-2 (2001), 51--89.
	
	\bibitem{BuzSilTei2006}
	{\sc Buzzi, C.~A., da~Silva, P.~R., and Teixeira, M.~A.}
	\newblock A singular approach to discontinuous vector fields on the plane.
	\newblock {\em Journal of Differential Equations 231}, 2 (2006), 633--655.
	
	\bibitem{BuzMedTor2020}
	{\sc Buzzi, C.~A., Medrado, J.~C., and Torregrosa, J.}
	\newblock Limit cycles in 4-star-symmetric planar piecewise linear systems.
	\newblock {\em J. Differential Equations 268}, 5 (2020), 2414--2434.
	
	\bibitem{BuzSan2019}
	{\sc Buzzi, C.~A., and Santos, R. A.~T.}
	\newblock Regularization of saddle-fold singularity for nonsmooth differential
	systems.
	\newblock {\em J. Math. Anal. Appl. 474}, 2 (2019), 1036--1048.
	
	\bibitem{SilMezNov2022}
	{\sc da~Silva, P.~R., Meza-Sarmiento, I.~S., and Novaes, D.~D.}
	\newblock Nonlinear sliding of discontinuous vector fields and singular
	perturbation.
	\newblock {\em Differential Equations and Dynamical Systems 30}, 3 (2022),
	675--693.
	
	\bibitem{SilNun2019}
	{\sc da~Silva, P.~R., and Nunes, W.~P.}
	\newblock Slow--fast systems and sliding on codimension 2 switching manifolds.
	\newblock {\em Dynamical Systems 34}, 4 (2019), 613--639.
	
	\bibitem{BerBudCha2001}
	{\sc di~Bernardo, M., Budd, C., and Champneys, A.}
	\newblock Corner collision implies border-collision bifurcation.
	\newblock {\em Physica D: Nonlinear Phenomena 154}, 3-4 (2001), 171--194.
	
	\bibitem{BerBudChaKow2008}
	{\sc di~Bernardo, M., Budd, C.~J., Champneys, A.~R., and Kowalczyk, P.}
	\newblock {\em Piecewise-smooth dynamical systems}, vol.~163 of {\em Applied
		Mathematical Sciences}.
	\newblock Springer-Verlag London, Ltd., London, 2008.
	
	\bibitem{DiPagPon2008}
	{\sc Di~Bernardo, M., Pagano, D.~J., and Ponce, E.}
	\newblock Nonhyperbolic boundary equilibrium bifurcations in planar filippov
	systems: a case study approach.
	\newblock {\em International Journal of Bifurcation and chaos 18}, 05 (2008),
	1377--1392.
	
	\bibitem{DieEli2015}
	{\sc Dieci, L., and Elia, C.}
	\newblock Piecewise smooth systems near a co-dimension 2 discontinuity
	manifold: can one say what should happen?
	\newblock {\em arXiv preprint arXiv:1508.02639\/} (2015).
	
	\bibitem{DieLop2011}
	{\sc Dieci, L., and Lopez, L.}
	\newblock Sliding motion on discontinuity surfaces of high co-dimension. a
	construction for selecting a filippov vector field.
	\newblock {\em Numerische Mathematik 117\/} (2011), 779--811.
	
	\bibitem{SanTon2023}
	{\sc dos Santos, M.~J., and Tonon, D.~J.}
	\newblock Structural stability for 2-dimensional piecewise smooth vector fields
	where the switching manifold is a double discontinuity.
	\newblock {\em Journal of Dynamical and Control Systems 29}, 4 (2023),
	1775--1808.
	
	\bibitem{FanXueChen2018}
	{\sc Fan, J., Xue, S., and Chen, G.}
	\newblock On discontinuous dynamics of a periodically forced double-belt
	friction oscillator.
	\newblock {\em Chaos, Solitons \& Fractals 109\/} (2018), 280--302.
	
	\bibitem{Fil1988}
	{\sc Filippov, A.~F.}
	\newblock {\em Differential equations with discontinuous righthand sides},
	vol.~18 of {\em Mathematics and its Applications (Soviet Series)}.
	\newblock Kluwer Academic Publishers Group, Dordrecht, 1988.
	
	\bibitem{GuaSeaTeixeira2011}
	{\sc Guardia, M., Seara, T.~M., and Teixeira, M.~A.}
	\newblock Generic bifurcations of low codimension of planar {F}ilippov systems.
	\newblock {\em J. Differential Equations 250}, 4 (2011), 1967--2023.
	
	\bibitem{Hen1997}
	{\sc Henson, M.~A., and Seborg, D.~E.}
	\newblock {\em Nonlinear process control}.
	\newblock Prentice Hall PTR Upper Saddle River, New Jersey, 1997.
	
	\bibitem{Jeffrey2014}
	{\sc Jeffrey, M.~R.}
	\newblock Hidden dynamics in models of discontinuity and switching.
	\newblock {\em Physica D: Nonlinear Phenomena 273\/} (2014), 34--45.
	
	\bibitem{Kozlova1984}
	{\sc Kozlova, V.~S.}
	\newblock Structural stability of a discontinuous system.
	\newblock {\em Vestnik Moskovskogo Universiteta. Seriya 1. Matematika.
		Mekhanika}, 5 (1984), 16--20.
	
	\bibitem{KriHog2015}
	{\sc Kristiansen, K.~U., and Hogan, S.~J.}
	\newblock Regularizations of two-fold bifurcations in planar piecewise smooth
	systems using blowup.
	\newblock {\em SIAM J. Appl. Dyn. Syst. 14}, 4 (2015), 1731--1786.
	
	\bibitem{Kuznetsov1998}
	{\sc Kuznetsov, Y.~A., Kuznetsov, I.~A., and Kuznetsov, Y.}
	\newblock {\em Elements of applied bifurcation theory}, vol.~112.
	\newblock Springer, 1998.
	
	\bibitem{KusGraRin2003}
	{\sc Kuznetsov, Y.~A., Rinaldi, S., and Gragnani, A.}
	\newblock One-parameter bifurcations in planar {F}ilippov systems.
	\newblock {\em Internat. J. Bifur. Chaos Appl. Sci. Engrg. 13}, 8 (2003),
	2157--2188.
	
	\bibitem{LarSeaTei2021}
	{\sc Larrosa, J., M-Seara, T., and Teixeira, M.~A.}
	\newblock Local generic behavior of a planar filippov system with non-smooth
	switching curve.
	\newblock {\em S{\~a}o Paulo Journal of Mathematical Sciences 15\/} (2021),
	882--915.
	
	\bibitem{LeiNij2004}
	{\sc Leine, R.~I., Nijmeijer, H., Leine, R.~I., and Nijmeijer, H.}
	\newblock Bifurcations of equilibria in non-smooth continuous systems.
	\newblock {\em Dynamics and bifurcations of non-smooth mechanical systems\/}
	(2004), 125--176.
	
	\bibitem{Liberzon2003}
	{\sc Liberzon, D.}
	\newblock {\em Switching in systems and control}, vol.~190.
	\newblock Springer, 2003.
	
	\bibitem{LliSilTei2007}
	{\sc Llibre, J., da~Silva, P.~R., and Teixeira, M.~A.}
	\newblock Regularization of discontinuous vector fields on {$ \mathbb{R}^3$}
	via singular perturbation.
	\newblock {\em J. Dynam. Differential Equations 19}, 2 (2007), 309--331.
	
	\bibitem{LliSilTei2008}
	{\sc Llibre, J., da~Silva, P.~R., and Teixeira, M.~A.}
	\newblock Sliding vector fields via slow-fast systems.
	\newblock {\em Bull. Belg. Math. Soc. Simon Stevin 15}, 5 (2008), 851--869.
	
	\bibitem{LliSilTei2009}
	{\sc Llibre, J., Da~Silva, P.~R., and Teixeira, M.~A.}
	\newblock Study of singularities in nonsmooth dynamical systems via singular
	perturbation.
	\newblock {\em SIAM Journal on Applied Dynamical Systems 8}, 1 (2009),
	508--526.
	
	\bibitem{LliSilTei2015}
	{\sc Llibre, J., Da~Silva, P.~R., and Teixeira, M.~A.}
	\newblock Sliding vector fields for non-smooth dynamical systems having
	intersecting switching manifolds.
	\newblock {\em Nonlinearity 28}, 2 (2015), 493.
	
	\bibitem{LliMerNov2015}
	{\sc Llibre, J., Mereu, A.~C., and Novaes, D.~D.}
	\newblock Averaging theory for discontinuous piecewise differential systems.
	\newblock {\em Journal of differential equations 258}, 11 (2015), 4007--4032.
	
	\bibitem{LliSot1996}
	{\sc Llibre, J., and Sotomayor, J.}
	\newblock Phase portraits of planar control systems.
	\newblock {\em Nonlinear Analysis: Theory, Methods \& Applications 27}, 10
	(1996), 1177--1197.
	
	\bibitem{LliTei1997}
	{\sc Llibre, J., and Teixeira, M.~A.}
	\newblock Regularization of discontinuous vector fields in dimension three.
	\newblock {\em Discrete Contin. Dynam. Systems 3}, 2 (1997), 235--241.
	
	\bibitem{LliTei2015}
	{\sc Llibre, J., and Teixeira, M.~A.}
	\newblock Limit cycles for {$m$}-piecewise discontinuous polynomial
	{L}i\'{e}nard differential equations.
	\newblock {\em Z. Angew. Math. Phys. 66}, 1 (2015), 51--66.
	
	\bibitem{LucGaj2006}
	{\sc Lucero, J.~C., and Gajo, C.~A.}
	\newblock Oscillation region of a piecewise-smooth model of the vocal folds.
	
	\bibitem{Maciel2009}
	{\sc Maciel, A.~L.}
	\newblock Bifurcações de campos vetoriais descontínuos, 2009.
	\newblock PhD Thesis, IME-USP,
	https://teses.usp.br/teses/disponiveis/45/45132/tde-16082009-201102/en.php.
	
	\bibitem{MakLam2012}
	{\sc Makarenkov, O., and Lamb, J.~S.}
	\newblock Dynamics and bifurcations of nonsmooth systems: A survey.
	\newblock {\em Physica D: Nonlinear Phenomena 241}, 22 (2012), 1826--1844.
	
	\bibitem{MarMcC2012}
	{\sc Marsden, J.~E., and McCracken, M.}
	\newblock {\em The Hopf bifurcation and its applications}, vol.~19.
	\newblock Springer Science \& Business Media, 2012.
	
	\bibitem{Murray2002}
	{\sc Murray, J.~D.}
	\newblock Mathematical biology: I. an introduction. interdisciplinary applied
	mathematics.
	\newblock {\em Mathematical Biology, Springer 17\/} (2002).
	
	\bibitem{Nov2014}
	{\sc Novaes, D.~D.}
	\newblock On nonsmooth perturbations of nondegenerate planar centers.
	\newblock {\em Publ. Mat. 58}, suppl. (2014), 395--420.
	
	\bibitem{NovJeff2015}
	{\sc Novaes, D.~D., and Jeffrey, M.~R.}
	\newblock Regularization of hidden dynamics in piecewise smooth flows.
	\newblock {\em Journal of differential equations 259}, 9 (2015), 4615--4633.
	
	\bibitem{NovRon2021}
	{\sc Novaes, D.~D., and Rond\'{o}n, G.}
	\newblock Smoothing of nonsmooth differential systems near regular-tangential
	singularities and boundary limit cycles.
	\newblock {\em Nonlinearity 34}, 6 (2021), 4202--4263.
	
	\bibitem{PanPau2017}
	{\sc Panazzolo, D., and da~Silva, P.~R.}
	\newblock Regularization of discontinuous foliations: blowing up and sliding
	conditions via {F}enichel theory.
	\newblock {\em J. Differential Equations 263}, 12 (2017), 8362--8390.
	
	\bibitem{PeiPei1959}
	{\sc Peixoto, M.~C., and Peixoto, M.~M.}
	\newblock Structural stability in the plane with enlarged boundary conditions.
	\newblock {\em An. Acad. Brasil. Ci 31}, 2 (1959), 135--160.
	
	\bibitem{Pei1962}
	{\sc Peixoto, M.~M.}
	\newblock Structural stability on two-dimensional manifolds.
	\newblock {\em Topology 1\/} (1962), 101--120.
	
	\bibitem{PerRonSil2023}
	{\sc Perez, O.~H., Rond{\'o}n, G., and da~Silva, P.~R.}
	\newblock Slow-fast normal forms arising from piecewise smooth vector fields.
	\newblock {\em Journal of Dynamical and Control Systems 29}, 4 (2023),
	1709--1726.
	
	\bibitem{Perko2013}
	{\sc Perko, L.}
	\newblock {\em Differential equations and dynamical systems}, vol.~7.
	\newblock Springer Science \& Business Media, 2013.
	
	\bibitem{PerBar2002}
	{\sc Perruquetti, W., and Barbot, J.~P.}
	\newblock {\em Sliding mode control in engineering}, vol.~11.
	\newblock Marcel Dekker New York, 2002.
	
	\bibitem{Russ1998}
	{\sc Roussarie, R., and Roussarie, R.}
	\newblock {\em Bifurcation of planar vector fields and Hilbert's sixteenth
		problem}, vol.~164.
	\newblock Springer, 1998.
	
	\bibitem{SimJohn2010}
	{\sc Simpson, D. J.~W.}
	\newblock {\em Bifurcations in piecewise-smooth continuous systems}, vol.~70.
	\newblock World Scientific, 2010.
	
	\bibitem{Smirnov2002}
	{\sc Smirnov, G.~V.}
	\newblock {\em Introduction to the theory of differential inclusions}, vol.~41
	of {\em Graduate Studies in Mathematics}.
	\newblock American Mathematical Society, Providence, RI, 2002.
	
	\bibitem{Sotomayor1974}
	{\sc Sotomayor, J.}
	\newblock Generic one-parameter families of vector fields on two-dimensional
	manifolds.
	\newblock {\em Publications Math{\'e}matiques de l'IH{\'E}S 43\/} (1974),
	5--46.
	
	\bibitem{Soto1979}
	{\sc Sotomayor, J.}
	\newblock {\em Li\c{c}\~oes de equa\c{c}\~oes diferenciais ordin\'arias},
	vol.~11.
	\newblock Instituto de Matem{\'a}tica Pura e Aplicada, CNPq, 1979.
	
	\bibitem{SotoMac2002}
	{\sc Sotomayor, J., and Machado, A. L.~F.}
	\newblock Structurally stable discontinuous vector fields in the plane.
	\newblock {\em Qual. Theory Dyn. Syst. 3}, 1 (2002), 227--250.
	
	\bibitem{SotTei1998}
	{\sc Sotomayor, J., and Teixeira, M.~A.}
	\newblock Regularization of discontinuous vector fields.
	\newblock In {\em International {C}onference on {D}ifferential {E}quations
		({L}isboa, 1995)}. World Sci. Publ., River Edge, NJ, 1998, pp.~207--223.
	
	\bibitem{Stewart2000}
	{\sc Stewart, D.~E.}
	\newblock Rigid-body dynamics with friction and impact.
	\newblock {\em SIAM Rev. 42}, 1 (2000), 3--39.
	
	\bibitem{TanOsoBer2009}
	{\sc Tanelli, M., Osorio, G., di~Bernardo, M., Savaresi, S.~M., and Astolfi,
		A.}
	\newblock Existence, stability and robustness analysis of limit cycles in
	hybrid anti-lock braking systems.
	\newblock {\em Internat. J. Control 82}, 4 (2009), 659--678.
	
	\bibitem{Teixeira1977}
	{\sc Teixeira, M.~A.}
	\newblock Generic bifurcation in manifolds with boundary.
	\newblock {\em Journal of Differential Equations 25}, 1 (1977), 65--89.
	
	\bibitem{Teixeira1990}
	{\sc Teixeira, M.~A.}
	\newblock Stability conditions for discontinuous vector fields.
	\newblock {\em Journal of Differential Equations 88}, 1 (1990), 15--29.
	
	\bibitem{TeiSilva2012}
	{\sc Teixeira, M.~A., and da~Silva, P.~R.}
	\newblock Regularization and singular perturbation techniques for non-smooth
	systems.
	\newblock {\em Phys. D 241}, 22 (2012), 1948--1955.
	
	\bibitem{Utkin2013}
	{\sc Utkin, V.~I.}
	\newblock {\em Sliding modes in control and optimization}.
	\newblock Springer Science \& Business Media, 2013.
	
	\bibitem{Van1920}
	{\sc van~der Pol, B.}
	\newblock A theory of the amplitude of free and forced triode vibrations.
	\newblock {\em Radio Rev. 1\/} (1920), 701--710.
	
	\bibitem{Van1926}
	{\sc van~der Pol, B.}
	\newblock On relaxation-oscillations.
	\newblock {\em Lond. Edinb. Dublin Philos. Mag. J. Sci. 2}, 7 (1926), 978--992.
	
	\bibitem{ZhuMose2003}
	{\sc Zhusubaliyev, Z.~T., and Mosekilde, E.}
	\newblock {\em Bifurcations and chaos in piecewise-smooth dynamical systems},
	vol.~44.
	\newblock World Scientific, 2003.
	
\end{thebibliography}
\end{document}